\documentclass[a4paper, 12pt]{amsart} 

\usepackage{amsmath,amssymb,enumitem,verbatim,stmaryrd,xcolor,microtype,graphicx,aliascnt,fancyvrb}

\usepackage[T1]{fontenc}
\usepackage[utf8]{inputenc}
\usepackage[english]{babel} % allow for hyphenation for words with dash: e.g. ``non-archimedean'' does't break, but ``non\hyph archimedean'' does
\usepackage[top=3.5cm,bottom=3.5cm,left=3.2cm,right=3.2cm]{geometry}
\usepackage[bookmarksdepth=2,linktoc=page,colorlinks,linkcolor={red!80!black},citecolor={red!80!black},urlcolor={blue!80!black},pdftitle={Flag matroids with coefficients},pdfauthor={Manoel Jarra and Oliver Lorscheid}]{hyperref}

\hypersetup{
           breaklinks=true,   % splits links across lines
           colorlinks=true,   % displays links as colored text instead of blocks
        }

\usepackage{tikz}\usetikzlibrary{matrix,arrows,decorations.markings}
\usepackage{tikz-cd}
\newcommand*{\DashedArrow}[1][]{\mathbin{\tikz [baseline=-0.25ex,-latex, dashed,#1] \draw [#1] (0pt,0.5ex) -- (1.3em,0.5ex);}}

\allowdisplaybreaks % allows equation environments to break between the pages

\usepackage{mathptmx}                                        % times font
\usepackage{etoolbox}\makeatletter\patchcmd{\@startsection}{\@afterindenttrue}{\@afterindentfalse}{}{}\makeatother    %omit indentation of the first paragraph of a section
\patchcmd{\section}{\scshape}{\bfseries}{}{}\makeatletter\renewcommand{\@secnumfont}{\bfseries}\makeatother           %boldface section and subsection titles (no caption), including numbers
\usepackage[backgroundcolor=orange!30!white,linecolor=orange!80!white,textsize=footnotesize]{todonotes} \makeatletter \providecommand \@dotsep{5} \def\listtodoname{List of Todos} \def\listoftodos{\@starttoc{tdo}\listtodoname} \makeatother %\todo{} for margin notes, supress in pdf with option [disable]

\addto\extrasenglish{}  % Change \autoref output from ``subsection 1.1'' to ``section 1.1''
\theoremstyle{plain}
\newtheorem{thm}{Theorem}[section] % provides command \autoref{}, which produces citations like ``Theorem 1.1''.
\newaliascnt{lemma}{thm}\newtheorem{lemma}[lemma]{Lemma}\aliascntresetthe{lemma}
\newaliascnt{cor}{thm}\newtheorem{cor}[cor]{Corollary}\aliascntresetthe{cor}
\newaliascnt{prop}{thm}\newtheorem{prop}[prop]{Proposition}\aliascntresetthe{prop}
\newtheorem{thmA}{Theorem} %alphabetic theorem counter: Theorem A, Theorem B, ...
\newaliascnt{propA}{thmA}\newtheorem{propA}[propA]{Proposition}\aliascntresetthe{propA}

\newtheorem*{thm*}{Theorem}
\newtheorem*{lem*}{Lemma}
\newtheorem*{cor*}{Corollary}
\newtheorem*{problem*}{Problem}

\theoremstyle{definition}
\newaliascnt{df}{thm}\newtheorem{df}[df]{Definition}\aliascntresetthe{df}
\newaliascnt{rem}{thm}\newtheorem{rem}[rem]{Remark}\aliascntresetthe{rem}
\newaliascnt{ex}{thm}\newtheorem{ex}[ex]{Example}\aliascntresetthe{ex}

\newtheorem*{df*}{Definition}
\newtheorem*{ex*}{Example}
\newtheorem*{rem*}{Remark}
\usepackage{etoolbox}
\makeatletter
\patchcmd{\@startsection}{\@afterindenttrue}{\@afterindentfalse}{}{}             %omit indentation of the first paragraph of a section
\patchcmd{\part}{\bfseries}{\bfseries\LARGE}{}{}
\patchcmd{\section}{\scshape}{\bfseries}{}{}\renewcommand{\@secnumfont}{\bfseries} %boldface no smallcaps section and subsection titles with numbers
\patchcmd{\@settitle}{\uppercasenonmath\@title}{\large}{}{}
\patchcmd{\@setauthors}{\MakeUppercase}{}{}{}
\addto{\captionsenglish}{} %boldface no smallcaps Abstract
\addto{\captionsenglish}{} %bolface no smallcaps Figure
\addto{\captionsenglish}{} %bolface no smallcaps Table
\makeatother

\usepackage{fancyhdr}

\DeclareRobustCommand{\gobblefour}[5]{}    % Command \SkipTocEntry for surpressing a section title in TOC

\DeclareFontFamily{OT1}{pzc}{}                                % Script font for small caligraphic letter, like in \cMat
\DeclareFontShape{OT1}{pzc}{m}{it}{<-> s * [1.10] pzcmi7t}{}
\DeclareMathAlphabet{\mathpzc}{OT1}{pzc}{m}{it}
\DeclareSymbolFont{sfoperators}{OT1}{bch}{m}{n} \DeclareSymbolFontAlphabet{\mathsf}{sfoperators} \makeatletter\def\operator@font{\mathgroup\symsfoperators}\makeatother % different font for math operators
\DeclareSymbolFont{cmletters}{OML}{cmm}{m}{it}              
\DeclareSymbolFont{cmsymbols}{OMS}{cmsy}{m}{n}
\DeclareSymbolFont{cmlargesymbols}{OMX}{cmex}{m}{n}
\DeclareMathSymbol{\myjmath}{\mathord}{cmletters}{"7C}     \let\jmath\myjmath %Defining the missing commands: \jmath, \amalg and \coprod
\DeclareMathSymbol{\myamalg}{\mathbin}{cmsymbols}{"71}     \let\amalg\myamalg
\DeclareMathSymbol{\mycoprod}{\mathop}{cmlargesymbols}{"60}\let\coprod\mycoprod
\DeclareMathSymbol{\myalpha}{\mathord}{cmletters}{"0B}     \let\alpha\myalpha %Greek letters from Computer Modern
\DeclareMathSymbol{\mybeta}{\mathord}{cmletters}{"0C}      \let\beta\mybeta
\DeclareMathSymbol{\mygamma}{\mathord}{cmletters}{"0D}     \let\gamma\mygamma
\DeclareMathSymbol{\mydelta}{\mathord}{cmletters}{"0E}     \let\delta\mydelta
\DeclareMathSymbol{\myepsilon}{\mathord}{cmletters}{"0F}   \let\epsilon\myepsilon
\DeclareMathSymbol{\myzeta}{\mathord}{cmletters}{"10}      \let\zeta\myzeta
\DeclareMathSymbol{\myeta}{\mathord}{cmletters}{"11}       \let\eta\myeta
\DeclareMathSymbol{\mytheta}{\mathord}{cmletters}{"12}     \let\theta\mytheta
\DeclareMathSymbol{\myiota}{\mathord}{cmletters}{"13}      \let\iota\myiota
\DeclareMathSymbol{\mykappa}{\mathord}{cmletters}{"14}     \let\kappa\mykappa
\DeclareMathSymbol{\mylambda}{\mathord}{cmletters}{"15}    \let\lambda\mylambda
\DeclareMathSymbol{\mymu}{\mathord}{cmletters}{"16}        \let\mu\mymu
\DeclareMathSymbol{\mynu}{\mathord}{cmletters}{"17}        \let\nu\mynu
\DeclareMathSymbol{\myxi}{\mathord}{cmletters}{"18}        \let\xi\myxi
\DeclareMathSymbol{\mypi}{\mathord}{cmletters}{"19}        \let\pi\mypi
\DeclareMathSymbol{\myrho}{\mathord}{cmletters}{"1A}       \let\rho\myrho
\DeclareMathSymbol{\mysigma}{\mathord}{cmletters}{"1B}     \let\sigma\mysigma
\DeclareMathSymbol{\mytau}{\mathord}{cmletters}{"1C}       \let\tau\mytau
\DeclareMathSymbol{\myupsilon}{\mathord}{cmletters}{"1D}   \let\upsilon\myupsilon
\DeclareMathSymbol{\myphi}{\mathord}{cmletters}{"1E}       \let\phi\myphi
\DeclareMathSymbol{\mychi}{\mathord}{cmletters}{"1F}       \let\chi\mychi
\DeclareMathSymbol{\mypsi}{\mathord}{cmletters}{"20}       \let\psi\mypsi
\DeclareMathSymbol{\myomega}{\mathord}{cmletters}{"21}     \let\omega\myomega
\DeclareMathSymbol{\myvarepsilon}{\mathord}{cmletters}{"22}\let\varepsilon\myvarepsilon
\DeclareMathSymbol{\myvartheta}{\mathord}{cmletters}{"23}  \let\vartheta\myvartheta
\DeclareMathSymbol{\myvarpi}{\mathord}{cmletters}{"24}     \let\varpi\myvarpi
\DeclareMathSymbol{\myvarrho}{\mathord}{cmletters}{"25}    \let\varrho\myvarrho
\DeclareMathSymbol{\myvarsigma}{\mathord}{cmletters}{"26}  \let\varsigma\myvarsigma
\DeclareMathSymbol{\myvarphi}{\mathord}{cmletters}{"27}    \let\varphi\myvarphi

\DeclareMathOperator{\Spec}{Spec}

\DeclareMathOperator{\Hom}{Hom}
\DeclareMathOperator{\Proj}{Proj}
\DeclareMathOperator{\Mat}{Mat}
\DeclareMathOperator{\Fl}{Fl}

\DeclareMathOperator{\Gr}{Gr}
\DeclareMathOperator{\GL}{GL}

\DeclareMathOperator{\cFl}{\mathpzc{Fl}}

\DeclareMathOperator{\OBlpr}{{OBlpr}}
\DeclareMathOperator{\Tracts}{{Tracts}}

\DeclareMathOperator*{\hypersum}{\,\raisebox{-2.2pt}{\larger[2]{$\boxplus$}}\,}
\DeclareMathOperator{\sign}{{sign}}

\DeclareMathOperator{\pl}{{pl}}
\DeclareMathOperator{\rk}{{rk}}

\newcommand\A{{\mathbb A}}

\newcommand\C{{\mathbb C}}
\newcommand\F{{\mathbb F}}

\newcommand\K{{\mathbb K}}
\newcommand\N{{\mathbb N}}
\renewcommand\P{{\mathbb P}}

\newcommand\R{{\mathbb R}}
\renewcommand\S{{\mathbb S}}
\newcommand\T{{\mathbb T}}
\newcommand\V{{\mathbb V}}
\newcommand\Z{{\mathbb Z}}

\newcommand\bI{{\mathbf I}}
\newcommand\bJ{{\mathbf J}}
\newcommand\bM{{\mathbf M}}
\newcommand\bN{{\mathbf N}}

\newcommand\bi{{\mathbf{i}}}

\newcommand\br{{\mathbf{r}}}

\newcommand\cB{{\mathcal B}}
\newcommand\cC{{\mathcal C}}
\newcommand\cD{{\mathcal D}}

\newcommand\cL{{\mathcal L}}

\newcommand\cO{{\mathcal O}}

\newcommand\cV{{\mathcal V}}
\newcommand\cW{{\mathcal W}}
\newcommand\cX{{\mathcal X}}

\newcommand\fr{{\mathfrak r}}

\newcommand\Fun{{\F_1}}

\newcommand\Funpm{{\F_1^\pm}}

\renewcommand\int{\textup{int}}

\newcommand\id{\textup{id}}

\newcommand\oblpr{\textup{oblpr}}
\newcommand\tract{\textup{tract}}

\newcommand\univ{\textup{univ}}

\renewcommand\geq{\geqslant}
\renewcommand\leq{\leqslant}
\newcommand{\gen}[1]{\langle #1 \rangle}

\newcommand{\bpquot}[2]{#1\!\sslash\!#2}

\newcommand{\hyperplus}{{\,\raisebox{-1.1pt}{\larger[-0]{$\boxplus$}}\,}}
\newcommand{\minor}[2]{\backslash #1 / #2}

\newcommand{\supp}{\textup{supp}}

\title{Flag matroids with coefficients}

\author{Manoel Jarra}
\address{\rm Manoel Jarra, University of Groningen, the Netherlands, and IMPA, Rio de Janeiro, Brazil}
\email{{manoel.jarra@impa.br}}

\author{Oliver Lorscheid}
\address{\rm Oliver Lorscheid, University of Groningen, the Netherlands, and IMPA, Rio de Janeiro, Brazil}
\email{{oliver@impa.br}}

\begin{document}

\begin{abstract}
 This paper is a direct generalization of Baker-Bowler theory to flag matroids, including its moduli interpretation as developed by Baker and the second author for matroids. More explicitly, we extend the notion of flag matroids to flag matroids over any tract, provide cryptomorphic descriptions in terms of basis axioms (Grassmann-Pl\"ucker functions), circuit/vector axioms and dual pairs, including additional characterizations in the case of perfect tracts. We establish duality of flag matroids and construct minors. Based on the theory of ordered blue schemes, we introduce flag matroid bundles and construct their moduli space, which leads to algebro-geometric descriptions of duality and minors. Taking rational points recovers flag varieties in several geometric contexts: over (topological) fields, in tropical geometry, and as a generalization of the MacPhersonian.
\end{abstract}

%  This paper is a direct generalization of Baker-Bowler theory to flag matroids, including its moduli interpretation as developed by Baker and the second author for matroids. More explicitly, we extend the notion of flag matroids to flag matroids over any tract, provide cryptomorphic descriptions in terms of basis axioms (Grassmann-Pluecker functions), circuit/vector axioms and dual pairs, including additional characterizations in the case of perfect tracts. We establish duality of flag matroids and construct minors. Based on the theory of ordered blue schemes, we introduce flag matroid bundles and construct their moduli space, which leads to algebro-geometric descriptions of duality and minors. Taking rational points recovers flag varieties in several geometric contexts: over (topological) fields, in tropical geometry, and as a generalization of the MacPhersonian.

\maketitle

\begin{small} \tableofcontents \end{small}

\VerbatimFootnotes   % Allows to use \verb|...| in footnotes, based on the package fancyvrb
\thispagestyle{empty} % supress page number on first page

%%%%%%%%%%%%%%%%%%%%%%%%%%%%%%%%%%%%%%%%%%%%%%%%%%%%%%%%%%%%%%%%%%%%%%%%%%%%%%%%%%%%%%%%%%%%%%%%%%%%%%%%%%%%%%%%%%%%%%%%%%%%%%%%%%%%%%%%%%%%%%%%%%%%%%%%%%%%%%%%%%%%%%%%%%%
%%%%%%%%%%%%%%%%%%%%%%%%%%%%%%%%%%%%%%%%%%%%%%%%%%%%%%%%%%%%%%%%%%%%%%%%%%%%%%%%%%%%%%%%%%%%%%%%%%%%%%%%%%%%%%%%%%%%%%%%%%%%%%%%%%%%%%%%%%%%%%%%%%%%%%%%%%%%%%%%%%%%%%%%%%%

\section*{Introduction}
\label{introduction}
Flag matroids appeared in different disguises---as sequences of strong maps in \cite{Higgs68,Kennedy75,Cheung-Crapo76,Kung77}, as particular cases of Coxeter matroids in \cite{Gelfand-Serganova87}, and implicitly as MacPhersonians in \cite{Mnev-Ziegler93,Babson94}---before their proper name was introduced and before they were systematically studied in the papers \cite{Borovik-Gelfand-Vince-White97,Borovik-Gelfand-White00,Borovik-Gelfand-White01,Borovik-Gelfand-Stone02} by Borovik, Gelfand, Stone, Vince and White (in varying constellations), culminating in an extended chapter in \cite{Borovik-Gelfand-White03}, which summarizes several cryptomorphic descriptions. Duality for flag matroids is developed in \cite{Garcia18} and minors are discussed in \cite[Thm. 4.1.5]{Brandt-Eur-Zhang21}. Other works on flag matroids include \cite{deMier07,Cameron-Dinu-Michalek-Seynnaeve18,Benedetti-Chavez-Tamayo19,Dinu-Eur-Seynnaeve21,Fujishige-Hirai22,Benedetti-Knauer22}.

A generalization of matroids of a different flavour are matroids with coefficients. The first type of such enriched matroids to appear in the literature are oriented matroids, as introduced by Bland and Las Vergnas in \cite{Bland-LasVergnas78}, which have strong ties to real geometry. Dress and Wenzel provide in \cite{Dress86,Dress-Wenzel91,Dress-Wenzel92} a vast generalization with their theory of matroids over fuzzy rings. A particular instance are valuated matroids, as introduced in \cite{Dress-Wenzel92b}, which are omnipresent in tropical geometry nowadays. Later Baker and Bowler streamlined and generalized Dress and Wenzel's theory to matroids over tracts in \cite{Baker-Bowler19}. Baker-Bowler theory encompasses cryptomorphic description of matroids over tracts in terms of Grassmann-Pl\"ucker functions, dual pairs and circuits, as well as vectors (see \cite{Anderson19}), and treats duality and minors.

To our best knowledge, the only types of flag matroids with coefficients---other than usual flag matroids (trivial coefficients) and flags of linear subspaces (coefficients in a field)---are valuated flag matroids (coefficients in the tropical hyperfield), which appear as flags of tropical linear spaces in \cite{Brandt-Eur-Zhang21,Boretsky21,Joswig-Loho-Luber-Olarte21,Boretsky-Eur-Williams22}, and, implicitly, oriented flag matroids (coefficients in the sign hyperfield) as points of MacPhersonians; see \cite{Mnev-Ziegler93,Anderson98,Bjorner-LasVergnas-Sturmfels-White-Ziegler99}.

%works that considers flag matroids with nontrivial coefficients are the papers \cite{Brandt-Eur-Zhang21,Joswig-Loho-Luber-Olarte21} on valuated flag matroids and, implicitly, the work on Macphersonnians as in \cite{Mnev-Ziegler93,Anderson98,Bjorner-LasVergnas-Sturmfels-White-Ziegler99}. 

\subsection*{Summary of results}
In this paper, we extend the notion of flag matroids to flag $F$-matroids for any tract $F$, and we exhibit cryptomorphic axiomatizations in terms of Grassmann-Pl\"ucker functions, dual pairs and circuits / vectors, as well as some additional descriptions in the case of perfect tracts. We also establish duality and minors. We explain all these aspects in terms of geometric constructions for the moduli space of flag matroids, which can be thought of as a flag variety over the {regular partial field $\Funpm$}.

\subsection*{Flag matroids}
Let $E=\{1,\dotsc,n\}$ and $\br=(r_1,\dotsc,r_s)$ with $0\leq r_1\leq\dotsc\leq r_s\leq n$.\footnote{We allow for $r_i=r_{i+1}$ since strict inequalities $r_i<r_{i+1}$ are not preserved under contractions and deletions.} A \emph{flag matroid of rank $\br$ on $E$} is a sequence $\bM=(M_1,\dotsc,M_s)$ of matroids $M_i$ of rank $r_i$ on $E$ such that every flat of $M_i$ is a flat of $M_{i+1}$ for $i=1,\dotsc,s-1$. We also say that \emph{$M_i$ is a quotient of $M_{i+1}$}, and write $M_{i+1}\twoheadrightarrow M_i$, in this case.

\begin{ex*}[Flags of matroid minors]
 A primary example of a flag matroid on $E$ is of the following form. Let $0\leq n_s \leq  \dotsc \leq n_1 \leq p$ and $n'=n+p$ be integers and $M'$ be a matroid on $E'=\{1,\dotsc,n'\}$. We define $M_i=M'\minor{I_i}{J_i}$ with $J_i=\{n+1,\dotsc,n+n_i\}$ and $I_i=\{n+n_i+1,\dotsc,n'\}$ for $i=1,\dotsc,s$. Then $(M_1,\dotsc,M_s)$ is a flag matroid and, in fact, every flag matroid is of this form; see \cite{Kung77}.
\end{ex*}

For the purpose of this introduction, we assume that the reader is familiar with tracts and Baker-Bowler theory; we refer to \autoref{section: Baker-Bowler theory} for a summary. Let $\K$ be the \emph{Krasner hyperfield}, whose incarnation as a tract is given by its unit group $\K^\times=\{1\}$ and its nullset $N_\K=\N-\{1\}$. Using Baker and Bowler's reinterpretation of a matroid $M$ of rank $r$ on $E$ as a $\K^\times$-class of a Grassmann-Pl\"ucker function $\varphi:E^r\to\K$ (mapping bases to $1$ and dependent sets to $0$) leads us towards the following cryptomorphism for flag matroids, which
is \autoref{prop: common flag matroids} and which is essentially known, as explained in \autoref{rem: crypto with base exchange}.

\begin{propA}\label{propA}
 Let $\bM=(M_1,\dotsc,M_s)$ be a sequence of matroids $M_i$ with respective Grassmann-Pl\"ucker functions $\varphi_i:E^{r_i}\to \K$. Then $\bM$ is a flag matroid if and only if for all $1\leq i\leq j\leq s$ and $x_1,\dotsc,x_{r_i-1},y_1,\dotsc,y_{r_j+1}\in E$,
 \[
  \sum_{k=1}^{r_j+1} \ \epsilon^k \, \varphi_i(y_k,x_1,\dotsc,x_{r_i-1}) \, \varphi_j(y_1,\dotsc,\widehat{y_k},\dotsc,y_{r_j+1})
 \]
 is in the nullset $N_\K$ where $\epsilon=1$ in this case.
\end{propA} 

This reinterpretation is amenable to a generalization to flag matroids over tracts.\footnote{Our definition of flag $F$-matroids extends Baker and Bowler's definition of a \emph{strong $F$-matroid}. It is perceivable that there is also a satisfactory theory of \emph{weak flag $F$-matroids}. We chose to work with strong flag $F$-matroids since the work \cite{Anderson19} by Anderson and \cite{Baker-Lorscheid21} by Baker and the second author show that strong matroids are better behaved with respect to vector axioms and moduli spaces.}

\begin{df*}
 Let $F$ be a tract. A \emph{flag $F$-matroid of rank $\br$ on $E$} is a sequence $\bM=(M_1,\dotsc,M_s)$ of $F$-matroids $M_i$ such that any choice of Grassmann-Pl\"ucker functions $\varphi_i:E^{r_i}\to F$ with $M_i=[\varphi_i]$ for $i=1,\dotsc,s$ satisfies the \emph{Pl\"ucker flag relations}
 \[
  \sum_{k=1}^{r_j+1} \ \epsilon^k \, \varphi_i(y_k,x_1,\dotsc,x_{r_i-1}) \, \varphi_j(y_1,\dotsc,\widehat{y_k},\dotsc,y_{r_j+1}) \quad \in \quad N_F
 \]
 for all $1\leq i\leq j\leq s$ and $x_1,\dotsc,x_{r_i-1},y_1,\dotsc,y_{r_j+1}\in E$.
\end{df*}

The stepping stone from which our theory of flag $F$-matroids lifts off the ground is the extension of Baker and Bowler's cryptomorphic description of $F$-matroids to flag $F$-matroids. Given an $F$-matroid $M$, we denote by $\cC^\ast(M)$ its set of cocircuits and by $\cV^\ast(M)$ its set of covectors. The following is \autoref{thm-cryptomorphism}.

\begin{thmA}\label{thmB}
 A sequence $\bM=(M_1,\dotsc,M_s)$ of $F$-matroids is a flag $F$-matroid if and only if $\cC^\ast(M_i)\subset\cV^\ast(M_j)$ for all $1\leq i\leq j\leq n$.
\end{thmA}

It's notable that the circuit characterization of flag $F$-matroids is, in fact, the analog of Baker-Bowler's dual pair characterization since $\cC^\ast(M_i)\subset\cV^\ast(M_j)$ says nothing less than that the circuit set $\cC(M_j)$ of $M_j$ is orthogonal to $\cC^\ast(M_i)$.

For perfect tracts, such as $\K$, $\T$ and (partial) fields, we find the following, \emph{a priori} different, characterizations of flag matroids (see \autoref{thm: cryptomorphisms for perfect tracts}), which reflect the more common descriptions of flags of linear subspaces over a field. 

\begin{thmA}\label{thmC}
 Let $F$ be a perfect tract and $\bM=(M_1,\dotsc,M_s)$ a sequence of matroids. Then the following are equivalent:
 \begin{enumerate}
  \item $\bM$ is a flag $F$-matroid;
  \item $\bM$ satisfies the {Pl\"ucker flag relations}
   \[
    \sum_{k=1}^{r_{j}+1} \ \epsilon^k \, \varphi_i(y_k,x_1,\dotsc,x_{r_i-1}) \, \varphi_{j}(y_1,\dotsc,\widehat{y_k},\dotsc,y_{r_{j}+1}) \quad \in \quad N_F
   \]
   for all $j=i+1\in\{2,\dotsc,s\}$ and $x_1,\dotsc,x_{r_i-1},y_1,\dotsc,y_{r_{j}+1}\in E$;
  \item $\cV^\ast(M_1)\subset \dotsb\subset\cV^\ast(M_s)$. 
 \end{enumerate}
\end{thmA}

\begin{ex*}[Flags of linear subspaces as flag $K$-matroids]
 As a tract, a field $K$ consists of the unit group $K^\times$ and the nullset $N_K=\{\sum a_i\in\N[K^\times]\mid \sum a_i=0\text{ as elements of }K\}$. By \autoref{thmC}, a $K$-matroid is a sequence $\bM=(M_1,\dotsc,M_s)$ of $K$-matroids whose associated covector sets $\cV^\ast(M_i)$ form a chain $\cV^\ast(M_1)\subset\dotsb\subset\cV^\ast(M_s)$ of linear subspaces of $K^E$ with $\dim\cV^\ast(M_i)=r_i$. For more details, see \autoref{prop: flag matroids over fields}.
\end{ex*}

\begin{ex*}[Tropical flag matroids as flag $\T$-matroids]
 As a tract, the tropical hyperfield $\T$ consists of the (multiplicative) unit group $\T^\times=\R_{>0}$ and the nullset 
 \[
  N_\T \ = \ \big\{a_1+\dotsc+a_n\in\N[\R_{>0}] \, \big| \, \text{the maximum occurs twice in }a_1,\dotsc,a_n\big\}.
 \]
 An immediate consequence of this definition is that the bend locus\footnote{Note that we use the Berkovich convention for the tropical semifield $\overline{\R}=\R_{\geq0}$ with addition $\max$ and the usual multiplication of real numbers.} $\cB(f)\subset\overline{\R}^n$ of a tropical polynomial $f(x_1,\dotsc,x_n)=\sum c_{(e_1,\dotsc,e_n)}x_1^{e_1}\dotsb x_n^{e_n}$ agrees with tropical points $(a_1,\dotsc,a_n)\in\T^n$ for which $f(a_1,\dotsc,a_n)\in N_\T$.
 
 Brandt, Eur and Zhang describe \emph{valuated flag matroids} in different disguises: one of them (\cite[Prop.\ 4.2.3]{Brandt-Eur-Zhang21}) identifies them as points of the \emph{flag Dressian}, which is defined by the very same equations that go under the name of Pl\"ucker flag relations in our text. This shows that a valuated flag matroid after Brandt, Eur and Zhang is the same thing as a flag $\T$-matroid in our sense. Since $\T$ is perfect and since a tropical linear space is the covector set of a valuated matroid, \autoref{thmC} identifies a flag $\T$-matroid with a chain $V_1\subset\dotsb\subset V_s$ of tropical linear spaces, which recovers \cite[Thm.\ 4.3.1]{Brandt-Eur-Zhang21}. For more details, see \autoref{prop: valuated matroids}.
\end{ex*}

\begin{ex*}[Flag $\Funpm$-matroids]% and regular flag matroids]
 In its incarnation as a tract, the regular partial field $\Funpm$ is given by its unit group $(\Funpm)^\times=\{1,\epsilon\}$ and its nullset $N_\Funpm=\{n.1+n.\epsilon\mid n\in\N\}$. Since $\Funpm$ is perfect, \autoref{thmC} implies that a flag $\Funpm$-matroid is a sequence $\bM=\big([\varphi_1],\dotsc,[\varphi_s]\big)$ of $(\Funpm)^\times$-classes of non-trivial functions $\varphi_i:E^{r_i}\to\Funpm$ for which
 \[
  \sum_{k=1}^{r_{i+1}+1} \ \epsilon^k \, \varphi_i(y_k,x_1,\dotsc, x_{r_i-1}) \,\varphi_{i+1}(y_1,\dotsc,\widehat{y_k},\dotsc,y_{r_{i+1}+1}) \quad \in \quad N_{\Funpm}
 \]
 for $i=1,\dotsc,s-1$ and all $x_1,\dotsc,x_{r_i-1},y_1,\dotsc,y_{r_{i+1}+1}\in E$. Since every regular matroid is represented by a unimodular matrix, whose rows span the set of covectors (considered as elements of $\Z^E$ with coefficients $0$, $1$ and $-1$), a sequence $(A_1,\dotsc,A_s)$ of unimodular $r_i\times n$-matrices $A_i$ represents a regular flag matroid if and only if the row space of $A_i$ is contained in the row space of $A_{i+1}$ for $i=1,\dotsc,s-1$.
\end{ex*}

\begin{ex*}[Oriented flag matroids]
 The \emph{sign hyperfield $\S$} is the tract with unit group $\S^\times=\{1,\epsilon\}$ and nullset $N_\S=\{n.1+m.\epsilon\mid n=m=0\text{ or }n\neq0\neq m\}$. Since $\S$ is perfect, we can invoke \autoref{thmC} and describe a flag $\S$-matroid, or \emph{oriented flag matroid}, as a sequence $\bM=(M_1,\dotsc,M_s)$ of oriented matroids $M_i$ that are represented by Grassmann-Pl\"ucker functions, or \emph{chirotopes}, $\varphi_i:E^{r_i}\to\S$ that satisfy
 \[
  \sum_{k=1}^{r_{i+1}+1} \ \epsilon^k \, \varphi_i(y_k,x_1,\dotsc, x_{r_i-1}) \,\varphi_{i+1}(y_1,\dotsc,\widehat{y_k},\dotsc,y_{r_{i+1}+1}) \quad \in \quad N_{\S}
 \]
 for $i=1,\dotsc,s-1$ and all $x_1,\dotsc,x_{r_i-1},y_1,\dotsc,y_{r_{i+1}+1}\in E$. Equivalently, a sequence $\bM=(M_1,\dotsc,M_s)$ of oriented matroids is an oriented flag matroid if every cocircuit of $M_i$ is in the span of the cocircuits of $M_{i+1}$ for $i=1,\dotsc,s-1$.
 
 Note that in our terminology, the MacPhersonian $\textup{MacPh}(r,N)$ of a rank $w$ matroid $N$ consists of all oriented flag matroids $(M,N)$ of type $(r,w)$.
\end{ex*}

Duality, minors of flag matroids and push forwards can be directly derived from the analogous constructions for the $F$-matroids of the flag (see \autoref{rem: properties of falg matroids}, \autoref{prop: functoriality}, \autoref{thm: duality}, \autoref{thm: minors}).

\begin{thmA}\label{thmD}
 Let $\bM=(M_1,\dotsc,M_s)$ be a flag $F$-matroid of rank $\br$ on $E$. Let $I$ and $J$ be disjoint subsets of $E$ and $\bi=(i_1,\dotsc,i_t)$ with $1\leq i_1\leq \dotsb\leq i_t\leq s$. Then
 \[
  \bM^\ast = (M_1^\ast,\dotsc,M^\ast_s), \quad \bM\minor IJ = (M_1\minor IJ,\dotsc,M_s\minor IJ), \quad\pi_\bi(\bM) = (M_{i_1},\dotsc,M_{i_t})
 \]
 are flag $F$-matroids, and $(\bM\minor IJ)^\ast=\bM^\ast\minor JI$. Given a tract morphism $f:F\to F'$, the sequence $f_\ast\bM=(f_\ast M_1,\dotsc, f_\ast M_s)$ is a flag $F'$-matroid.
\end{thmA}

Our initial example of a flag matroid as a sequence of matroid minors extends to flag $F$-matroids in the following way (see \autoref{thm: flags of minors}, \autoref{prop: contraction-deletion flag}).

\begin{thmA}\label{thmE}
 Let $0\leq n_s \leq  \dotsc \leq n_1 \leq p$ and $n'=n+p$ be integers and $M'$ be an $F$-matroid on $E'=\{1,\dotsc,n'\}$. For $i=1,\dotsc,s$, we define $M_i=M'\minor{I_i}{J_i}$ where $J_i=\{n+1,\dotsc,n+n_i\}$ and $I_i=\{n+(n_i+1),\dotsc,n'\}$. Then $(M_1,\dotsc,M_s)$ is a flag $F$-matroid. If $F$ is perfect, then every flag $F$-matroid of rank $(r,r+1)$ is of this form.
\end{thmA}

%%%%%%%%%%%%%%%%%%%%%%%%%%%%%%%%%%%%%%%%%%%%%%%%%%%%%%%%%%%%%%%%%%%%%%%%%%%%%%%%%%%%%%%%%%%%%%%%%%%%%%%%%%%%%%%%%%%%%%%%%%%%%%%%%%%%%%%%%%%%%%%%%%%%%%%
\subsection*{Application to representation theory}

\autoref{thmE} has some interesting consequences for the representation theory of flag matroids. We explain a sample application in the following. For a tract $F$, we denote by $t_F:F\to \K$ the unique tract morphism to the Krasner hyperfield $\K$. A flag matroid is 
\begin{itemize}
 \item \emph{regular} if it is of the form $t_{\Funpm,\ast}\bM$ for a flag $\Funpm$-matroid $\bM$;
 \item \emph{binary} if it is of the form $t_{\F_2,\ast}\bM$ for a flag $\F_2$-matroid $\bM$;
 \item \emph{orientable} if it is of the form $t_{\S,\ast}\bM$ for a flag $\S$-matroid $\bM$.
\end{itemize}

\begin{thmA}\label{thmF}
 A flag matroid $\bM$ of rank $(r,r+1)$ is regular if and only if it is binary and orientable.
\end{thmA}

\begin{proof}
 Since the tract morphism $t_{\Funpm}:\Funpm\to\K$ factors through both $\F_2$ and $\S$, every regular flag matroid is binary and orientable.
 
 Conversely assume that $\bM$ is a binary and orientable flag matroid of rank $(r,r+1)$, i.e., $\bM=t_{\F_2,\ast}(\bM_{\F_2})$ for a flag $\F_2$-matroid $\bM_{\F_2}$ and $\bM=t_{\S,\ast}(\bM_{\S})$ for a flag $\S$-matroid $\bM_{\S}$. Since $\F_2$ and $\S$ are perfect, we can apply \autoref{thmE} to find an $\F_2$-matroid $M'_{\F_2}$ and an $\S$-matroid $M'_{\S}$ with $\bM_{\F_2}=(M'_{\F_2}/e, \, M'_{\F_2}\backslash e)$ and $\bM_{\S}=(M'_{\S}/e, \, M'_{\S}\backslash e)$. 
% \[
%  \bM_{\F_2} \ = \ (M'_{\F_2}/e, \, M'_{\F_2}\backslash e) \qquad \text{and} \qquad \bM_{\S} \ = \ (M'_{\S}/e, \, M'_{\S}\backslash e). 
% \]
 By \cite[Prop.\ 5.1]{Kung77}, $\bM=(M'/e,\, M'\backslash e)$ for a rank $r+1$-matroid $M'$ on $E'=E\cup\{e\}$, which is unique by \autoref{prop: contraction-deletion flag} as $\K^\times=\{1\}$. Since push-forwards commute with taking minors, we conclude that $t_{\F_2,\ast}(M'_{\F_2})=M'=t_{\S,\ast}(M'_{\S})$, which shows that $M'$ is binary and orientable. By \cite[Cor.\ 6.2.6]{Bland-LasVergnas78}, $M'=t_{\Funpm,\ast}(M'_{\Funpm})$ for an $\Funpm$-matroid $M'_{\Funpm}$. Using again the compatibility of push-forwards with minors yields that $\bM=t_{\Funpm,\ast}(\bM_\Funpm)$ for the flag $\Funpm$-matroid $\bM_\Funpm=(M'_\Funpm/e,\, M'_{\Funpm}\backslash e)$, which proves our claim.
\end{proof}

The strategy of this proof does not extend to other ranks $\br$ since not every flag $F$-matroid is a sequence of minors of a single matroid $M'$ (where $F$ stands for an arbitrary perfect tract). But it is perceivable that other techniques from the representation theory of matroids generalize to flag matroids of arbitrary rank. As a sample problem for future investigations, we pose the question:
\begin{problem*}
 Is every binary and orientable flag matroid regular?
\end{problem*}

%%%%%%%%%%%%%%%%%%%%%%%%%%%%%%%%%%%%%%%%%%%%%%%%%%%%%%%%%%%%%%%%%%%%%%%%%%%%%%%%%%%%%%%%%%%%%%%%%%%%%%%%%%%%%%%%%%%%%%%%%%%%%%%%%%%%%%%%%%%%%%%%%%%%%%%
\subsection*{The moduli space}

The theory of moduli spaces of matroids from Baker and the second author's paper \cite{Baker-Lorscheid21} extends to flag matroids, utilizing ordered blue schemes. We recall some aspects from the theory of ordered blueprints and ordered blue schemes and refer to \cite{Baker-Lorscheid21} for full details.

An \emph{ordered blueprint $B$} consists of a commutative semiring $B^+$ with $0$ and $1$ together with a multiplicatively closed subset $B^\bullet$ that contains $0$ and $1$ and that generates $B^+$ as a semiring and together with a partial order $\leq$ on $B^+$ that is closed under addition and multiplication in the sense that $x\leq y$ implies $x+z\leq y+z$ and $xz\leq yz$. A tract $F$ defines the ordered blueprint $B=F^\oblpr$ with ambient semiring $B^+=\N[F^\times]$, underlying monoid $B^\bullet=F$ and partial order $\leq$ that is generated by the relations $0\leq x$ for $x\in N_F$. 

Since $\leq$ is closed under addition, we lose information about tracts $F$ for which $N_F$ is not closed under addition. Therefore we restrict our attention in this part of the paper to \emph{idylls}, which are tracts with additively closed nullset and which can be identified with the associated ordered blueprint. This technical restriction is mild since all (partial) fields and all hyperfields, including $\Funpm$, $\K$, $\T$ and $\S$, are idylls. In the following, $F$ denotes an idyll.

Flag $F$-matroids are canonically identified with $F$-rational points of the flag variety $\Fl(\br,E)$ over $\Funpm$, which is defined as the closed subscheme of $\prod_{i=1}^s \P_\Funpm^{{n}^{r_i}-1}$ given by the identities $T_{(\sigma(x_1),\dotsc,\sigma(x_{r_i}))}=\sign(\sigma)T_{(x_1,\dotsc,x_{r_i})}$ for $\sigma\in S_{r_i}$ and $T_{(x_1,\dotsc,x_{r_i})}=0$ if $\#\{x_1,\dotsc,x_{r_i}\}<r_i$, as well as the multi-homogeneous Pl\"ucker flag relations
\[
 0 \ \leq \ \sum_{k=1}^{r_j+1} \ \epsilon^k \, T_{(y_k,x_1,\dotsc,x_{r_i-1})} \, T_{(y_1,\dotsc,\widehat{y_k},\dotsc,y_{r_j+1})}
\]
for all $1\leq i\leq j\leq s$ and $x_1\dotsc,x_{r_i-1},y_1,\dotsc,y_{r_j+1}\in E$. As the special case $\br=(r_1)$, the flag variety is the Grassmannian, or \emph{matroid space}, $\Gr(r_1,E)=\Mat(r_1,E)$; see \cite[section 5]{Baker-Lorscheid21}. The following is \autoref{cor: flag matroids as rational points}.

\begin{thmA}\label{thmG}
 Let $F$ be an idyll. Then there is a canonical bijection between flag $F$-matroids and $F$-rational points of $\Fl(\br,E)$.
\end{thmA}

\autoref{thmG} generalizes the fact that the $K$-rational points of a flag variety correspond to flags of $K$-linear subspaces for fields $K$ and that the points of the flag Dressian correspond to flags of tropical linear spaces; see \cite{Brandt-Eur-Zhang21}. In fact, \autoref{thmG} follows from the stronger property that $\Fl(\br,E)$ is the fine moduli space of flag matroid bundles; see \autoref{thm: moduli property} for the precise statement.

The previously discussed constructions for flag $F$-matroids extend to geometric constructions in terms of certain canonical morphisms between flag varieties under the identification of flag $F$-matroids with $F$-rational points of $\Fl(\br,E)$ in \autoref{thmG}. 

\subsubsection*{Coordinate projection} 
Let $\bi=(i_1,\dotsc,i_t)$ be a sequence of integers with $1\leq i_1\leq\dotsb\leq i_t\leq s$ and $\br'=(r_{i_1},\dotsc,r_{i_t})$. Then there is a canonical morphism $\pi_\bi:\Fl(\br,E)\to\Fl(\br',E)$, which maps a flag $F$-matroid $\bM=(M_1,\dotsc,M_s)$ to $\bM'=(M_{i_1},\dotsc,M_{i_t})$. This is \autoref{prop: coordinate projections}.

\subsubsection*{Duality} 
Define $\br^\ast=(r_s^\ast,\dotsc,r_1^\ast)$ with $r_i^\ast=n-r_i$. Then there is a canonical morphism $\Fl(\br,E)\to\Fl(\br^\ast,E)$, which maps a flag $F$-matroid $\bM$ to its dual $\bM^\ast$ (after composition with the duality involution of $F$). This is \autoref{thm: dual flag variety}.

\subsubsection*{Minors} 
Let $e\in E$ and $\br_k=(r_1-1,\dotsc,r_k-1,r_{k+1},\dotsc,r_s)$. Then there exist locally closed subschemes $W_{\br,k,/e}\hookrightarrow\Fl(\br,E)$ and morphisms $\Psi_{\br,k,/e}:W_{\br,k,/e}\to \Fl(\br_k,E)$ such that $\Fl(\br,E)(F)=\coprod_{k=1}^sW_{\br,k,/e}(F)$ and such that $\Psi_{\br,k,/e}$ maps a flag $F$-matroid $\bM$ to its contraction $\bM/e$. There are analogous morphisms for the deletion $\bM\backslash e$. This is \autoref{thm: minors of flag varieties}.

\subsubsection*{Flags of minors} 
Let $n'=n+r_s-r_1$ and $E'=\{1,\dotsc,n'\}$. Let $W_{\br,E}$ be the open locus in $\Gr(r_s,E')$ of matroids for which $\{n+1,\dotsc,n'\}$ is independent and co-independent. Then there is a canonical morphism $W_{\br,E}\to\Fl(\br,E)$, which sends an $F$-matroid $M'$ of rank $r_s$ on $E'$ to the flag $F$-matroid $(M'\minor{I_1}{J_1},\dotsc,M'\minor{I_s}{J_s})$ of rank $\br$ on $E$ where $J_i=\{n+1,\dotsc,r_s-r_i\}$ and $I_i=\{r_s-r_i+1,\dotsc,n'\}$. This is \autoref{thm: flags of minors morphism}.

As explained in our initial example, every flag matroid is a sequence of minors, i.e., $W_{\br,E}(\K)\to\Fl(\br,E)(\K)$ is surjective for all $\br$ and $E$. This is also true for fields $K$: the map $W_{\br,E}(K)\to\Fl(\br,E)(K)$ is surjective since it is $\GL(E,K)$-equivariant and since $\GL(E,K)$ acts transitively on $\Fl(\br,E)(K)$. Initially Las Vergnas expected the same for pairs of oriented matroids, but this was later disproven by Richter-Gebert; see \cite[Cor.\ 3.5]{Richter-Gebert93}. This makes us wonder:

\begin{problem*}
 For which $\br$, $E$ and $F$ is $W_{\br,E}(F)\to\Fl(\br,E)(F)$ surjective?
\end{problem*}

\subsubsection*{Topologies on rational point sets}
A topology on $F$ induces a topology on $\Fl(\br,E)(F)$, which recovers several known instances of flag varieties. If $F=K$ is a topological field, then $\Fl(\br,E)(K)$ is the flag variety over $K$ together with its strong topology; in particular, $\Fl(\br,E)(\R)$ is the manifold of flags of real linear subspaces. Endowing $F=\T$ with the real topology identifies $\Fl(\br,E)(\T)$ with the flag Dressian from \cite{Brandt-Eur-Zhang21} as a tropical variety. Endowing the sign hyperfield $\S$ with the topology generated by the open subsets $\{1\}$ and $\{\epsilon\}$ yields a generalization of the MacPhersonian to oriented flag matroids. Endowing the Krasner hyperfield $\K$ with the topology generated by the open subset $\{1\}$ endows $\Fl(\br,E)(\K)$ with a topology, in which a matroid $M$ is in the closure of another matroid $N$ precisely if all bases of $M$ are bases of $N$. A more detailed discussion of topologies on rational point sets can be found in \autoref{subsection: rational point sets}.

\subsection*{Relation to combinatorial flag varieties and Tits geometries}% of Borovik, Gelfand and White}

Borovik, Gelfand and White introduce in \cite{Borovik-Gelfand-White01} the \emph{combinatorial flag variety $\Omega_W$} for the symmetric group $W=S_n$ as the order complex of the collection of all matroids (of arbitrary rank) on $E=\{1,\dotsc,n\}$, endowed with the partial order $M\leq N$ if and only if $N\twoheadrightarrow M$. Thus $\Omega_W$ is a chamber complex, and its chambers are indexed by flag matroids on $E$. The maximal chambers have dimension $n-2$ and correspond to flag matroids of rank $(1,\dotsc,n-1)$.

The same authors mention at the end of section 7.14 of their book \cite{Borovik-Gelfand-White03} that:
\begin{quote}\small%\footnotesize
 Many geometries over fields have formal analogues which can be thought of as geometries over the field of $1$ element. For example, the projective plane over the field $\F_q$ has $q^2+q+1$ points and the same number of lines; every line in the plane has $q+1$ points. When $q=1$, we have a plane with three points and three lines, i.e., a triangle. The flag complex of the triangle is a thin building of type $A_2 = Sym_3$. In general, the Coxeter complex $\cW$ of a Coxeter group $W$ is a thin building of type $W$ and behaves like the building of type $W$ over the field of $1$ element.

 However, the Coxeter complex has a relatively poor structure. In many aspects, $\Omega_W$ and $\Omega_W^\ast$ are more suitable candidates for the role of a ``universal'' combinatorial geometry of type $W$ over the field of $1$ element.
\end{quote}

For the Coxeter group $W=S_n$ of type $A_{n-1}$, the combinatorial flag variety $\Omega_W$ resurfaces as the set of $\K$-rational points of the $\Funpm$-schemes $\Fl(\br,E)$ (with varying $\br$). More precisely, the chambers of $\Omega_W$ correspond bijectively to
\[
 \coprod_{\br\in\Theta} \ \Fl(\br,E)(\K) \quad \text{where} \quad \Theta \ = \ \big\{ \, (r_1,\dotsc,r_s) \, \big| \, s>0,\; 0<r_1<\dotsb<r_s<n \, \big\},
\]
and the chamber of a flag matroid $\bN$ is the face of the chamber of a flag matroid $\bM$ if and only if $\bN=\pi_\bi(\bM)$ for an appropriate coordinate projection $\pi_\bi$ (see \autoref{thmD}).

Moreover, Borovik, Gelfand and White observe in \cite{Borovik-Gelfand-White01} that the Coxeter complex $\cW$ of $W$ appears naturally as the subcomplex of $\Omega_W$ that consists of flags of matroids with exactly one basis, which correspond to the closed points of $\Fl(\br,E)(\K)$. This links their idea to Tits' seminal paper \cite{Tits57} on $\Fun$, where Tits introduces \emph{geometries}, which can be thought of as a predecessor of a building over a finite field, and where he muses over the (lack of a) \emph{field of characteristic one}, which could explain the role of the Coxeter complexes $\cW$. 

In so far, our flag varieties $\Fl(\br,E)$, together with the various coordinate projections $\pi_\bi$, can be seen as an enrichment of both Tits' and Borovik, Gelfand and White's perspectives on $\Fun$-geometry.

%%%%%%%%%%%%%%%%%%%%%%%%%%%%%%%%%%%%%%%%%%%%%%%%%%%%%%%%%%%%%%%%%%%%%%%%%%%%%%%%%%%%%%%%%%%%%%%%%%%%%%%%%%%%%%%%%%%%%%%%%%%%%%%%%%%%%%%%%%%%%%%%%%%%%%%%%%%%%%%%%%%%%%%%%%%

\subsection*{Acknowledgements}
We would like to thank Matthew Baker for several discussions and Christopher Eur for pointing us to the work on strong maps of oriented matroids. We would like to thank Eduardo Vital for his comments on a previous version. The first author was supported by a CNPq fellowship - Brazil (140325/2019-0).

%%%%%%%%%%%%%%%%%%%%%%%%%%%%%%%%%%%%%%%%%%%%%%%%%%%%%%%%%%%%%%%%%%%%%%%%%%%%%%%%%%%%%%%%%%%%%%%%%%%%%%%%%%%%%%%%%%%%%%%%%%%%%%%%%%%%%%%%%%%%%%%%%%%%%%%%%%%%%%%%%%%%%%%%%%%
%%%%%%%%%%%%%%%%%%%%%%%%%%%%%%%%%%%%%%%%%%%%%%%%%%%%%%%%%%%%%%%%%%%%%%%%%%%%%%%%%%%%%%%%%%%%%%%%%%%%%%%%%%%%%%%%%%%%%%%%%%%%%%%%%%%%%%%%%%%%%%%%%%%%%%%%%%%%%%%%%%%%%%%%%%%
\section{Baker-Bowler theory}
\label{section: Baker-Bowler theory}
%%%%%%%%%%%%%%%%%%%%%%%%%%%%%%%%%%%%%%%%%%%%%%%%%%%%%%%%%%%%%%%%%%%%%%%%%%%%%%%%%%%%%%%%%%%%%%%%%%%%%%%%%%%%%%%%%%%%%%%%%%%%%%%%%%%%%%%%%%%%%%%%%%%%%%%%%%%%%%%%%%%%%%%%%%%
%%%%%%%%%%%%%%%%%%%%%%%%%%%%%%%%%%%%%%%%%%%%%%%%%%%%%%%%%%%%%%%%%%%%%%%%%%%%%%%%%%%%%%%%%%%%%%%%%%%%%%%%%%%%%%%%%%%%%%%%%%%%%%%%%%%%%%%%%%%%%%%%%%%%%%%%%%%%%%%%%%%%%%%%%%%

In this section, we review the theory of matroids over a tract, as developed by Baker and Bowler in \cite{Baker-Bowler19}.

\subsection{Tracts}

A \emph{pointed monoid} is a (multiplicatively written) monoid $F$ with neutral element $1$ and an \emph{absorbing element $0$} (or \emph{zero} for short), which is characterized by the property that $0\cdot a=0$ for all $a\in F$. The \emph{unit group of $F$} is the submonoid $F^\times$ of all invertible elements of $F$, which is a group.

A \emph{tract} is a commutative pointed monoid $F$ with unit group $F^\times=F-\{0\}$ together with a subset $N_F$ of the group semiring $\N[F^\times]$, called the \emph{nullset of $F$}, which satisfies:
\begin{enumerate}[label = (T\arabic*)]\setcounter{enumi}{-1}
    \item The zero element of $\N[F^\times]$ belongs to $N_F$.
    \item The multiplicative identity of $\N[F^\times]$ is not in $N_F$.
    \item There is a unique element $\epsilon$ in $F^\times$ with $1+\epsilon \in N_F$.
    \item $N_F$ is closed under the natural action of $F^\times$ on $\N[F^\times]$.
\end{enumerate}
Note that the axioms imply that $\epsilon^2=1$ and that $a+b\in N_F$ if and only if $b=\epsilon a$. A \emph{morphism between tracts $F$ and $F'$} is a multiplicative map $f:F\to F'$ such that $f(0)=0$ and $f(1)=1$ and such that $\sum f(a_i) \in N_{F'}$ whenever $\sum a_i \in N_F$. 

\begin{ex}[Fields]
\label{field-tract}
 A field $K$ can be considered as the tract whose multiplicative monoid equals that of $K$ and whose nullset is $N_{K}=\{\sum a_i \in \N[K^\times]|\; \sum a_i = 0 \text{ in } K\}$. If the context is clear, we denote the tract by the same symbol $K$, and we say that a tract $K$ is a field if it is associated with a field.
 
 More generally, partial fields can be considered as tracts, as explained in \cite[Thm.\ 2.21]{Baker-Lorscheid21}. The most relevant example for our purposes is the \emph{regular partial field $\Funpm$}, whose appearance as a tract consists of the multiplicative monoid $\Funpm=\{0,1,\epsilon\}$ and the nullset $N_\Funpm=\{n.1+m.\epsilon\mid n=m\}$. 
\end{ex}

\begin{ex}[Hyperfields]
A \textit{hyperoperation} on a set $S$ is a map $S\times S \rightarrow 2^S$. A \textit{hyperfield} is a generalization of a field whose addition is replaced by a hyperoperation $(a,b)\mapsto a\hyperplus b$, which satisfies analogous properties to the addition of a field. We consider a hyperfield $K$ as the tract $(K^\times, N_{K})$ with nullset $N_K=\big\{\sum a_i \in \N[K^\times] \, \big| \, 0 \in \hypersum a_i\big\}$. 

Some particular examples are the following:
\begin{itemize}
 \item The \textit{Krasner hyperfield} is the tract $\K=\{0,1\}$ with nullset $N_\K=\N-\{1\}$ and $\epsilon=1$. Its hyperaddition is given by $0\hyperplus a = \{a\}$ for $a=0,1$ and $1\hyperplus 1 = \{0, 1\}$. 
 \item The \textit{tropical hyperfield} is the tract $\T=\R_{\geq0}$ with nullset 
       \[\textstyle
        N_\T \ = \ \big\{\sum a_i\in\N[\R_{>0}] \,\big|\, \text{the maximum appears twice in }\{a_i\}\big\}
       \]
        and $\epsilon=1$. Its hyperaddition is given by $a\hyperplus b=\big\{\max\{a,b\}\big\}$ for $a\neq b$ and $a\hyperplus a=[0,a]$. 
 \item The \textit{sign hyperfield} is the tract $\S=\{0,1,-1\}$ with nullset 
       \[
        N_\S \ = \ \big\{n.1+m.(-1) \, \big| \, n=m=0\text{ or }n\neq0\neq m\big\}
       \]
       and $\epsilon=-1$. Its hyperaddition satisfies $a\hyperplus a=\{a\}$ and $a\hyperplus(-a)=\{0,1,-1\}$ for $a=1,-1$.
\end{itemize}
Note that $\K$ is a terminal object in the category of tracts: for every tract $F$, the \emph{terminal map} $t_F:F\to \K$ that maps $F^\times$ to $1$ is the unique tract morphism from $F$ to $\K$.
\end{ex}

\subsection{Matroids over tracts}

Let $E=\{1,\dotsc,n\}$ and $0\leq r\leq n$. A \textit{Grassmann–Plücker function of rank $r$ on $E$ with coefficients in $F$} is a function $\varphi: E^r\rightarrow F$ such that: 
\begin{enumerate}[label = (GP\arabic*)]
    \item\label{GP1} $\varphi$ is not identically zero;
    \item\label{GP2} $\varphi$ is alternating, i.e., $\varphi(x_1, \dotsc, x_i, \dotsc, x_j, \dotsc, x_r) = \epsilon \cdot \varphi(x_1, \dotsc, x_j, \dotsc, x_i, \dotsc, x_r)$ and $\varphi(x_1, \dotsc, x_r) = 0$ if $x_i = x_j$ for some $i\neq j$;
    \item\label{GP3} for all $x_1, \dotsc,  x_{r-1}, y_1, \dotsc, y_{r+1}\in E$, we have 
    \[
     \underset{k=1}{\overset{r+1}{\sum}} \ \epsilon^k \cdot \varphi(y_1, \dotsc, \widehat{y_k}, \dotsc, y_{r+1}) \cdot \varphi(y_k, x_1, \dotsc, x_{r-1}) \quad \in \quad N_F.
    \]
\end{enumerate}
The relations in \ref{GP3} are called the \emph{Pl\"ucker relations}.

We say that two Grassmann–Plücker functions $\varphi_1$ and $\varphi_2$ are \textit{equivalent} if $\varphi_1 = a\cdot \varphi_2$ for some $a\in F^{\times}$, and define an \textit{$F$-matroid (of rank $r$ on $E$)} as the equivalence class $M_\varphi=[\varphi]$ of a Grassmann–Plücker function $\varphi:E^r\to F$.

\begin{rem}
 These are, in fact, the definition of \textit{strong} Grassmann-Pl\"ucker functions and \textit{strong} $F$-matroids in \cite{Baker-Bowler19}. Though weak matroids are important to understand the representations of matroids over fields and other tracts, strong matroids are more suitable to study ``cryptomorphic'' properties (as in \cite{Anderson19}) and ``algebro-geometric'' properties (as in \cite{Baker-Lorscheid21}). We will not encounter weak matroids in this text and omit the attribute ``strong.''
\end{rem}

\subsection{Push-forwards and the underlying matroid}
Let $f: F \rightarrow F'$ be a morphism of tracts and $\varphi:E^r\to  F$ a Grassmann-Pl\"ucker function. Then $f\circ \varphi:  E^r\rightarrow F'$ is also a Grassmann-Pl\"ucker function. We define the \textit{push-forward of $M_\varphi$ along $f$} as $f_\ast M_\varphi= M_{f\circ \varphi}$.

Let $M$ be an $F$-matroid. The \textit{underlying matroid of $M$} is defined as the classical matroid $\underline{M}$ whose set of bases is
\[
 \mathcal{B}(\underline{M}) = \big\{ \{x_1, \dotsc, x_r\} \subseteq E \, \big|\, \varphi(x_1, \dotsc, x_r) \neq 0 \big\}.
\]
For a tract morphism $f:F\to F'$ and an $F$-matroid $M$, we have $\underline{f_\ast M}=\underline M$. The map
\[
 \begin{array}{ccc}
  \big\{ \text{$\K$-matroids} \big\} & \longrightarrow & \big\{ \text{usual matroids} \big\} \\
  M & \longmapsto & \underline M
 \end{array}
\]
is a bijection that identifies classical matroids with $\K$-matroids.

\subsection{Circuits}
For a tuple $X = (X_i)_{i\in E}$ of $F^E$, we define the \textit{support of $X$} as the set $\underline{X}:=\{i \in E\mid X_i \neq 0\}$, and for a subset $\mathcal{X} \subseteq F^E$, we define the \textit{support of $\mathcal{X}$} as $\supp(\mathcal{X}):= \{ \underline{X}\mid X \in \mathcal{X}\}$. 

We define the set of \textit{$F$-circuits of $M$} as follows. Let $\mathcal{C}(\underline{M})$ be the set of circuits of $\underline M$. For each $C \in \mathcal{C}(\underline{M})$, fix a $y_0 \in C$ and a basis $\{y_1, \dotsc, y_r\}$ of $\underline{M}$ containing $C-y_0$. We define $X_C \in F^E$ by
\[
\begin{aligned}
X_C(y): =&  
    \begin{cases}
    \epsilon^i \cdot \varphi(y_0, \dotsc, \widehat{y_i}, \dotsc, y_r) & \text{if } y = y_i \text{ for some $i$,}\\
    0 & \text{if } y \notin \{y_0, \dotsc, y_r\}.
    \end{cases}
\end{aligned}
\]
The set of $F$-circuits of $M$ is given by $\mathcal{C}(M):= \{a\cdot X_C| \; a \in F^\times, \, C\in \mathcal{C}(\underline{M})\}$. It does not depend on the choice of elements $y_0\in C$ and bases $\{y_1, \dotsc ,y_r\}$ containing $C-y_0$. Note that $\supp\big(\mathcal{C}(M)\big) = \mathcal{C}(\underline{M})$ and that $\mathcal{C}(M)$ satisfies the following three properties:
\begin{enumerate}[start=0, label = (C\arabic*)]
 \item \label{C0} $0\notin \mathcal{C}(M)$.
 \item \label{C1} If $X\in \mathcal{C}(M)$ and $a \in F^\times$, then $a \cdot X \in \mathcal{C}(M)$.
 \item \label{C2} If $X,Y \in \mathcal{C}(M)$ and $\underline{X}\subseteq \underline{Y}$, then there exists $a \in F^\times$ such that $X =a\cdot Y$.
\end{enumerate}

\begin{rem}\label{rem 1.4}
 The $F$-circuits satisfy an elimination property, which characterizes together with \ref{C0} -- \ref{C2} the subsets of $F^E$ that are sets of $F$-circuits of an $F$-matroid $M$. Moreover, $M$ is determined by $\cC(M)$, which yields a cryptomorphic description of $F$-matroids in terms of $F$-circuits (see \cite[Thm. 3.17]{Baker-Bowler19}). We forgo to spell out the elimination axiom, however, since it is somewhat involved and since we do not use it in this text.
\end{rem}

\subsection{Duality}
An \emph{involution of $F$} is a tract morphism $\tau:F\to F$ such that $\tau^2$ is the identity on $F$. In the following, we fix an involution $\tau$ and write $\overline x=\tau(x)$. 

Fix a total order for $E=\{x_1,\dotsc,x_n\}$ and let $\sigma$ be the unique permutation such that $x_{\sigma(1)}<\dotsb< x_{\sigma(n)}$. We consider $\sign(x_1, \dotsc, x_n)=\sign(\sigma)\in\{\pm1\}$ as an element of $F$ by identifying $-1$ with $\epsilon$. 

The \textit{dual of $M$} is the $F$-matroid $M^* = M_{\varphi^*}$, where $\varphi^*$ is the Grassmann-Pl\"ucker function $\varphi^\ast:E^{n-r}\to F$ that is determined by
\[
 \varphi^\ast(x_1, \dotsc, x_{n-r}) \ = \ \sign(x_1, \dotsc, x_{n-r}, x'_1, \dotsc, x'_r)\cdot \overline{\varphi(x'_1, \dotsc, x'_r)}
\]
whenever $E=\{x_1, \dotsc, x_{n-r},x'_1, \dotsc, x'_r\}$. The dual of $M$ satisfies $M^{**} = M$, and the underlying matroid of $M^\ast$ is the dual of $\underline M$.

\subsection{Orthogonality}
Let $X,Y\in F^E$. The \textit{inner product of $X$ and $Y$} is
\[
 X\cdot Y \ =  \ \sum_{i\in E} X_i\cdot \overline{Y_i},
\]
considered as an element of $\N[F^{\times}]$. We say that $X$ and $Y$ are \textit{orthogonal}, and write $X\perp Y$, if $X\cdot Y \in N_F$. We say that two subsets $\mathcal{X}$ and $\mathcal{Y}$ of $F^E$ are \textit{orthogonal}, and write $\mathcal{X} \perp \mathcal{Y}$, if $X\perp Y$ for all $X\in \mathcal{X}$ and $Y\in\mathcal{Y}$. 

The circuits of $M^*$ are called the \textit{cocircuits of $M$}, and we write $\mathcal{C}^*(M)=\mathcal{C}(M^*)$. Circuits and cocircuits are orthogonal: $\mathcal{C}(M) \perp \mathcal{C}^*(M)$.

We denote by $\cX^\perp=\{X\in F^E\mid X\perp\cX\}$ the orthogonal complement of a subset $\cX$ of $F^E$. The set $\mathcal{V}(M)=\mathcal{C}^*(M)^{\perp}$ is called the set of \textit{$F$-vectors of $M$}, and $\mathcal{V^*}(M)=\mathcal{C}(M)^{\perp}$ the set of \textit{$F$-covectors of $M$}. There is a cryptomorphic description of $F$-matroids in terms of their vectors, as explained in \cite{Anderson19}. Note that, as $\mathcal{C}(M)\subseteq \mathcal{V}(M)$, we have $\mathcal{V}(M)^{\perp} \subseteq \mathcal{V}^*(M)$. 

\begin{df}
 A tract $F$ is \textit{perfect} if $\mathcal{V}(M) \perp \mathcal{V}^*(M)$ for every $F$-matroid $M$.
\end{df}

Note that all fields and partial fields are perfect, and so are the hyperfields $\K$, $\S$ and $\T$.

\subsection{Dual pairs}
Let $N$ be a (classical) matroid on $E$. We call a subset $\mathcal{C}$ of $F^E$ an \textit{$F$-signature} of $N$ if $\supp(\mathcal{C})$ is the set of circuits of $N$ and $\mathcal{C}$ satisfies \ref{C0} -- \ref{C2}.

\begin{df}
 Let $\mathcal{C}$ and $\mathcal{D}$ be subsets of $F^E$. We say that $(\mathcal{C}, \mathcal{D})$ is a \textit{dual pair of $F$-signatures of $N$} if:

\begin{enumerate}[start=1, label = (DP\arabic*)]
\item $\mathcal{C}$ is an $F$-signature of $N$;
\item $\mathcal{D}$ is an $F$-signature of $N^*$;
\item $\mathcal{C}\perp \mathcal{D}$.
\end{enumerate}
\end{df}

\begin{thm}[{\cite[Thm. 3.26]{Baker-Bowler19}}]
\label{thm-dual pairs}
There is a bijection between $F$-matroids $M$ with underlying matroid $\underline{M} = N$ and dual pairs of $F$-signatures of $N$, given by
\[
M\longmapsto (\mathcal{C}(M), \mathcal{C}^*(M)).
\]
\end{thm}

\subsection{Minors}
Let $\varphi:E^r\to F$ be a Grassmann–Plücker function with associated matroid $M=M_\varphi$ and $A\subset E$.
\begin{enumerate}[label = (\arabic*)]
    \item (Contraction) Let $\ell$ be the rank of $A$ in $\underline{M}_\varphi$, and let $\{a_1, a_2,\dotsc, a_\ell \}$ be a maximal $\underline{M}_\varphi$-independent subset of $A$. Define $\varphi/A: (E\backslash A)^{r-\ell} \rightarrow F$ by
    \[
     (\varphi/A)(x_1, \dotsc, x_{r-\ell}) := \varphi(x_1, \dotsc, x_{r-\ell}, a_1, \dotsc, a_\ell).
    \]
    The \emph{contraction of $M$ by $A$} is $M/A=M_{\varphi/A}$.
    \item (Deletion) Let $k$ be the rank of $E\backslash A$ in $\underline{M}_\varphi$, and choose $ a_1,\dotsc, a_{r-k} \in A$ such that $\{a_1,\dotsc, a_{r-k}\}$ is a basis of $\underline{M}_\varphi / (E\backslash A)$. Define $\varphi \backslash A: (E\backslash A)^k \rightarrow F$ by
    \[
     (\varphi\backslash A)(x_1, \dotsc, x_k) := \varphi(x_1, \dotsc, x_k, a_1, \dotsc, a_{r-k}).
    \]
    The \emph{deletion of $A$ from $M$} is $M\backslash A=M_{\varphi\backslash A}$.
\end{enumerate}
Contractions and deletions are well-defined due to the following fact.

\begin{lemma}[{\cite[Lemma 4.4]{Baker-Bowler19}}]\label{lemma-minors}%How to remove the italics here? %ol: do you mean the italics in (1), (2), (3)? Simply uncomment the %[label = ...]
\
\begin{enumerate}%[label = \textup{(\arabic*)}]  
    \item Both $\varphi/A$ and $\varphi \backslash A$ are Grassmann–Plücker functions. Their definitions are independent of all choices up to global multiplication by an element of $F^\times$.
    
    \item $\underline{M}_{\varphi/A} = \underline{M}_\varphi / A$ and $\underline{M}_{\varphi \backslash A} = \underline{M}_\varphi \backslash A$.
    
    \item $(\varphi \backslash A)^* = \varphi^* / A$ and $(\varphi / A)^* = \varphi^* \backslash A$.
\end{enumerate}
\end{lemma}

\subsection{Examples}

We have mentioned already that usual matroids reappear as $\K$-matroids in Baker-Bowler theory. Other examples are:
\begin{itemize}
 \item Let $K$ be a field. There is a bijection from $K$-matroids of rank $r$ on $E$ to $r$-dimensional $K$-subspaces of $K^E$, given by $M\mapsto \mathcal{V}^*(M)$; see \cite[Prop.\ 2.19]{Anderson19}).
    \item A valuated matroid in the sense of \cite{Dress-Wenzel92} is the same thing as a $\T$-matroid.
    \item There is a bijection from $\S$-matroids to oriented matroids in the sense of \cite{Bland-LasVergnas78}, given by $M\mapsto \big(E, \mathcal{C}(M)\big)$.
\end{itemize}

%%%%%%%%%%%%%%%%%%%%%%%%%%%%%%%%%%%%%%%%%%%%%%%%%%%%%%%%%%%%%%%%%%%%%%%%%%%%%%%%%%%%%%%%%%%%%%%%%%%%%%%%%%%%%%%%%%%%%%%%%%%%%%%%%%%%%%%%%%%%%%%%%%%%%%%%%%%%%%%%%%%%%%%%%%%
%%%%%%%%%%%%%%%%%%%%%%%%%%%%%%%%%%%%%%%%%%%%%%%%%%%%%%%%%%%%%%%%%%%%%%%%%%%%%%%%%%%%%%%%%%%%%%%%%%%%%%%%%%%%%%%%%%%%%%%%%%%%%%%%%%%%%%%%%%%%%%%%%%%%%%%%%%%%%%%%%%%%%%%%%%%
\section{Flag matroids}
\label{section: flag matroids}
%%%%%%%%%%%%%%%%%%%%%%%%%%%%%%%%%%%%%%%%%%%%%%%%%%%%%%%%%%%%%%%%%%%%%%%%%%%%%%%%%%%%%%%%%%%%%%%%%%%%%%%%%%%%%%%%%%%%%%%%%%%%%%%%%%%%%%%%%%%%%%%%%%%%%%%%%%%%%%%%%%%%%%%%%%%
%%%%%%%%%%%%%%%%%%%%%%%%%%%%%%%%%%%%%%%%%%%%%%%%%%%%%%%%%%%%%%%%%%%%%%%%%%%%%%%%%%%%%%%%%%%%%%%%%%%%%%%%%%%%%%%%%%%%%%%%%%%%%%%%%%%%%%%%%%%%%%%%%%%%%%%%%%%%%%%%%%%%%%%%%%%

\subsection{Definitions}
\label{subsection: definition}

Let us introduce the central notion of this text: flag matroids with coefficients in tracts. Throughout the whole section, we fix $E=\{1,\dots,n\}$ and integers $r$ and $w$ between $0$ and $n$.

\begin{df}
 Let $M$ and $N$ be $F$-matroids of respective ranks $r$ and $w$ on $E$. We say that $M$ is a \textit{quotient} of $N$, and write $N \twoheadrightarrow M$, if every choice of Grassmann-Pl\"ucker functions $\mu$ and $\nu$ representing $M$ and $N$, respectively, satisfies the \emph{Pl\"ucker flag relations}
\begin{equation}
\label{flag equation}
 \underset{k = 1}{\overset{w+1}{\sum}}\ \epsilon^k \; \nu(y_1, \dotsc, \widehat{y_k}, \dotsc y_{w+1}) \cdot \mu(y_k, x_1, \dotsc, x_{r-1}) \ \in \ N_F
\end{equation}
for all $y_1, \dotsc, y_{w+1},x_1, \dotsc, x_{r-1}\in E$.

A \textit{flag $F$-matroid} on $E$ is a sequence $\textbf{M} = (M_1, \dotsc, M_s)$ of $F$-matroids such that $M_j \twoheadrightarrow M_i$ for all $1 \leq i < j \leq s$. The sequence $\rk(\bM)=\big( \rk(\underline{M_1}), \dotsc, \rk(\underline{M_s}) \big)$ is called the \textit{rank of $\bM$}.
\end{df}

The identification of classical matroids with $\K$-matroids yields an identification of classical flag matroids with flag $\K$-matroids. The proof of this fact relies, however, on the circuit-vector characterization of flag matroids. We postpone this discussion to \autoref{subsection: flag matroids as flag K-matroids}.

\begin{rem}\label{rem: properties of falg matroids}
 The following are some immediate observations.
\begin{enumerate}
    \item Since two Grassmann-Pl\"ucker functions representing the same $F$-matroid only differ by a non-zero factor, the validity of equation \eqref{flag equation} does not depend on the choice of Grassmann-Pl\"ucker functions.
    \item Note that for $N = M$, the Pl\"ucker flag relations in \eqref{flag equation} are nothing else than the usual Pl\"ucker relations of a Grassmann-Pl\"ucker function, see \ref{GP3}. Thus one always has $M \twoheadrightarrow M$.
    \item If $\bM=(M_1,...,M_s)$ is a flag $F$-matroid, then $(M_1,...,\widehat{M_i},...,M_s)$ is also a flag $F$-matroid, for all $i$ in $[s]$.
\end{enumerate}
\end{rem}

%%%%%%%%%%%%%%%%%%%%%%%%%%%%%%%%%%%%%%%%%%%%%%%%%%%%%%%%%%%%%%%%%%%%%%%%%%%%%%%%%%%%%%%%%%%%%%%%%%%%%%%%%%%%%%%%%%%%%%%%%%%%%%%%%%%%%%%%%%%%%%%%%%%%%%%%%%%%%%%%%%%%%%%%%%%

\subsection{Functoriality}
\label{subsection: functoriality}

As for single matroids, we can change the coefficients of flag matroids along tract morphisms.

\begin{prop}\label{prop: functoriality}
 Let $f:F\to F'$ be a tract morphism and $N\twoheadrightarrow M$ a quotient of $F$-matroids. Then $f_\ast N\twoheadrightarrow f_\ast M$. Consequently, if $\bM=(M_1,\dotsc,M_s)$ is a flag $F$-matroid, then $f_\ast\bM=(f_\ast M_1,\dotsc,f_\ast M_s)$ is a flag $F'$-matroid.
\end{prop}

\begin{proof}
 Since the Pl\"ucker flag relations for $f_\ast N\twoheadrightarrow f_\ast M$ are indexed by the same tuples of elements $x_1,\dotsc,x_{\rk(M)-1},y_1,\dotsc,y_{\rk(N)+1}\in E$ and are of the same shape, the first assertion follows at once from the definition of a morphism of tracts. The second assertion follows at once from the first and the definition of a flag $F$-matroid.
\end{proof}

%%%%%%%%%%%%%%%%%%%%%%%%%%%%%%%%%%%%%%%%%%%%%%%%%%%%%%%%%%%%%%%%%%%%%%%%%%%%%%%%%%%%%%%%%%%%%%%%%%%%%%%%%%%%%%%%%%%%%%%%%%%%%%%%%%%%%%%%%%%%%%%%%%%%%%%%%%%%%%%%%%%%%%%%%%%

\subsection{Cryptomorphism}
\label{subsection: cryptomorphism}

The core result of our theory consists of the following cryptomorphic description of flag $F$-matroids in terms of their cocircuits and covectors.

\begin{thm}[Cryptomorphism for flag $F$-matroids]\label{thm-cryptomorphism}
 Let $(M_1,\dotsc,M_s)$ be a sequence of $F$-matroids on $E$ with respective cocircuit sets $\cC^\ast(M_i)$ and covector sets $\cV^\ast(M_i)$. The following are equivalent:
\begin{enumerate}
    \item $(M_1, \dotsc, M_s)$ is a flag $F$-matroid;
    \item \label{cryptomorphism 2} $\mathcal{C}^*(M_i) \subseteq \mathcal{V}^*(M_j)$ for all $1 \leq i < j \leq s$.
\end{enumerate}
\end{thm}

\begin{proof}
It suffices to show that the following assertions are equivalent for two $F$-matroids $M$ and $N$ on $E$:
\begin{enumerate}
    \item \label{cryp1} $M$ is a quotient of $N$;
    \item \label{cryp2} $\mathcal{C}^*(M) \subseteq \mathcal{V}^*(N)$.
\end{enumerate}
As a first step, we show that \eqref{cryp1} implies \eqref{cryp2}. Assume that $M$ is a quotient of $N$. Fix Grassmann-Pl\"ucker functions $\mu$ and $\nu$ that represent $M$ and $N$, respectively. Let $Z \in \mathcal{C}^*(M) = \mathcal{C}(M^*)$. By \cite[p.\ 841]{Baker-Bowler19}, there exists a $z_0 \in \underline{Z}$, an $\alpha \in F^\times$ and a basis $D = \{z_1, \dotsc, z_{n-r}\}$ of $\underline{M}^*$ containing $\underline{Z} - z_0$ such that
\[
Z(z_i) = \alpha\cdot \epsilon^i \cdot \mu^*(z_0, \dotsc, \widehat{z_i} \dotsc, z_{n-r}).
\]
Let $\{x_1, \dotsc, x_{r-1}\} = E \backslash \{z_0, \dotsc, z_{n-r}\}$. Similarly, for $Y \in \mathcal{C}(N)$, there exist an $y_0 \in \underline{Y}$, a $\lambda \in F^{\times}$ and a basis $B = \{y_1, \dotsc, y_w\}$ of $\underline{N}$ containing $\underline{Y} - y_0$ such that
\[
Y(y_j) = \lambda\cdot \epsilon^j \cdot \nu(y_0, \dotsc, \widehat{y_j} \dotsc,y_w).
\]
Let $S=(D\cup\{z_0\})\cap(B\cup\{y_0\})$ and 
\[
  I \ = \ \big\{i \in \{0,\dotsc,n-r\}\, \big| \, z_i\in S\big\} \qquad \text{and} \qquad J \ = \ \big\{j \in \{0,\dotsc,w\} \,\big|\, y_j\in S\big\}.  
\]
There exists a bijection $\mathfrak{b}: J\rightarrow I$ such that $y_j=z_{\mathfrak{b}(j)}$, for all $j\in J$. Note that
\[
\begin{aligned}
      & \;(\lambda \overline{\alpha})^{-1} \cdot \big(Y \cdot Z\big) = (\lambda \overline{\alpha})^{-1} \cdot \underset{e \in E}{\sum} Y(e)\cdot \overline{Z(e)}
\\
    = & \;(\lambda \overline{\alpha})^{-1} \cdot \underset{e \in S}{\sum} \ Y(e) \cdot \overline{Z(e)}
\\
    = & \;(\lambda \overline{\alpha})^{-1} \cdot \underset{j \in J}{\sum} \big( \lambda\cdot \epsilon^j \nu(y_0, \dotsc, \widehat{y_j}, \dotsc, y_w)\big) \cdot\big(\overline{\alpha \cdot \epsilon^{\mathfrak{b}(j)}\mu^*(z_0, \dotsc, \widehat{z_{\mathfrak{b}(j)}} \dotsc,z_{n-r})}\big)
\\
    = & \;\underset{j \in J}{\sum} \ \epsilon^{\big( j + \mathfrak{b}(j) \big)} \cdot \nu(y_0, \dotsc, \widehat{y_j} \dotsc, y_w) \cdot \sign(z_0, \dotsc, \widehat{z_{\mathfrak{b}(j)}}, \dotsc,z_{n-r}, z_{\mathfrak{b}(j)}, x_1, \dotsc, x_{r-1})
\\    
      & \hspace{2cm} \cdot  \mu( z_{\mathfrak{b}(j)}, x_1, \dotsc, x_{r-1})
\\
    = & \;\underset{j\in J}{\sum} \ \epsilon^{\big( j + \mathfrak{b}(j)\big)} \cdot \sign(z_0, \dotsc, \widehat{z_{\mathfrak{b}(j)}}, \dotsc,z_{n-r}, z_{\mathfrak{b}(j)}, x_1, \dotsc, x_{r-1}) \cdot \nu(y_0, \dotsc, \widehat{y_j} \dotsc, y_w) 
\\
      & \hspace{2cm} \cdot \mu( y_j, x_1, \dotsc, x_{r-1}),
\end{aligned}
\]
because $\mu( y_j, x_1, \dotsc, x_r) = 0$ for $j \in \{0, \dotsc, s\}\backslash J$, since in this case $y_j \in \{x_1, \dotsc, x_{r-1}\}$. Further we have for all $j\in J$ that 
\[
\begin{aligned}
      & \; \epsilon^{(n-r)} \cdot \sign(z_0, \dotsc,z_{n-r}, x_1, \dotsc, x_{r-1})
\\
   = & \; \epsilon^{2(n-r) - \mathfrak{b}(j)} \cdot \sign(z_0, \dotsc, \widehat{z_{\mathfrak{b}(j)}}, \dotsc,z_{n-r}, z_{\mathfrak{b}(j)}, x_1, \dotsc, x_{r-1})
\\
    = & \; \epsilon^{\mathfrak{b}(j)} \cdot \sign(z_0, \dotsc, \widehat{z_{\mathfrak{b}(j)}}, \dotsc,z_{n-r}, z_{\mathfrak{b}(j)}, x_1, \dotsc, x_{r-1}).
\end{aligned}
\]
Thus
\[
\begin{aligned}
      & \;\epsilon^{ (n-r)} \cdot \sign(z_0, \dotsc,z_{n-r}, x_1, \dotsc, x_{r-1}) \cdot (\lambda \overline{\alpha})^{-1} \cdot \big( Y \cdot Z \big)
\\
    = & \;\underset{j \in J}{\sum} \ \epsilon^j \cdot \nu(y_0, \dotsc, \widehat{y_j} \dotsc, y_w) \cdot \mu( y_j, x_1, \dotsc, x_{r-1})
\\
    = & \;\underset{j=0}{\overset{w}{\sum}} \ \epsilon^j \cdot \nu(y_0, \dotsc, \widehat{y_j} \dotsc, y_w) \cdot \mu( y_j, x_1, \dotsc, x_{r-1}).
\end{aligned}
\]
This implies that $Y\cdot Z \in N_F$. Therefore $Z \in \mathcal{V}^*(N)$, which shows that \eqref{cryp1} implies \eqref{cryp2}.

Next we show that \eqref{cryp2} implies \eqref{cryp1}. Assume that $\mathcal{C}^*(M)$ is a subset of $\mathcal{V}^*(N)$ and let $\{y_0, \dotsc, y_w\}$ be an $(w+1)$-subset and $\{x_1, \dotsc, x_{r-1}\}$ be an $(r-1)$-subset of $E$.

\medskip\noindent
\textbf{Case 1.} If there is no $\ell$ in $\{0, \dotsc, w\}$ such that $\{y_\ell, x_1,  \dotsc, x_{r-1}\}$ is a basis of $\underline{M}$ and $\{y_0, \dotsc, \widehat{y_\ell},$ $\dotsc, y_w\}$ is a basis of $\underline{N}$, one has that $\nu(y_0, \dotsc, \widehat{y_j}, \dotsc, y_w) \cdot \mu(y_j, x_1, \dotsc, x_{r-1}) = 0$ for all $j \in \{0, \dotsc, w\}$. Thus 
\[
\underset{k = 0}{\overset{w}{\sum}} \epsilon^k \cdot 
\nu( y_0, \dotsc, \widehat{y_k}, \dotsc, y_w) \cdot \mu(y_k, x_1, \dotsc, x_{r-1}) \ = \ 0 \ \in \ N_F.
\]

\medskip\noindent
\textbf{Case 2.} If there is an $\ell$ in $\{0, \dotsc, w\}$ such that $\{y_\ell, x_1,  \dotsc, x_{r-1}\}$ is a basis of $\underline{M}$ and $\{y_0, \dotsc, \widehat{y_\ell}, \dotsc, y_w\}$ is a basis of $\underline{N}$, then $\{z_1, \dotsc, z_{n-r}\} = E-\{y_\ell,x_1,  \dotsc, x_{r-1}\}$ is a basis of $\underline{M^*}$. Define $z_0 := y_\ell$. Then
\begin{equation*}
    H(z) \ = \ 
    \begin{cases}
      \epsilon^i \cdot \mu^*( z_0, \dotsc, \widehat{z_i}, \dotsc, z_{n-r} ) & \text{if}\ z=z_i \text{ for some}\ i\in\{0, \dotsc, n-r\}, \\
      0 & \text{otherwise,}
    \end{cases}\
\end{equation*}
defines a circuit of $M^*$ and 
\begin{equation*}
    G(y) \ = \ 
    \begin{cases}
      \epsilon^j \cdot \nu(y_0, \dotsc, \widehat{y_j}, \dotsc, y_w) & \text{if}\ y=y_j \text{ for some}\ j\in\{0, \dotsc, w\},\\
      0 & \text{otherwise,}
    \end{cases}
\end{equation*}
defines a circuit of $N$.

Note that 
\[
H(z_i) = \epsilon^{i} \cdot \sign(z_0, \dotsc, \widehat{z_i}, \dotsc, z_{n-r}, z_i, x_1, \dotsc, x_{r-1}) \cdot \overline{\mu(z_i, x_1, \dotsc, x_{r-1})},
\]
for all $i \in \{0, \dotsc, n-r\}$. Let $S=\{z_0, \dotsc, z_{n-r}\}\cap\{y_0, \dotsc, y_w\}$ and 
\[
  I \ = \ \big\{i \in \{0,\dotsc,n-r\}\, \big| \, z_i\in S\big\} \qquad \text{and} \qquad J \ = \ \big\{j \in \{0,\dotsc,w\} \,\big|\, y_j\in S\big\}.  
\]
There exists a bijection $\mathfrak{b}: J\rightarrow I$ such that $y_j=z_{\mathfrak{b}(j)}$, for all $j\in J$. Since $G \cdot H$ is in $N_F$, we have 
\[
\begin{aligned}
        &\ \underset{k = 0}{\overset{w}{\sum}} \ \epsilon^k \cdot \nu(y_0, \dotsc, \widehat{y_k}, \dotsc, y_w) \cdot \overline{\mu(y_k, x_1, \dotsc, x_{r-1})} \\
     =  &\ \underset{j\in J}{\sum} \ \epsilon^j \cdot \nu(y_0, \dotsc, \widehat{y_j}, \dotsc, y_w) \cdot \overline{\mu(z_{\mathfrak{b}(j)}, x_1, \dotsc, x_{r-1})} \\
     =  &\ \epsilon^{(n-r)} \cdot \sign(z_0, \dotsc,z_{n-r}, x_1, \dotsc, x_{r-1}) \cdot \underset{j\in J}{\sum} G(y_j) \cdot \overline{H(z_{\mathfrak{b}(j)})} \\
     =  &\ \epsilon^{(n-r)} \cdot \sign(z_0, \dotsc,z_{n-r}, x_1, \dotsc, x_{r-1}) \cdot \underset{e \in E}{\sum} G(e) \cdot \overline{H(e)} \\
     =  &\ \epsilon^{(n-r)} \cdot \sign(z_0, \dotsc,z_{n-r}, x_1, \dotsc, x_{r-1}) \cdot G \cdot H \ \in \ N_F,
\end{aligned}
\]
which concludes the proof of the theorem.  
\end{proof}

\begin{rem}
 \autoref{thm-cryptomorphism} shows that our definition of quotients extends the concept of quotients of oriented matroids; see \cite[Def. 7.7.2]{Bjorner-LasVergnas-Sturmfels-White-Ziegler99}.

 The characterization \eqref{cryptomorphism 2} of flag matroids can be seen, in fact, as an expansion of the concept of dual pairs of $F$-signatures since $\cC^\ast(M_i)\subset\cV^\ast(M_j)$ if and only if $\mathcal{C}^*(M_i) \perp \mathcal{C}(M_j)$. This latter form of condition \eqref{cryptomorphism 2} exhibits at once the duality property of flag matroids; see \autoref{rem-cryptomorphism}.
\end{rem}

\begin{cor}
 Let $F$ be a tract and $M$ and $N$ two $F$-matroids such that $N\twoheadrightarrow M$. Then $\rk({N})=\rk({M})$ implies $N=M$.
\end{cor}

\begin{proof}
One has $\underline{N}\twoheadrightarrow \underline{M}$ by \autoref{prop: common flag matroids} and $\mathcal{C}^*(M) \subseteq \mathcal{V}^*(N)$ by \autoref{thm-cryptomorphism}. As $\rk(\underline{N})=\rk(\underline{M})$, one also has $\underline{N}=\underline{M}$ (see \cite[Prop.\ 8.1.6 and Lemma 8.1.7]{Kung86}). As the (co)circuit set characterizes the $F$-matroid (see \cite[Thm. 3.17]{Baker-Bowler19}), everything follows if we can show that $\mathcal{C}^*(M) = \mathcal{C}^*(N)$.

Let $Z\in \mathcal{C}^*(M)$. Then $Z \in \mathcal{V}^*(N)$ and $\underline{Z} \in \mathcal{C}^*(\underline{M}) = \mathcal{C}^*(\underline{N}) = \textup{supp}\big(\mathcal{C}^*(N)\big)$. As $\mathcal{C}^*(N)$ is equal to the set of nonzero covectors of minimal support by \cite[Thm. 2.18]{Anderson19}, one has that $Z\in \mathcal{C}^*(N)$. Therefore $\mathcal{C}^*(M)\subseteq \mathcal{C}^*(N)$.

Let $X\in \mathcal{C}^*(N)$. Then there exists an $Y\in \mathcal{C}^*(M)$ such that $\underline{X}=\underline{Y}$. As $\mathcal{C}^*(N)$ satisfies \ref{C2}, there is an $\alpha\in F^\times$ such that $X=\alpha Y$. As $\mathcal{C}^*(M)$ satisfies \ref{C1}, $X\in \mathcal{C}^*(M)$. This shows that $\mathcal{C}^*(M) = \mathcal{C}^*(N)$, which implies $M=N$.
\end{proof}

\begin{rem}[Exterior algebra description of flag matroids]\label{rem: exterior algebras}
 The identification of $F$-matroids with classes of exterior $F$-algebras from the first author's paper \cite{Jarra22} leads to yet another description of flag matroids. To explain, the exterior algebra $\Lambda F^E$ is an $F$-module that generalizes exterior algebras of vector spaces and the Giansiracusa exterior algebra from \cite{Giansiracusa-Giansiracusa18} to all idylls $F$, which are tracts with additively closed nullset $N_F$; we refer the reader to \cite{Jarra22} for details.
 
 A Grassmann-Pl\"ucker function $\mu:E^r\to F$ determines an element $\nu\in\Lambda^r F^E$ with coordinates $\nu_\bI=\mu(\bI)$ for $\bI\in E^r$. This association yields a bijection between $F$-matroids of rank $r$ on $E$ and $F^\times$-classes $[\nu]$ of elements $\nu\in\Lambda^r F$ that satisfy the Pl\"ucker relations.
 
 Thus a flag $F$-matroid of rank $(r_1,\dotsc,r_s)$ on $E$ corresponds to a tuple $\big([\nu_1],\dotsc,[\nu_s]\big)$ of $F^\times$-classes $[\nu_i]$ of elements $\nu_i\in\Lambda^{r_i}F^E$ that satisfy the Pl\"ucker flag relations
 \[
  0 \ \leq \ \sum_{k=1}^{r_j+1} \epsilon^k \; \nu_{i,(y_k,x_1,\dotsc,x_{r_i-1})} \nu_{j,(y_1,\dotsc,\widehat{y_k},\dotsc,y_{r_j+1})}
 \]
 in $\Lambda F^E$ for all $1\leq i\leq j\leq n$ and $x_1,\dotsc,x_{r_i-1},y_1,\dotsc,y_{r_j+1}\in E$.
\end{rem}

%%%%%%%%%%%%%%%%%%%%%%%%%%%%%%%%%%%%%%%%%%%%%%%%%%%%%%%%%%%%%%%%%%%%%%%%%%%%%%%%%%%%%%%%%%%%%%%%%%%%%%%%%%%%%%%%%%%%%%%%%%%%%%%%%%%%%%%%%%%%%%%%%%%%%%%%%%%%%%%%%%%%%%%%%%%

\subsection{Flag matroids as flag \texorpdfstring{$\K$}{K}-matroids}
\label{subsection: flag matroids as flag K-matroids}

The realization of matroids as $\K$-matroids extends to flag matroids as explained in the following.

Let us recall the notion of a flag matroid from \cite{Borovik-Gelfand-Vince-White97}. Given two matroids $M$ and $N$ on the same ground set $E$, we say that \emph{$M$ is a quotient of $N$} and write $N\twoheadrightarrow M$ if the identity on $E$ is a strong map from $N$ to $M$, i.e., every flat of $M$ is a flat of $N$ or, equivalently, every cocircuit of $M$ is a union of cocircuits of $N$; see \cite[Prop.\ 8.1.6]{Kung86} for details. A \emph{flag matroid} is a sequence $(M_1,\dotsc,M_s)$ of matroids such that $M_i$ is a quotient of $M_{i+1}$ for $i=1,\dotsc,s-1$.

\begin{prop}[Classical flag matroids] \label{prop: common flag matroids}
Let $M$ and $N$ be $\K$-matroids on $E$. Then $M$ is a quotient of $N$ if and only if $\underline{M}$ is a quotient of $\underline{N}$. In consequence, a sequence $(M_1,\dotsc,M_s)$ of $\K$-matroids is a flag $\K$-matroid if and only if $(\underline{M_1},\dotsc,\underline{M_s})$ is a flag matroid.
\end{prop}

\begin{proof}
 Baker-Bowler theory provides a bijection between $\cC^\ast(N)$ and the cocircuit set of $\underline N$, which sends a cocircuit $C:E\to\K$ of $N$ to its support $\underline C$. By \cite[Prop. 5.2]{Anderson19}, this association extends to a bijection between $\cV^\ast(N)$ and unions of cocircuits of $N$. Therefore $\underline M$ is a quotient of $\underline N$, i.e., every cocircuit of $M$ is a union of cocircuits of $N$, if and only if $\cC^\ast(M)\subset\cV^\ast(N)$. By \autoref{thm-cryptomorphism}, the latter property is equivalent to $M$ being a quotient of $N$, which establishes the first claim of the proposition.
 
 The second claim follows from the analogous definitions of flag $\K$-matroids and flag matroids, taking into account that strong maps of classical matroids are composable.
\end{proof}

\begin{rem}\label{rem: crypto with base exchange}
 An alternative proof of \autoref{prop: common flag matroids} is as follows. It is known that $\underline N$ is a quotient of $\underline M$ if and only if for every basis $B_N$ of $\underline N$, for every basis $B_M$ of $\underline M$ and for every $e\in B_N-B_M$ there is an $f\in B_M-B_N$ such that $B_N-e+f$ is a basis of $\underline N$ and $B_M-f+e$ is a basis of $\underline M$; cf.\ \cite{mathoverflow302574} as well as \cite{Tardos85,Bouchet87,Bouchet89}. This latter condition is directly equivalent to the Pl\"ucker flag relations for $M\twoheadrightarrow N$.
\end{rem}

%\footnote{Given $1\leq i\leq j\leq s$ and bases $B_i$ of $M_i$ and $B_j$ of $M_j$, the characterization of $M_i\to M_j$

\begin{cor}
Let $F$ be a tract and $\bM = (M_1, \dotsc, M_s)$ a flag $F$-matroid. Then $\rk(\bM)$ is a non-decreasing sequence of natural numbers.
\end{cor}

\begin{proof}
 Since the rank of an $F$-matroid $M$ is equal to the rank of its underlying matroid $\underline M$, we need only to show that $N\twoheadrightarrow M$ implies $\rk(\underline{N})\geq \rk(\underline{M})$. Let $t_F:F\to\K$ be the terminal map. By \autoref{prop: functoriality}, we have $t_{F,\ast} N\twoheadrightarrow t_{F,\ast} M$, and by \autoref{prop: common flag matroids}, we have $\underline{N}=\underline{t_{F,\ast} N}\twoheadrightarrow \underline{t_{F,\ast} M}=\underline{M}$. By \cite[Lemma 8.1.7]{Kung86}, we have $\rk(\underline{N})\geq \rk(\underline{M})$, as desired.
\end{proof}

%%%%%%%%%%%%%%%%%%%%%%%%%%%%%%%%%%%%%%%%%%%%%%%%%%%%%%%%%%%%%%%%%%%%%%%%%%%%%%%%%%%%%%%%%%%%%%%%%%%%%%%%%%%%%%%%%%%%%%%%%%%%%%%%%%%%%%%%%%%%%%%%%%%%%%%%%%%%%%%%%%%%%%%%%%%

\subsection{Duality}
\label{subsection: duality}

Thanks to the cryptomorphism from \autoref{thm-cryptomorphism}, many standard properties of matroids fall into their places, the first one being duality.

\begin{prop}
\label{prop-duality}
Let $M$ and $N$ be $F$-matroids on $E$. Then $N \twoheadrightarrow M$ is equivalent to $M^* \twoheadrightarrow N^*$.
\end{prop}

\begin{proof}
By the symmetry of the affirmation, it is enough to prove only one implication. If $N \twoheadrightarrow M$, by \autoref{thm-cryptomorphism}, one has $\mathcal{C}(M^*)=\mathcal{C}^*(M)\subseteq \mathcal{V}^*(N)$. Thus
\[
 \mathcal{C}^*(N^*) = \mathcal{C}(N) \subseteq \mathcal{V}^*(N)^{\perp} \subseteq \mathcal{C}(M^*)^{\perp} = \mathcal{V}^*(M^*). 
\]
Again by \autoref{thm-cryptomorphism}, we conclude that $M^* \twoheadrightarrow N^*$.
\end{proof}

\begin{rem}
\label{rem-cryptomorphism}
Putting \autoref{thm-cryptomorphism} and \autoref{prop-duality} together, one has that the following are equivalent:
\[
 \begin{aligned}
   & (1)\ N\twoheadrightarrow M; 
  && (3)\ \mathcal{C}(N)\subset\mathcal{V}(M); 
  && (5)\ \mathcal{C}(N)\perp\mathcal{C}^\ast(M); \\
   & (2)\ M^\ast\twoheadrightarrow N^\ast; \qquad
  && (4)\ \mathcal{C}^*(M)\subset\mathcal{V}^*(N); \qquad 
  && (6)\ \mathcal{C}^\ast(M)\perp\mathcal{C}(N). 
 \end{aligned}
\]
\end{rem}

\begin{thm}[Duality for flag matroids] \label{thm: duality}
Let $E$ be a set with $n$ elements and let $0\leq r_1 \leq \dotsc \leq r_s \leq n$ be integers. The association \[
\bM = (M_1, \dotsc, M_s) \longmapsto \bM^* := (M_s^*, \dotsc, M_1^*) 
\]
is a bijection between the flag $F$-matroids of rank $(r_1, \dotsc, r_s)$ and the flag $F$-matroids of rank $(n-r_s, \dotsc, n-r_1)$.
\end{thm}

\begin{proof}
Let $\bM = (M_1, \dotsc, M_s)$ be a flag $F$-matroid of rank $(r_1, \dotsc, r_s)$. As $M_j\twoheadrightarrow M_i$, one has $M_i^*\twoheadrightarrow M_j^*$ for all $s\geq j > i \geq 1$, which means that $\textbf{M}^* := (M_s^*, \dotsc, M_1^*)$ is a flag $F$-matroid. Note that 
\[
 \rk(\bM^*) = \big(\rk(\underline{M_s^*}), \dotsc, \rk(\underline{M_1^*})\big) = (n-r_s, \dotsc, n-r_1), 
\]
because $\underline{M^*} = \underline{M}^*$ by \cite[Theorem 3.24]{Baker-Bowler19}.

Again by \cite[Theorem 3.24]{Baker-Bowler19}, one has $(M^*)^* = M$ for all $F$-matroids $M$, which implies that $(\bM^*)^* = \bM$ for all flag $F$-matroids. This finishes the proof.
\end{proof}

%%%%%%%%%%%%%%%%%%%%%%%%%%%%%%%%%%%%%%%%%%%%%%%%%%%%%%%%%%%%%%%%%%%%%%%%%%%%%%%%%%%%%%%%%%%%%%%%%%%%%%%%%%%%%%%%%%%%%%%%%%%%%%%%%%%%%%%%%%%%%%%%%%%%%%%%%%%%%%%%%%%%%%%%%%%

\subsection{Minors}
\label{subsection: minors}

Minors of flag matroids are defined by taking minors of the components of the flag. This leads to a meaningful notion of minors due to the following fact.

\begin{prop}
\label{prop-minor}
Let $M$ and $N$ be $F$-matroids on $E$ such that $N\twoheadrightarrow M$, and let $e$ be an element of $E$. Then $N/e \twoheadrightarrow M/e$ and $N\backslash e \twoheadrightarrow M\backslash e$.
\end{prop}

\begin{proof}
Let $\mu$ and $\nu$ be Grassmann-Pl\"ucker functions that represent $M$ and $N$, respectively, and let $r$ and $w$ be their respective ranks. We begin with showing that $N/e \twoheadrightarrow M/e$. By \autoref{lemma-minors}, $\mu/e$ and $\nu/e$ represent $M/e$ and $N/e$, respectively.

\medskip\noindent
\textbf{Case 1:} Assume that $e$ is not a loop of $\underline{M}$. By \cite[Lemma 1]{Recski05}, $e$ is not a loop of $\underline{N}$. Thus $\rk(M/e) = r-1$ and $\rk(N/e) = w-1$. Let $\{y_1, \dotsc, y_w\}$ and $\{x_1, \dotsc, x_{r-2}\}$ be subsets of $E-e$. Note that $\mu(e, x_1, \dotsc, x_{r-2}, e) = 0$. Thus
\[
\begin{aligned}
      & \; \underset{k=1}{\overset{w}{\sum}} \epsilon^k \cdot (\nu/e)(y_1, \dotsc, \widehat{y_k}, \dotsc, y_w)\cdot (\mu/e)(y_k, x_1, \dotsc, x_{r-2})
\\
    = & \; \underset{k=1}{\overset{w}{\sum}} \epsilon^k \cdot \nu(y_1, \dotsc, \widehat{y_k}, \dotsc, y_w, e) \cdot \mu(y_k, x_1, \dotsc, x_{r-2}, e) + 0 
\\
    = & \; \underset{k=1}{\overset{w}{\sum}} \epsilon^k \cdot \nu(y_1, \dotsc, \widehat{y_k}, \dotsc, y_w, e) \cdot \mu(y_k, x_1, \dotsc, x_{r-2}, e)\\
      & \hspace{3cm} + \epsilon^{w+1} \cdot \nu(y_1, \dotsc, y_w) \cdot \mu(e, x_1, \dotsc, x_{r-2}, e) \ \in \ N_F.
\end{aligned}
\]

\medskip\noindent
\textbf{Case 2:} Assume that $e$ is a loop of both $\underline{M}$ and $\underline{N}$. We have $\rk(M/e) = r$ and $\rk(N/e) = w$. Let $\{y_1, \dotsc, y_{w+1}\}$ and $\{x_1, \dotsc, x_{r-1}\}$ be subsets of $E-e$. Thus
\[
\begin{aligned}
      & \; \underset{k=1}{\overset{w+1}{\sum}} \epsilon^k \cdot (\nu/e)(y_1, \dotsc, \widehat{y_k}, \dotsc, y_{w+1})\cdot (\mu/e)(y_k, x_1, \dotsc, x_{r-1})
\\
    = & \; \underset{k=1}{\overset{w+1}{\sum}} \epsilon^k \cdot \nu(y_1, \dotsc, \widehat{y_k}, \dotsc, y_{w+1}) \cdot \mu(y_k, x_1, \dotsc, x_{r-1}) \ \in \ N_F.
\end{aligned}
\]

\medskip\noindent
\textbf{Case 3:} Assume that $e$ is a loop of $\underline{M}$ but not a loop of $\underline{N}$. We have $\rk(M/e) = r$ and $\rk(N/e) = w-1$. Let $\{y_1, \dotsc, y_w\}$ and $\{x_1, \dotsc, x_{r-1}\}$ be subsets of $E-e$. Note that $\mu(e, x_1, \dotsc, x_{r-1}) = 0$. Thus
\[
\begin{aligned}
      & \ \underset{k=1}{\overset{w}{\sum}} \epsilon^k \cdot (\nu/e)(y_1, \dotsc, \widehat{y_k}, \dotsc, y_w)\cdot (\mu/e)(y_k, x_1, \dotsc, x_{r-1})
\\
    = & \ \underset{k=1}{\overset{w}{\sum}} \epsilon^k \cdot \nu(y_1, \dotsc, \widehat{y_k}, \dotsc, y_w, e) \cdot \mu(y_k, x_1, \dotsc, x_{r-1}) \quad + \quad 0
\\
    = & \ \underset{k=1}{\overset{w}{\sum}} \epsilon^k \cdot \nu(y_1, \dotsc, \widehat{y_k}, \dotsc, y_w, e) \cdot \mu(y_k, x_1, \dotsc, x_{r-1})
\\
      & \hspace{3cm} + \ \epsilon^{w+1} \cdot \nu(y_1, \dotsc, y_w) \cdot \mu(e, x_1, \dotsc, x_{r-1}) \ \in \ N_F.
\end{aligned}
\]
This shows that $N/e \twoheadrightarrow M/e$ in all cases.

Next we show that $N\backslash e \twoheadrightarrow M\backslash e$. We have $M^*\twoheadrightarrow N^*$, by \autoref{prop-duality}. Thus $M^*/e \twoheadrightarrow N^*/e$, by what was proved above. This implies that $(N^*/e)^* \twoheadrightarrow (M^*/e)^*$, by \autoref{prop-duality}. By \autoref{lemma-minors}, we have $N\backslash e = (N^*/e)^* \twoheadrightarrow (M^*/e)^* = M\backslash e$.
\end{proof}

\begin{thm}[Minors of flag matroids] \label{thm: minors}
Let $\bM = (M_1, \dotsc, M_s)$ be a flag $F$-matroid on $E$ and $I$, $J$ disjoint subsets of $E$. Then $\bM\backslash I/J := (M_1\backslash I/J, \dots,M_s\backslash I/J)$ is a flag $F$-matroid on $E\backslash (I\cup J)$.
\end{thm}

\begin{proof}
By \autoref{thm-cryptomorphism}, we only need to show that if  $N \twoheadrightarrow M$, then $N\backslash I/J \twoheadrightarrow M\backslash I/J$. A repeated application of \autoref{prop-minor} to the elements of $I$ shows that $N\backslash I \twoheadrightarrow M\backslash I$, and a similar argument for the elements of $J$ proves that $(N\backslash I)/J \twoheadrightarrow (M\backslash I)/J$.
\end{proof}

\begin{rem}
Even if $\rk(\bM)$ is strictly increasing, $\rk(\bM\backslash I/J)$ might not be strictly increasing. For example, let $M$ and $N$ be the $\K$-matroids on $E = \{1, 2\}$ whose circuit sets are $\mathcal{C}(M) = \{(0,1)\}$ and $\mathcal{C}(N) = \emptyset$, respectively (i.e., $\underline{M} = U_{1,1}\oplus U_{0,1}$ and $\underline{N} = U_{2,2}$). Then $\textbf{M} = (M, N)$ is a flag matroid of rank $(1, 2)$, but $\textbf{M}\backslash\{2\}$ has rank $(1,1)$.
\end{rem}

%%%%%%%%%%%%%%%%%%%%%%%%%%%%%%%%%%%%%%%%%%%%%%%%%%%%%%%%%%%%%%%%%%%%%%%%%%%%%%%%%%%%%%%%%%%%%%%%%%%%%%%%%%%%%%%%%%%%%%%%%%%%%%%%%%%%%%%%%%%%%%%%%%%%%%%%%%%%%%%%%%%%%%%%%%%

\subsection{Flag matroids over perfect tracts}
\label{subsection: perfect tracts}

Flag matroids behave particularly well for perfect tracts in a way that carries over the intuition of flags of linear subspaces over a field. 

\begin{thm}\label{thm: cryptomorphisms for perfect tracts}
 Let $F$ be a tract and $M_1, \dotsc, M_s$ $F$-matroids. Consider the following properties:
 \begin{enumerate}
    \item\label{perfect1} The covector sets form a chain $\mathcal{V}^*(M_1) \subseteq \dotsc \subseteq \mathcal{V}^*(M_s)$;
    \item\label{perfect2} $(M_1, \dotsc, M_s)$ is a flag $F$-matroid (i.e., $M_j\twoheadrightarrow M_i$ for all $1\leq i<j\leq s$);
    \item\label{perfect3} $M_i$ is a quotient of $M_{i+1}$ for all $1 \leq i \leq s-1$.
 \end{enumerate}
 Then the implications \eqref{perfect1}$\Rightarrow$\eqref{perfect2}$\Rightarrow$\eqref{perfect3} hold in general, and \eqref{perfect3}$\Rightarrow$\eqref{perfect1} holds if $F$ is perfect. 
\end{thm}

\begin{proof}
 It is evident that \eqref{perfect2} implies \eqref{perfect3}. Given \eqref{perfect1}, one has $\mathcal{C}^*(M_i) \subseteq \mathcal{V}^*(M_i) \subseteq \mathcal{V}^*(M_{j})$ for $i<j$. Thus $M_j \twoheadrightarrow M_i$ by \autoref{thm-cryptomorphism}, which implies \eqref{perfect2}.
 
 Assume \eqref{perfect3}, i.e., $M_{i+1}\twoheadrightarrow M_i$. By \autoref{prop-duality}, we have $M_i^* \twoheadrightarrow M_{i+1}^*$. Thus $\mathcal{C}^*(M_{i+1}^*) \subseteq \mathcal{V}^*(M_i^*)$ by \autoref{thm-cryptomorphism}. If $F$ is perfect, then
 \[
  \mathcal{V}^*(M_i) \ = \ \mathcal{V}^*(M_i^*)^{\perp} \ \subseteq \ \mathcal{C}^*(M_{i+1}^*)^{\perp} \ = \ \mathcal{V}(M_{i+1}^*) \ = \ \mathcal{V}^*(M_{i+1}),
 \]
 which shows \eqref{perfect1}.
\end{proof}
 
\begin{rem}
For a perfect tract $F$, note that the equivalence of conditions \eqref{perfect2} and \eqref{perfect3} in \autoref{thm: cryptomorphisms for perfect tracts} implies that quotients of $F$-matroids are composable, i.e., if $M_3\twoheadrightarrow M_2$ and $M_2\twoheadrightarrow M_1$, then also $M_3\twoheadrightarrow M_1$.This fails to be true in general for non-perfect tracts, as the following example shows. 
\end{rem}

\begin{ex} 
We exhibit $\P$-matroids $M_1$, $M_2$ and $M_3$, where $\P$ is the phase hyperfield, that satisfy the following properties:
 \begin{enumerate}[label=(\alph*)]
  \item \label{counterexampleA} $M_3\twoheadrightarrow M_2$ and $M_2\twoheadrightarrow M_1$, but $M_1$ is \textbf{not} a quotient of $M_3$, which shows that quotients are not composable in general;
  \item \label{counterexampleB} $(M_1, M_2, M_3)$ is not a flag matroid, which shows that \eqref{perfect3} does not imply \eqref{perfect2} in general;
  \item \label{counterexampleC} $\mathcal{V}^*(M_2)$ is not a subset of $\mathcal{V}^*(M_3)$, which shows that \eqref{perfect2} does not imply \eqref{perfect1} in general.
 \end{enumerate}

 The \emph{phase hyperfield} is the hyperfield quotient of $\C$ by $\R_{>0}$, whose multiplicative monoid is $\P=S^1\cup\{0\}$, where $S^1 = \{ z\in \C \mid |z| = 1\}$ is the complex unit circle, and whose hyperaddition is given by
 \[
  x\hyperplus y \ = \ \bigg\{ \frac{\alpha x + \beta y}{ \| \alpha x + \beta y \|} \, \bigg| \, \alpha, \beta \in {\mathbb R}_{>0} \bigg\}
 \]
 for $x,y\in\P^\times$ with $y\neq -x$, which is the smallest open arc in $S^1$ connecting $x$ and $y$ if $x\neq y$ and which is $\{x\}$ if $y=x$. If $y=-x$, then $x\hyperplus y=\{0,x,y\}$. See \cite[Example 2.15]{Baker-Bowler19} for details. 
 
 Considered as a tract, the nullset of $\P$ is
 \[\textstyle
  N_\P \ = \ \big\{ \sum x_i \in\N[\P^\times] \, \big| \, \sum\alpha_ix_i=0\text{ in $\C$ for some }\alpha_i\in\R_{>0} \big\}.
 \]

Endow $\P$ with the trivial involution (\textit{cf.} \cite[p. 837]{Baker-Bowler19}) and let $E=\{1,\dotsc,4\}$. We define the aforementioned $\P$-matroids $M_i$ in terms of dual pairs of $\P$-signatures $\cC_i$ for $U_{i,4}$ and $\cD_i$ for $U_{4-i,4}$ for $i=1,2,3$. Namely, the $\P$-signature $\cC_1$ of $U_{1,4}$ consists of the multiples (by elements of $\P^\times=S^1$) of the elements
 \[
  (1, -1, 0, 0),\ (1, 0, -1, 0),\ (1, 0, 0, -1), \ (0, 1, -1, 0),\ (0, 1, 0, -1),\ (0, 0, 1, -1)
 \]
 of $\P^4=\P^E$ and its dual $\P$-signature $\cD_1$ of $U_{3,4}$ consists of the multiples of $(1,1,1,1)$. The $\P$-signature $\cC_2$ of $U_{2,4}$ consists of the multiples of
 \[
  \big(1, e^{6\pi i/4}, e^{3\pi i/4}, 0\big),\ \big(1, e^{2\pi i/4}, 0, e^{5\pi i/4}\big),\ \big(1, 0, e^{3\pi i/4}, e^{5\pi i/4}\big),\ \big(0, 1, e^{5\pi i/4}, e^{3\pi i/4}\big)
 \]
 and its dual $\P$-signature $\cD_2$ consists of the multiples of
 \[
  \big(1, 0, e^{\pi i/4}, e^{7\pi i/4}\big),\ \big(0, 1, e^{7\pi i/4}, e^{\pi i/4}\big),\ \big(e^{7\pi i/4}, e^{\pi i/4}, 1, 0\big),\ \big(e^{\pi i/4}, e^{7\pi i/4}, 0, 1\big).
 \]
 The $\P$-signature $\cC_3$ of $U_{3,4}$ consists of the multiples of $\big(1,1,e^{7\pi i/8}, e^{7\pi i/8}\big)$ and its dual $\P$-signature $\cD_3$ consists of the multiples of
 \begin{align*}
  (1, -1, 0, 0), && \big(e^{15\pi i/8}, 0, 1, 0\big), && \big(0, e^{15\pi i/8}, 1, 0\big), \\
  (0, 0, 1, -1), && \big(e^{15\pi i/8}, 0, 0, 1\big), && \big(0, e^{15\pi i/8}, 0, 1\big).
 \end{align*}
 This defines for $i=1,2,3$ the $\P$-matroids $M_i$ with circuit set $\cC(M_i)=\cC_i$ and cocircuit set $\cC^\ast(M_i)=\cD_i$.
 
 The cocircuit $w=(1,1,1,1)$ of $M_1$ is a covector of $M_2$ since $w\cdot v\in N_\P$ for every $v\in\cC(M_2)$, as can be verified by a direct computation. Since orthogonality is invariant under scaling vectors and $\cC^\ast(M_1)$ consists of the multiples of $w$, we conclude that $\cC^\ast(M_1)\subset\cV^\ast(M_2)$ and therefore $M_2\twoheadrightarrow M_1$. Similarly, we can verify that $\cD_2\perp \cC_3$ and therefore $M_3\twoheadrightarrow M_2$. We have, however, that
 \[
  w \cdot \big(1,1,e^{7\pi i/8}, e^{7\pi i/8}\big) \ = \ 1+1+e^{7\pi i/8}+e^{7\pi i/8},
 \]
 which is not in $N_\P$ since the summands span a strict cone in $\C=\R^2$. Thus $w$ is not a covector of $M_3$. This shows: \ref{counterexampleA} $M_1$ is not a quotient of $M_3$ even though $M_3\twoheadrightarrow M_2$ and $M_2\twoheadrightarrow M_1$, \ref{counterexampleB} $(M_1,M_2,M_3)$ is not a flag $\P$-matroid even though condition \eqref{perfect3} of \autoref{thm: cryptomorphisms for perfect tracts} holds, and \ref{counterexampleC} $\cV^\ast(M_2)$ is not a subset of $\cV^\ast(M_3)$ even though $M_3\twoheadrightarrow M_2$.
\end{ex}

%%%%%%%%%%%%%%%%%%%%%%%%%%%%%%%%%%%%%%%%%%%%%%%%%%%%%%%%%%%%%%%%%%%%%%%%%%%%%%%%%%%%%%%%%%%%%%%%%%%%%%%%%%%%%%%%%%%%%%%%%%%%%%%%%%%%%%%%%%%%%%%%%%%%%%%%%%%%%%%%%%%%%%%%%%%

\subsection{Flags of linear subspaces and valuated flag matroids}
\label{subsection: examples}

At this point, we are prepared for a comprehensive discussion of flag matroids over fields and over the tropical hyperfield.

Recall that the tract associated with a field $K$ replaces the addition of $K$ by the nullset $N_K=\{\sum a_i\in\N[K^\times]\mid \sum a_i=0\text{ in $K$}\}$.

\begin{prop}[Flag matroids over fields]\label{prop: flag matroids over fields}
 Let $K$ be a field and $\bM=(M_1,\dotsc,M_s)$ a flag $K$-matroid. Then $\cV^\ast(M_1)\subset\dotsb\subset \cV^\ast(M_s)$. This establishes a bijection between the set of flag $K$-matroids of rank $\br$ on $E$ and the set of flags $V_1\subset\dotsc\subset V_s$ of linear subspaces of $K^E$ with $\dim V_i=r_i$.
\end{prop}

\begin{proof}
 By \cite[Prop. 2.19]{Anderson19}, the tract $K$ is perfect and the covector set $\cV^\ast(M)$ of a $K$-matroid $M$ forms a linear subspace $V$ of $K^E$, which establishes a bijection between the set of $K$-matroids of rank $r$ on $E$ and the set of linear subspaces of $K^E$ of dimension $r$. By \autoref{thm: cryptomorphisms for perfect tracts}, a sequence $(M_1,\dotsc,M_s)$ of $K$-matroids forms a flag $K$-matroid if and only if $\cV^\ast(M_1)\subset\dotsb\subset \cV^\ast(M_s)$.
\end{proof}

We turn to the comparison of valuated flag matroids in the sense of \cite[Def.\ 4.2.2]{Brandt-Eur-Zhang21} with flag $\T$-matroids in our sense. We rephrase the definitions of \cite{Brandt-Eur-Zhang21} in terms of the Berkovich model $\R_{\geq0}$ of the tropical semifield, using the semiring isomorphism $-\exp: \R\cup\{\infty\} \to\R_{\geq0}$ between the min-plus algebra and the Berkovich model, which transforms the tropical addition ``min'' into ``max'' and the tropical multiplication ``plus'' into usual multiplication.

A \emph{Dress-Wenzel valuation} is a map $\mu:E^r\to\R_{\geq0}$ such that for every choice of elements $x_1,\dotsc,x_{r-1},y_1,\dotsc,y_{r+1}\in E$ there is an $i\in\{1,\dotsc,r\}$ such that
\[
 \mu(y_1,\dotsc,y_r)\cdot\mu(y_{r+1},x_1,\dotsc,x_r) \ \leq \ \mu(y_1,\dotsc,\widehat{y_i},\dotsc,y_{r+1})\cdot\mu(y_{i},x_1,\dotsc,x_r).
\]
Two valuations $\mu_1, \mu_2:E^r\to\R_{\geq0}$ are \textit{equivalent} if there exists $\alpha \in \R_{>0}$ such that $\mu_1 = \alpha \mu_2$. A \emph{valuated matroid}\footnote{There is a discrepancy of terminology in the literature. What is called a \emph{valuated matroid} in \cite{Brandt-Eur-Zhang21} is called a \emph{valuation} in Dress-Wenzel's paper that introduces valuated matroids (cf.\ \cite[Def.\ 1.1]{Dress-Wenzel92b}), and it corresponds to a Grassmann-Pl\"ucker function with tropical coefficients in Baker-Bowler theory. Valuated matroids in Dress-Wenzel's sense appear in \cite{Brandt-Eur-Zhang21} as projective classes of valuated matroids in the latter sense, but without a distinct name. We follow the terminological conventions of Dress-Wenzel in our exposition.} is the equivalence class $[\mu]$ of a Dress-Wenzel valuation $\mu:E^r\to\R_{\geq0}$.

Recall that the multiplicative monoid of $\T$ is $\R_{\geq0}$ and the nullset 
 \[
  N_\T \ = \ \big\{a_1+\dotsc+a_n\in\N[\R_{>0}] \, \big| \, \text{the maximum occurs twice in }a_1,\dotsc,a_n\big\}.
 \]
 We denote by $\iota:\R_{\geq0}\to\T$ the identity map. Note that a function $\mu:E^r\to\R_{\geq0}$ is a Dress-Wenzel valuation if and only if $\iota \circ \mu:E^r\to\T$ is a Grassmann-Pl\"ucker function. This defines a bijection that sends a valuated matroid $M=[\mu]$ to the $\T$-matroid $\widetilde M=[\iota\circ\mu]$ (cf. \cite[Ex. 3.32]{Baker-Bowler19}).

Let $M=[\mu]$ and $N=[\nu]$ be valuated matroids on $E$ of respective ranks $r$ and $w$. Following \cite{Brandt-Eur-Zhang21}, we say that \emph{$M$ is a quotient of $N$}, and write $M\twoheadleftarrow N$, if 
for all $x_1,\dotsc,x_{r-1},y_1,\dotsc,y_{w+1}\in E$ there is some $i\in\{1,\dotsc,w+1\}$ such that
\begin{equation}
\label{valuated equation}
 \nu(y_1,\dotsc,y_w)\cdot\mu(y_{w+1},x_1,\dotsc,x_r) \ \leq \ \nu(y_1,\dotsc,\widehat{y_i},\dotsc,y_{w+1})\cdot\mu(y_{i},x_1,\dotsc,x_r).
\end{equation} 
A \emph{valuated flag matroid} is a sequence $\bM=(M_1,\dotsc,M_s)$ of valuated matroids such that $M_i\twoheadleftarrow M_j$ for all $1\leq i\leq j\leq s$.

\begin{prop}[Valuated flag matroids]\label{prop: valuated matroids}
 A sequence $\bM=(M_1,\dotsc,M_s)$ of valuated matroids is a valuated flag matroid if and only if the sequence $\widetilde\bM=(\widetilde M_1,\dotsc,\widetilde M_s)$ of associated $\T$-matroids is a flag $\T$-matroid.
\end{prop}

\begin{proof}
We only need to show for a pair of valuated matroids $M$ and $N$ on $E$ that $M\twoheadleftarrow N$ if and only if $\widetilde N\twoheadrightarrow\widetilde M$. Let $\mu:E^r\to\R_{\geq0}$ and $\nu:E^w\to\R_{\geq0}$ be Dress-Wenzel valuations representing $M$ and $N$, respectively. Let $x_1,\dotsc,x_{r-1},y_1,\dotsc,y_{w+1}\in E$. Then for every reordering of $x_1,\dotsc,x_{r-1}$ and of $y_1,\dotsc,y_{w+1}\in E$ there is an $i\in\{1,\dotsc,w+1\}$ such that equation \eqref{valuated equation} holds if and only if the maximum of
 \[
  \big\{ \nu(y_1,\dotsc,\widehat{y_i},\dotsc,y_{w+1})\cdot\mu(y_{i},x_1,\dotsc,x_r)  \, \big| \, i=1,\dotsc,w+1 \big\}
 \]
 is attained at least twice. By the definition of $\T$, this happens if and only if
 \[
  \sum_{i=1}^{w+1} \ \iota\circ\nu(y_1,\dotsc,\widehat{y_i},\dotsc,y_{w+1})\cdot\iota\circ\mu(y_{i},x_1,\dotsc,x_r) \ \ \in \ \ N_\T.
 \]
 Since $\epsilon=1$ in $\T$, this is precisely the condition for $\widetilde M\twoheadrightarrow\widetilde N$ if we vary through all $x_1,\dotsc,x_{r-1},y_1,\dotsc,y_{w+1}\in E$, and thus the result follows.
\end{proof}

\begin{rem}\label{rem: valuated flag matroids}
 Since $\T$ is perfect (see \cite[Cor. 3.45]{Baker-Bowler19}) and the covector set $\cV^\ast(M)$ of a valuated matroid $M$ is a tropical linear space, \autoref{thm: cryptomorphisms for perfect tracts} identifies flag $\T$-matroids $(M_1,\dotsc,M_s)$ with flags $\cV^\ast(M_1)\subset\dotsb\subset\cV^\ast(M_s)$ of tropical linear subspaces in $\T^E$. This recovers \cite[Thm.\ 4.3.1]{Brandt-Eur-Zhang21}. The Pl\"ucker flag relations also show at once that a flag $\T$-matroid is the same thing as a point of the flag Dressian $FlDr(r_1,\dotsc,r_s;n)$, which recovers \cite[Prop.\ 4.2.3]{Brandt-Eur-Zhang21}; also see \autoref{subsubsection: tropical points of the flag variety}.
\end{rem}

%%%%%%%%%%%%%%%%%%%%%%%%%%%%%%%%%%%%%%%%%%%%%%%%%%%%%%%%%%%%%%%%%%%%%%%%%%%%%%%%%%%%%%%%%%%%%%%%%%%%%%%%%%%%%%%%%%%%%%%%%%%%%%%%%%%%%%%%%%%%%%%%%%%%%%%%%%%%%%%%%%%%%%%%%%%

\subsection{Flags of minors}
\label{subsection: flags of minors}

As explained in the first example of the introduction, certain sequences of minors of a matroid $M'$ on $E'$ are flag matroids, see \cite{Kung77}. This generalizes verbatim to flag $F$-matroids over an arbitrary tract $F$.

\begin{thm}\label{thm: flags of minors}
 Let $F$ be a tract, $p\in\N$, $M'$ an $F$-matroid on $E'=E\sqcup\{n+1, \dotsc, n+p\}$ and fix integers $0\leq n_s\leq \dotsc\leq n_1 \leq p$. For $i=1,\dotsc,s$, we define $M_i=M'\minor{I_i}{J_i}$ where $J_i=\{n+1, \dotsc, n+n_i\}$ and $I_i=\{n+(n_i+1),\dotsc,n+p\}$. Then $(M_1,\dotsc,M_s)$ is a flag $F$-matroid on $E$.
\end{thm}

\begin{proof}
It is enough to show that if $a$ and $b$ are integers such that $0\leq b < a \leq p$, then $M'\backslash I_b/J_b \twoheadrightarrow M'\backslash I_a/J_a$, where $J_c:=\{n+1, \dotsc , n+c\}$ and $I_c:=\{n+(c+1), \dotsc , n+p\}$ for $c \in \{a, b\}$.

Let $\varphi$ be a Grassmann-Pl\"ucker function such that $M' = M_\varphi$. Let $r:= \rk(M')$ and $r_c:=\rk(M'\backslash I_c/J_c)$ for $c\in\{a, b\}$. We aim to find suitable sets $\{q_1, \dotsc, q_{r-r_b}\}$ and $\{u_1, \dotsc, u_{r-r_a}\}$ such that 
\[
\varphi \backslash I_b/J_b(h_1, \dotsc, h_{r_b}) = \varphi(h_1, \dotsc, h_{r_b}, q_1, \dotsc, q_{r-r_b})
\]
and 
\[
\varphi \backslash I_a/J_a(g_1,\dotsc, g_{r_a}) = \varphi(g_1,\dotsc, g_{r_a},u_1, \dotsc, u_{r-r_a}).
\]

Let $Y\in\mathcal{B}(M'|J_b)$. Then there are $X\subseteq E$, $V\subseteq  J_a\backslash J_b$ and $Z\subseteq I_a$ such that $Y\sqcup X\in\mathcal{B}(M'|E\sqcup J_b)$, $Y\sqcup X\sqcup V \in\mathcal{B}(M'|E\sqcup J_a)$ and $Y\sqcup X \sqcup V \sqcup Z \in\mathcal{B}(M')$. This means that $Z\in \mathcal{B}(M'/E\sqcup J_a)$, $Z\sqcup V\in \mathcal{B}(M'/E\sqcup J_b)$ and $Y\sqcup V \in \mathcal{I}(M'|J_a)$. Thus there is $Y^+\subseteq J_a\backslash (Y\sqcup V)$ such that $Y\sqcup Y^+\sqcup V \in \mathcal{B}(M'|J_a)$. Let $Y = \{y_1, \dotsc, y_\ell\}$, $Y^+ = \{y_{\ell+1}, \dotsc, y_t\}$, $V = \{v_1, \dotsc, v_m\}$ and $Z = \{z_1, \dotsc, z_w\}$. Note that
\[
\begin{aligned}
    & \varphi \backslash I_b/J_b(h_1, \dotsc, h_{r_b})
    \\
=\; & \varphi \backslash I_b(h_1, \dotsc, h_{r_b}, y_1, \dotsc, y_\ell)
\\
=\; & \varphi(h_1, \dotsc, h_{r_b}, y_1, \dotsc, y_\ell, z_1, \dotsc, z_w, v_1, \dotsc, v_m)
\end{aligned}
\]
and
\[
\begin{aligned}
    & \varphi \backslash I_a/J_a(g_1,\dotsc, g_{r_a})
    \\
=\; & \varphi \backslash I_a(g_1,\dotsc, g_{r_a}, y_1, \dotsc, y_\ell, y_{\ell+1}, \dotsc, y_t, v_1, \dotsc, v_m)
\\
=\; & \varphi(g_1,\dotsc, g_{r_a}, y_1, \dotsc, y_\ell, y_{\ell+1}, \dotsc, y_t, v_1, \dotsc, v_m, z_1, \dotsc, z_w).
\end{aligned}
\]

To finish the proof, we only need to show that $\varphi \backslash I_b/J_b$ and $\varphi \backslash I_a/J_a$ satisfy the Pl\"ucker flag relations \eqref{flag equation}. Let $\{w_1, \dotsc, w_{r_b+1}\}$ and $\{f_1, \dotsc, f_{r_a-1}\}$ be subsets of $E$. Define the following two sequences
\[
\begin{aligned}
    & (w_{r_b+2}, \dotsc, w_{r+1}) = (y_1, \dotsc, y_\ell, z_1, \dotsc, z_w, v_1, \dotsc, v_m)
    \\
\text{and }\; & (f_{r_a}, \dotsc, f_{r-1}) = (y_1, \dotsc, y_\ell, y_{\ell+1}, \dotsc, y_t, v_1, \dotsc, v_m, z_1, \dotsc, z_w).
\end{aligned}
\]

Then 
\[
\begin{aligned}
    & \underset{k = 1}{\overset{r_b+1}{\sum}} \epsilon^k \varphi \backslash I_b/J_b(w_1, \dotsc, \widehat{w_k}, \dotsc w_{r_b+1}) \cdot \varphi \backslash I_a/J_a(w_k, f_1, \dotsc, f_{r_a-1})
    \\
=\; & \underset{k = 1}{\overset{r_b+1}{\sum}} \epsilon^k \varphi(w_1, \dotsc, \widehat{w_k}, \dotsc w_{r+1}) \cdot \varphi(w_k, f_1, \dotsc, f_{r-1})
\\
=\; & \, \underset{k = 1}{\overset{r+1}{\sum}} \, \epsilon^k \varphi(w_1, \dotsc, \widehat{w_k}, \dotsc w_{r+1}) \cdot \varphi(w_k, f_1, \dotsc, f_{r-1}),
\end{aligned}
\]
as the terms for $k=r_b+2,\dotsc,r+1$ are zero, because $\{w_{r_b+2}, \dotsc, w_{r+1}\} \subseteq \{f_{r_a}, \dotsc, f_{r-1}\}$. As $\varphi$ is a Grassmann-Plücker function, the last sum is in $N_F$.
\end{proof}

\begin{rem}\label{rem: flags of minors over fields}
 By \cite{Kung77}, every flag matroid is a sequence of minors. This also holds for flag matroids over a field $K$ since the association $M'\mapsto\big(M'\minor{I_i}{J_i}\big)_{i=1,\dotsc,s}$ defines a $\GL(E,K)$-equivariant rational map $\Gr(r_s,E')(K)\DashedArrow[->]\Fl(\br,E)(K)$ where $\br=(r_1,\dotsc,r_s)$, $E=\{1,\dotsc,n\}$ and $E'=\{1,\dotsc,n+r_s-r_1\}$. Since its image is not empty and since $\GL(E,K)$ acts transitively on $\Fl(\br,E)(K)$, we conclude that every flag $K$-matroid is a flag of minors of a $K$-matroid $M'$ on $E'$.
 
 Las Vergnas expected that the same holds true for oriented flag matroids of rank $(r_1,r_2)$, i.e., that every quotient $M_2\twoheadrightarrow M_1$ of oriented matroids is of the form $M'\backslash I\twoheadrightarrow M'/I$ for some oriented matroid $M'$ on a larger set $E'=E\cup I$. This was however disproven by Richter-Gebert in \cite[Cor.\ 3.5]{Richter-Gebert93}. 

 For applications to the representation theory of flag matroids, it would be useful to get a better hold on the question for which $\br$, $E$ and $F$ all flag $F$-matroids of rank $\br$ on $E$ are flags of minors. 
\end{rem}

In the special case of rank $(r,r+1)$, we obtain a positive answer for perfect tracts.

\begin{prop}\label{prop: contraction-deletion flag}
 Let $F$ be a perfect tract and $\bM$ a flag $F$-matroid of rank $(r,r+1)$. Let $\widehat{E}=E\sqcup\{e\}$. Then there is an $F$-matroid $\widehat{M}$ of rank $r+1$ on $\widehat{E}$ such that $\bM=(\widehat{M}/e,\, \widehat{M}\backslash e)$. More precisely, the set of $F$-matroids $\widehat{M}$ with $\bM=(\widehat{M}/e,\, \widehat{M}\backslash e)$ stays in bijection with $F^\times$.
\end{prop}

\begin{proof}
Let $\mu: E^r\rightarrow F$ and $\nu: E^{r+1}\rightarrow F$ be Grassmann-Plücker functions such that $\bM = (M_\mu, M_\nu)$. We define the function $\varphi:\widehat{E}^{r+1}\to F$ by
\[
 \varphi(x_1, \dotsc, x_{r+1}) \ = \ 
 \begin{cases}
    \nu(x_1, \dotsc,x_{r+1}) & \text{if } e\notin\{x_1, \dotsc, x_{r+1}\}; \\
    \epsilon^{r+1-\ell}\mu(x_1, \dotsc, \widehat{x_\ell}, \dotsc, x_{r+1}) 
                             & \text{if } e=x_\ell\notin\{x_1,\dotsc,\widehat{x_\ell},\dotsc,x_{r+1}\}; \\
    0                        & \text{if } e=x_\ell=x_k\text{ for }\ell\neq k.
 \end{cases}
\]
We aim to show that $\varphi$ is a Grassmann-Plücker function and $\bM = (M_\varphi/e, M_\varphi\backslash e)$. As $\nu$ is not identically zero, $\varphi$ satisfies \ref{GP1}. Property \ref{GP2} is obvious but for the case that $e=x_i\notin\{x_1,\dotsc,\widehat{x_i},\dotsc,x_{r+1}\}$ where we find
\[
\begin{aligned}
    & \varphi(x_1, \dotsc, x_i, \dotsc, x_j, \dotsc,x_{r+1})\\
  = \;& \epsilon^{r+1-i}\mu(x_1, \dotsc, x_{i-1}, \widehat{x_i}, x_{i+1}, \dotsc, x_j, \dotsc, x_{r+1})\\
  = \; & \epsilon^{(r+1-i)+j-(i+1)}\mu(x_1, \dotsc, x_{i-1}, x_j, x_{i+1}, \dotsc, x_{j-1}, \widehat{x_i}, x_{j+1}, \dotsc, x_{r+1})\\
  = \; & \epsilon \cdot \varphi(x_1, \dotsc, x_j, \dotsc, x_i, \dotsc,x_{r+1}).
\end{aligned}
\]
Since $F$ is perfect, we only need to show that $\varphi$ satisfies \ref{GP3} for subsets $I=\{y_1, \dotsc, y_{r+2}\}$ and $J=\{x_1, \dotsc, x_{r}\}$ of $E$ with $\# I\backslash J = 3$ (see \cite[Thm.\ 3.46]{Baker-Bowler19}). As $\varphi$ satisfies \ref{GP2}, one can suppose that $a:= y_1, \,b:=y_2, \,c:=y_3$ are the elements of $I\backslash J$, $d=x_1$ is the unique element in $J\backslash I$ and $x_i=y_{i+2}$ for $i=2, \dotsc, r$. For $w,z \in \{a, b, c, d\}$ define $\widehat{\varphi}(w, z):=\varphi(w, z, y_4, \dotsc, y_{r+2})$. Note that 
\begin{equation}
\label{equation aux}
    \underset{k=1}{\overset{r+2}{\sum}} \epsilon^k \cdot \varphi(y_1, \dotsc, \widehat{y_k}, \dotsc, y_{r+2}) \cdot \varphi(y_k, x_1, \dotsc, x_{r})
\end{equation}
\[
=\; \epsilon^1 \cdot \widehat{\varphi}(b, c) \cdot \widehat{\varphi}(a,d) 
  + \epsilon^2 \widehat{\varphi}(a, c) \cdot \widehat{\varphi}(b,d)
+ \epsilon^3 \cdot \widehat{\varphi}(a, b) \cdot \widehat{\varphi}(c, d). 
\]
We have 4 cases to analyze.

\medskip\noindent
\textbf{Case 1:} If $e\notin I\cup J$, then equation \eqref{equation aux} becomes
\[
 \underset{k=1}{\overset{r+2}{\sum}} \epsilon^k \cdot \nu(y_1, \dotsc, \widehat{y_k}, \dotsc, y_{r+2}) \cdot \nu(y_k, x_1, \dotsc, x_{r}),
\]
which is in $N_F$ because $\nu$ is a Grassmann-Pl\"ucker function.

\medskip\noindent
\textbf{Case 2:} If $e = y_\ell = x_{\ell+2}\in I\cap J$, then we define $(t_1, \dotsc, t_{r+1}) = (y_1, \dotsc, \widehat{y_\ell}, \dotsc, y_{r+2})$ and $(u_1, \dotsc, u_{r-1}) = (x_1, \dotsc, \widehat{x_{\ell-2}}, \dotsc, x_{r})$. With this, equation \eqref{equation aux} becomes
\[
\begin{aligned}
    & \underset{k=1}{\overset{\ell-1}{\sum}} \epsilon^{k} \cdot \mu(y_1, \dotsc, \widehat{y_k}, \dotsc, \widehat{y_\ell}, \dotsc, y_{r+2}) \cdot \mu(y_k, x_1, \dotsc, \widehat{x_{\ell-2}}, \dotsc, x_{r})
    \\
+ \;& \epsilon^\ell\cdot \nu(y_1, \dotsc, \widehat{y_\ell}, \dotsc, y_{r+2})\cdot 0
    \\ 
+ \;& \underset{k=\ell+1}{\overset{r+2}{\sum}} \epsilon^{k-1} \cdot \mu(y_1, \dotsc, \widehat{y_\ell}, \dotsc, \widehat{y_k}, \dotsc, y_{r+2}) \cdot \mu(y_k, x_1, \dotsc, \widehat{x_{\ell-2}}, \dotsc, x_{r})
    \\
= \;& \underset{k=1}{\overset{r+1}{\sum}} \epsilon^{k} \cdot \mu(t_1, \dotsc, \widehat{t_k}, \dotsc, t_{r+1}) \cdot \mu(t_k, u_1, \dotsc, u_{r-1}),
\end{aligned}
\]
which is in $N_F$ because $\mu$ is a Grassmann-Pl\"ucker function.

\medskip\noindent
\textbf{Case 3:} If $e=d$, then equation \eqref{equation aux} becomes
\[
\begin{aligned}
    & \epsilon^{r}\nu(b, c, y_4, \dotsc, y_{r+2}) \cdot \mu(a, x_2, \dotsc, x_r)
    \\
+ \;& \epsilon^{r+1}\nu(a,c, y_4, \dotsc, y_{r+2}) \cdot \mu(b, x_2, \dotsc, x_r)
\\
+ \;& \epsilon^{r+2}\nu(a, b, y_4, \dotsc, y_{r+2}) \cdot \mu(c, x_2, \dotsc, x_r) 
\\
=\; & \epsilon^{r-1}\cdot \underset{k = 1}{\overset{r+2}{\sum}} \epsilon^k \nu(y_1, \dotsc, \widehat{y_k}, \dotsc y_{r+2}) \cdot \mu(y_k, x_2, \dotsc, x_{r}),
\end{aligned}
\]
which is in $N_F$ because $M_\nu\twoheadrightarrow M_\mu$.

\medskip\noindent
\textbf{Case 4:} The cases for $e\in\{a,b,c\}$ are similar to prove. We demonstrate the case $e=a$, in which equation \eqref{equation aux} becomes
\[
\begin{aligned}
    & \epsilon^{(r+1)+2} \nu(b, c, y_4, \dotsc, y_{r+2}) \cdot \mu(d, x_2, \dotsc, x_r)
    \\
+\; & \epsilon^{(r+2)}  \mu(c, x_2, \dotsc, x_r) \cdot \nu(b, d, y_4, \dotsc, y_{r+2})
\\
+\; & \epsilon^{(r+3)-2}  \mu(b, x_2, \dotsc, x_r) \cdot \nu(c, d, y_4, \dotsc, y_{r+2})
\\
=\; & \epsilon^{r}\cdot \underset{k = 1}{\overset{r+2}{\sum}} \epsilon^k \nu(z_1, \dotsc, \widehat{z_k}, \dotsc z_{r+2}) \cdot \mu(z_k, x_2, \dotsc, x_{r}),
\end{aligned}
\]
which is in $N_F$ because $M_\nu\twoheadrightarrow M_\mu$, where $(z_1, \dotsc, z_{r+2}) := (b, c, d, y_4, \dotsc, y_{r+2})$.

To conclude, this shows that $\varphi$ is a Grassmann-Pl\"ucker function. Since $\varphi/e = \mu$ and $\varphi\backslash e= \nu$, one has $(M_\varphi/e, M_\varphi\backslash e) = \bM$.

For the last part of the theorem, it is enough to prove that the set of Grassmann-Pl\"ucker functions $\lambda$ with $\bM = (M_\lambda/e, M_\lambda\backslash e)$ stays in bijection with $(F^\times)^2$. This follows from the following two facts:

\begin{enumerate}
    \item If $\lambda$ is a Grassmann-Pl\"ucker function on $\widehat{E}$ satisfying $\bM = (M_\lambda/e, M_\lambda\backslash e)$, then there are $\alpha$, $\beta\in F^\times$ such that $\lambda/e = \alpha \cdot \mu$ and $\lambda\backslash e = \beta \cdot \nu$.
    \item Conversely, given $\alpha$, $\beta\in F^\times$, by a construction analogue to that of $\varphi$ above, there is a unique Grassmann-Pl\"ucker function $\lambda$ on $\widehat{E}$ such that $\lambda/e = \alpha \cdot \mu$ and $\lambda\backslash e = \beta \cdot \nu$, which implies 
$\bM = (M_\lambda/e, M_\lambda\backslash e)$.
\end{enumerate}
This concludes the prof of \autoref{prop: contraction-deletion flag}.
\end{proof}

\begin{cor}\label{cor: lifting flag matroids}
 Let $F$ be a perfect tract. Let $M$ and $N$ be $F$-matroids with respective Grassmann-Pl\"ucker functions $\mu:E^r\to F$ and $\nu:E^{r+1}\to F$. Assume that $(\underline{M},\underline{N})$ is a flag matroid. Then $(M,N)$ is a flag $F$-matroid if 
 \begin{multline*}
  \mu(a,x_1,\dotsc,x_{r-1})\cdot\nu(b,c,x_1,\dotsc,x_{r-1}) \ - \ \mu(b,x_1,\dotsc,x_{r-1})\cdot\nu(a,c,x_1,\dotsc,x_{r-1})\\
  + \ \mu(c,x_1,\dotsc,x_{r-1})\cdot\nu(a,b,x_1,\dotsc,x_{r-1}) \quad \in \ N_F
 \end{multline*}
 for all $x_1,\dotsc,x_{r-1},a,b,c\in E$.
\end{cor}

\begin{proof}
 Let $\widehat E=E\sqcup \{e\}$ and $\varphi:\widehat E^{r+1}\to F$ be defined by 
 \[
  \varphi(x_1, \dotsc, x_{r+1}) \ = \ 
  \begin{cases}
     \nu(x_1, \dotsc,x_{r+1}) & \text{if } e\notin\{x_1, \dotsc, x_{r+1}\}; \\
     \epsilon^{r+1-\ell}\mu(x_1, \dotsc, \widehat{x_\ell}, \dotsc, x_{r+1}) 
                              & \text{if } e=x_\ell\notin\{x_1,\dotsc,\widehat{x_\ell},\dotsc,x_{r+1}\}; \\
     0                        & \text{if } e=x_\ell=x_k\text{ for }\ell\neq k.
  \end{cases}
 \]
 Following through the steps of the proof of \autoref{prop: contraction-deletion flag}, we see that the relations in the statement of the corollary are sufficient to prove that $\varphi$ is a weak Grassmann-Pl\"ucker function. 
 
% By \cite[Prop.\ 5.1]{Kung77}, $(\underline{M},\underline{N})=(M'/e,\, M'\backslash e)$ for a rank $r+1$-matroid $M'$ on $\widehat E$, which is unique by \autoref{prop: contraction-deletion flag} as $\K^\times=\{1\}$. 

 Since $(\underline{M},\underline{N})$ is a flag matroid, it is of the form $(\widetilde M/e,\widetilde M\backslash e)$ for a matroid $\widetilde M$ on $\widehat E$ by \autoref{prop: contraction-deletion flag}, which is unique since $\K^\times=\{1\}$. Since the construction of $\varphi$ is formally independent of the tract $F$, it is evident that $\varphi$ represents the matroid $\widetilde M$. We conclude that $\varphi$ defines an $F$-matroid $M_\varphi=[\varphi]$ by \cite[Thm.\ 3.46]{Baker-Bowler19}.
 
 By the construction of $\varphi$, we have $M=M_\varphi/e$ and $N=M_\varphi\backslash e$. Thus it follows from \autoref{thm: flags of minors} that $(M,N)=(M_\varphi/e,M_\varphi\backslash e)$ is a flag $F$-matroid. 
\end{proof}

%%%%%%%%%%%%%%%%%%%%%%%%%%%%%%%%%%%%%%%%%%%%%%%%%%%%%%%%%%%%%%%%%%%%%%%%%%%%%%%%%%%%%%%%%%%%%%%%%%%%%%%%%%%%%%%%%%%%%%%%%%%%%%%%%%%%%%%%%%%%%%%%%%%%%%%%%%%%%%%%%%%%%%%%%%%
%%%%%%%%%%%%%%%%%%%%%%%%%%%%%%%%%%%%%%%%%%%%%%%%%%%%%%%%%%%%%%%%%%%%%%%%%%%%%%%%%%%%%%%%%%%%%%%%%%%%%%%%%%%%%%%%%%%%%%%%%%%%%%%%%%%%%%%%%%%%%%%%%%%%%%%%%%%%%%%%%%%%%%%%%%%
\section{The moduli space of flag matroids}
\label{section: moduli space of flag matroids}
%%%%%%%%%%%%%%%%%%%%%%%%%%%%%%%%%%%%%%%%%%%%%%%%%%%%%%%%%%%%%%%%%%%%%%%%%%%%%%%%%%%%%%%%%%%%%%%%%%%%%%%%%%%%%%%%%%%%%%%%%%%%%%%%%%%%%%%%%%%%%%%%%%%%%%%%%%%%%%%%%%%%%%%%%%%
%%%%%%%%%%%%%%%%%%%%%%%%%%%%%%%%%%%%%%%%%%%%%%%%%%%%%%%%%%%%%%%%%%%%%%%%%%%%%%%%%%%%%%%%%%%%%%%%%%%%%%%%%%%%%%%%%%%%%%%%%%%%%%%%%%%%%%%%%%%%%%%%%%%%%%%%%%%%%%%%%%%%%%%%%%%

In this part, we construct the moduli space $\Fl(r_1,\dots,r_s;E)$ of flag matroids, which adds the cryptomorphic description of a flag $F$-matroid as an $F$-rational point of $\Fl(r_1,\dots,r_s;E)$. This extends results by Baker and the second author in \cite{Baker-Lorscheid21} from matroids to flag matroids.

Our construction of $\Fl(r_1,\dots,r_s;E)$ utilizes ordered blue schemes, a theory that was developed in \cite{Lorscheid15} and \cite{Baker-Lorscheid21}. Since we cannot present the necessary background on this theory in a compact way, we assume that the reader is familiar with the latter paper by Baker and the second author. In particular, we assume that the reader is familiar with the terminologies and notations for ordered blueprints; see \cite{Baker-Lorscheid21} for details.

Throughout the section, we fix a ground set $E=\{1,\dotsc,n\}$ and a rank $\br=(r_1,\dotsc,r_s)$ with $0\leq r_1\leq\dotsb\leq r_s\leq n$. For $\bI=(x_1,\dotsc,x_r)\in E^r$ and $e\in E$, we use the notations $\bI e=(x_1,\dotsc,x_r,e)$ and $|\bI|=\{x_1,\dotsc,x_r\}$.

%%%%%%%%%%%%%%%%%%%%%%%%%%%%%%%%%%%%%%%%%%%%%%%%%%%%%%%%%%%%%%%%%%%%%%%%%%%%%%%%%%%%%%%%%%%%%%%%%%%%%%%%%%%%%%%%%%%%%%%%%%%%%%%%%%%%%%%%%%%%%%%%%%%%%%%%%%%%%%%%%%%%%%%%%%%
\subsection{Idylls and ordered blueprints}

We recall from \cite[Thm.\ 2.21]{Baker-Lorscheid21} that a tract $F=(F^\times,N_F)$ can be realized as the ordered blueprint $B=F^\oblpr$ with underlying monoid $B^\bullet=F$, ambient semiring $B^+=\N[F^\times]$ and (additive and multiplicative) partial order $\fr_B \ = \ \gen{0\leq \sum a_i\mid \sum a_i\in N_F}$ on $B^+$. An \emph{idyll} is an ordered blueprint of the form $B=F^\oblpr$ for a tract $F$.

The association $F\mapsto F^\oblpr$ extends naturally to a functor $(-)^\oblpr:\Tracts\to\OBlpr$ from tracts to ordered blueprints that comes with a right adjoint $(-)^\tract:\OBlpr\to\Tracts$. Composing both functors sends a tract $F$ to the tract $F'=(F^\oblpr)^\tract$ that equals $F$ as a monoid and whose nullset $N_{F'}$ equals the closure of $N_F$ under addition. This means that the functors $(-)^\tract$ and $(-)^\oblpr$ restrict to mutually inverse equivalence of categories between idylls and tracts whose nullset is additively closed. 

By abuse of language, we call tracts with additively closed nullset also idylls and we do not make a distinction between the tract and the associated ordered blueprint. In particular, we denote the associated ordered blueprint with the same symbol, which applies, for instance, to the idylls $\Funpm$, $\K$, $\S$ and $\T$. Note that, more generally, all partial fields and all hyperfields are idylls. Following this logic, we define a \emph{flag $F$-matroid} for an idyll $F$ with associated tract $F'=F^\tract$ as a flag $F'$-matroid.

Note that every idyll is an $\Funpm$-algebra in a unique way. 

\begin{ex}
 The incarnations of $\Funpm$ and $\K$ as ordered blueprints are
 \[
  \Funpm \ = \ \bpquot{\{0,\pm1\}}{\gen{0\leq 1+(-1)}} \quad \text{and} \quad \K \ = \ \bpquot{\{0,1\}}{\gen{0\leq 1+1,\ 0\leq 1+1+1}}.
 \]
 The associated ordered blueprint $B=K^\oblpr$ of a field $K$ has underlying monoid $B^\bullet=K$, ambient semiring $B^+=\N[K^\times]$ and partial order
  $\fr_B \ = \ \gen{0\leq \sum a_i \mid \sum a_i=0\text{ in $K$}}$.
\end{ex}

%%%%%%%%%%%%%%%%%%%%%%%%%%%%%%%%%%%%%%%%%%%%%%%%%%%%%%%%%%%%%%%%%%%%%%%%%%%%%%%%%%%%%%%%%%%%%%%%%%%%%%%%%%%%%%%%%%%%%%%%%%%%%%%%%%%%%%%%%%%%%%%%%%%%%%%%%%%%%%%%
\subsection{Flag matroid bundles}
\label{subsection: flag matroid bundles}

As a first step towards the moduli space of flag matroids, we generalize the notion of a flag $F$-matroid to that of a flag matroid bundle on an (ordered blue) $\Funpm$-scheme. 

\begin{df}
 Let $X$ be an $\Funpm$-scheme, $E=\{1,\dotsc,n\}$ and $\br=(r_1,\dotsc,r_s)$. A \emph{flag of Grassmann-Pl\"ucker functions (of rank $\br$ on $E$) in $X$} are line bundles $\cL_1,\dotsc,\cL_s$ on $X$ and functions
 \[
  \varphi_i:\ E^{r_i} \ \longrightarrow \ \Gamma(X,\cL_i)
 \]
 for $i=1,\dotsc,s$ such that for all $1\leq i\leq j\leq s$,
 \begin{enumerate}[label = (GP\arabic*${}^\ast$)]
  \item\label{GP1*} the global sections $\varphi_i(\bI)$ with $\bI\in E^{r_i}$ \emph{generate $\cL_i$}, i.e., for every point $x\in X$, there is an $\bI\in E^{r_i}$ such that the class of $\varphi_i(\bI)$ in $\cO_{X,x}$ is a unit;
  \item\label{GP2*} $\varphi_i$ is \emph{alternating}, i.e., $\varphi_i(x_{\sigma(1)},\dotsc,x_{\sigma(r_i)})=\sign(\sigma)\cdot\varphi_i(x_1,\dotsc,x_{r_i})$ and $\varphi_i(\bI)=0$ if $\#|\bI|<r_i$;
  \item\label{GP3*} $\varphi_i$ and $\varphi_j$ satisfy the \emph{Pl\"ucker flag relations}
        \[
         0 \ \leq \ \sum_{k=1}^{r_j+1} \ \epsilon^k \, \cdot \, \varphi_i(y_k,\, x_1,\dotsc,\, x_{r_i-1}) \, \cdot \, \varphi_j(y_1,\dotsc,\, \widehat{y_k},\dotsc,\, y_{r_j+1})
        \]
        as relations in $\Gamma(X,\cL_i\otimes\cL_j)$ for all $x_1,\dotsc,x_{r_i-1},y_1,\dotsc,y_{r_j+1}\in E$.
 \end{enumerate}
 \begin{comment}
 An \emph{isomorphism between two flags of Grassmann-Pl\"ucker functions} %$(\varphi_1,\dots,\varphi_s)$ and $(\varphi'_1,\dots,\varphi'_s)$ of the same type 
 \[
  \big(\varphi_i:E^{r_i}\to\Gamma(X,\cL_i)\big)_{i=1,\dotsc,s} \quad \text{and} \quad \big(\varphi'_i:E^{r_i}\to\Gamma(X,\cL'_i)\big)_{i=1,\dotsc,s}
 \]
 is a collection of isomorphisms $\iota_i:\cL_i\to \cL_i'$ of line bundles (for $i=1,\dotsc,s$) such that $\varphi'_i=\iota_{i,X}\circ\varphi_i$ for $i=1,\dotsc,s$ where $\iota_{i,X}:\Gamma(X,\cL_i)\to\Gamma(X,\cL_i')$ is $\iota_i$ evaluated on global sections. A \emph{flag matroid bundle (of rank $(r_1,\dotsc,r_s)$ on $E$) over $X$} is an isomorphism class of a flag of Grassmann-Pl\"ucker functions (of rank $(r_1,\dotsc,r_s)$ on $E$) in $X$.
 \end{comment}
 Two flags of Grassmann-Pl\"ucker functions %$(\varphi_1,\dots,\varphi_s)$ and $(\varphi'_1,\dots,\varphi'_s)$ of the same type 
 \[
  \big(\varphi_i:E^{r_i}\to\Gamma(X,\cL_i)\big)_{i=1,\dotsc,s} \quad \text{and} \quad \big(\varphi'_i:E^{r_i}\to\Gamma(X,\cL'_i)\big)_{i=1,\dotsc,s}
 \]
 are \textit{equivalent} if there exists a collection of isomorphisms $\iota_i:\cL_i\to \cL_i'$ of line bundles (for $i=1,\dotsc,s$) such that $\varphi'_i=\iota_{i,X}\circ\varphi_i$ for $i=1,\dotsc,s$ where $\iota_{i,X}:\Gamma(X,\cL_i)\to\Gamma(X,\cL_i')$ is $\iota_i$ evaluated on global sections. A \emph{flag matroid bundle (of rank $(r_1,\dotsc,r_s)$ on $E$) over $X$} is the equivalence class of a flag of Grassmann-Pl\"ucker functions (of rank $(r_1,\dotsc,r_s)$ on $E$) in $X$.
\end{df}

\begin{rem}\label{rem: flag matroids as flag matroid bundles}
 Flag matroid bundles extend the notion of flag matroids in the following sense. Let $F$ be an idyll. Then there is a canonical bijection
 \[
  \big\{\text{flag $F$-matroids}\big\} \ \longrightarrow \ \big\{\text{flag matroid bundles over $\Spec F$}\big\},
 \]
 which is given as follows. Note that we have a canonical bijection $\iota:F\to \Gamma(X,\cO_X)$ for $X=\Spec F$. As explained in \cite[Prop.\ 5.3]{Baker-Lorscheid21}, we can associate with a Grassmann-Pl\"ucker function $\varphi:E^r\to F$ the Grassmann-Pl\"ucker function $\iota\circ\varphi:E^r\to\Gamma(X,\cO_X)$, which yields a (well-defined) bijection $M\mapsto\widetilde M$ between $F$-matroids and matroid bundles over $\Spec F$.
 
 Unraveling definitions, we see that a sequence $(\varphi_1,\dotsc,\varphi_s)$ of Grassmann-Pl\"ucker functions $\varphi_i:E^{r_i}\to F$ satisfies the Pl\"ucker flag relations if and only if the associated sequence $(\iota\circ\varphi_1,\dotsc,\iota\circ\varphi_s)$ satisfies them (considered as Grassmann-Pl\"ucker functions in $X$). Taking classes yields the desired bijection $\bM=(M_1,\dotsc,M_s)\mapsto(\widetilde M_1,\dotsc,\widetilde M_s)=\widetilde\bM$.
\end{rem}

\begin{rem}
 Note that a Grassmann-Pl\"ucker function of rank $r$ on $E$ in \cite{Baker-Lorscheid21} is defined as a function on the collection $\binom Er$ of $r$-subsets of $E$. This is equivalent with our definition since functions on $\binom Er$ identify with alternating functions on $E^r$ by choosing a total order on $E$ (which we do implicitly by identifying $E$ with $\{1,\dotsc,n\}$).
 
 While the Pl\"ucker coordinates of Grassmannians are usually indexed by elements of $\binom Er$, the representation of matroids as (alternating) Grassmann-Pl\"ucker functions on $E^r$ is more natural with relation to several constructions in matroid theory, which is our reason to choose alternating functions over functions on $r$-subsets.
\end{rem}

%%%%%%%%%%%%%%%%%%%%%%%%%%%%%%%%%%%%%%%%%%%%%%%%%%%%%%%%%%%%%%%%%%%%%%%%%%%%%%%%%%%%%%%%%%%%%%%%%%%%%%%%%%%%%%%%%%%%%%%%%%%%%%%%%%%%%%%%%%%%%%%%%%%%%%%%%%%%%%%%
\subsection{The moduli problem}
\label{subsection: the moduli problem}

Given an ordered blue $\Funpm$-scheme $X$, we define $\cFl(\br;E)(X)$ as the set of all flag matroid bundles of rank $\br$ on $E$ over $X$. A morphism $\alpha:X\to Y$ defines the map 
\[
 \alpha^{\ast}:\ \cFl(\br;E)(Y) \ \longrightarrow \ \cFl(\br;E)(X)
\]
by pulling back functions via $\alpha_Y^\sharp:\Gamma(Y,\cL)\to\Gamma(X,\alpha^\ast\cL)$: given a flag matroid bundle $\bM$ on $Y$ that is represented by a flag $\big(\varphi_i:E^{r_i}\to\Gamma(Y,\cL_i)\big)_{i=1,\dotsc,s}$, then we define $\alpha^\ast(\bM)$ as the flag matroid bundle on $X$ that is represented by $\big(\alpha_Y^\sharp\circ\varphi_i:E^{r_i}\to\Gamma(X,\alpha^\ast\cL_i)\big)_{i=1,\dotsc,s}$. We omit the verification that the latter family is indeed a flag of Grassmann-Pl\"ucker functions. We pose the following moduli problem:
\begin{quote}
 \it Is there an ordered blue $\Funpm$-scheme that represents $\cFl(\br;E)$?
\end{quote}
We work in the following sections towards its answer, which is the content of \autoref{thm: moduli property}.

%%%%%%%%%%%%%%%%%%%%%%%%%%%%%%%%%%%%%%%%%%%%%%%%%%%%%%%%%%%%%%%%%%%%%%%%%%%%%%%%%%%%%%%%%%%%%%%%%%%%%%%%%%%%%%%%%%%%%%%%%%%%%%%%%%%%%%%%%%%%%%%%%%%%%%%%%%%%%%%%
\subsection{Flag varieties}
\label{subsection: flag varieties}

The answer to our moduli problem is given by flag varieties over $\Funpm$, which are closed subschemes of a product of Grassmannians that parametrize the matroid bundles of the flag. The locus of the flag variety is given by the Pl\"ucker flag relations, which are multi-homogeneous with respect to the Pl\"ucker coordinates of each Grassmannian in the product. In order to bypass a rigorous treatment of multi-homogeneous calculus in $\Fun$-geometry, we give an explicit description of flag varieties in terms of open affine coverings.

We define the flag variety \emph{$\Fl(\br;E)$} as a closed subscheme of the product space
\[
 \P_\Funpm^{n^{r_1}-1}\times \dotsb\times\P_\Funpm^{n^{r_s}-1} \ = \ \prod_{i=1}^s \ \Proj\, \Funpm[T_{\bI}\mid \bI\in E^{r_i}],
\]
which is covered by products $\prod_{i=1}^s U_{\bJ_i}$ of the canonical open subsets 
\[
 U_{\bJ_i} \ = \ \Spec\, \Funpm[T_\bI/T_{\bJ_i}\mid\bI\in E^{r_i}]
\]
of $\P_\Funpm^{n^{r_i}-1}$ where $\bJ_i\in E^{r_i}$ for $i=1,\dotsc,s$. 

\begin{df}
 The \emph{flag variety $\Fl(\br;E)$ (of type $\br$ on $E$) over $\Funpm$} is the closed subscheme of $\prod \P^{n^{r_i}-1}_\Funpm$ that is covered by the open subschemes  
 \[\textstyle
  \prod \ U_{\bJ_i} \ = \ \Spec \Big(\bpquot{\big(\bigotimes\Funpm[T_{i,\bI}/T_{i,\bJ_i}\mid\bI\in E^{r_i}]\big)}{\fr}\Big)
 \]
 where $\fr$ is generated by the following relations for $1\leq i\leq j\leq s$:
 \[\tag{Fl1}\label{Fl1}
  \frac{T_{i,(x_{\sigma(1)},\dotsc,x_{\sigma(r_i)})}}{T_{i,\bJ_i}} \ = \ \sign(\sigma) \ \cdot \ \frac{T_{i,(x_1,\dotsc,x_{r_i})}}{T_{i,\bJ_i}}
 \]
 for all $(x_1,\dotsc,x_{r_i})\in E^{r_i}$ and all permutations $\sigma\in \mathfrak{S}_{r_i}$;
 \[\tag{Fl2}\label{Fl2}
 % T_{i,\bI}/T_{i,\bJ_i} \ = \ 0 \qquad  \qquad \tfrac{T_{i,\bI}}{T_{i,\bJ_i}} \ = \ 0
  \frac{T_{i,\bI}}{T_{i,\bJ_i}} \ = \ 0
 \]
 for all $\bI\in E^{r_i}$ with $\#|\bI|<r_i$;
 \[\tag{Fl3}\label{Fl3}
  0 \ \ \leq \ \  \sum_{k=1}^{r_j+1} \quad \epsilon^k \, \cdot \, \frac{T_{i,(y_k,\, x_1,\dotsc,\, x_{r_i-1})}}{T_{i,\bJ_i}} \, \cdot \, \frac{T_{j,(y_1,\dotsc,\, \widehat{y_k},\dotsc ,\, y_{r_j+1})}}{T_{j,\bJ_j}}
 \]
 for all $x_1,\dotsc,x_{r_i-1},y_1,\dotsc,y_{r_j+1}\in E$. The \emph{Pl\"ucker embedding} is the closed immersion
 \[
  \pl: \ \Fl(\br,E) \ \longrightarrow \ \prod \P^{n^{r_i}-1}_\Funpm
 \]
 as a subscheme.
\end{df}

Note that the intersection $V$ of affine opens $\prod U_{\bJ_i}$ and $\prod U_{\bJ'_i}$ is affine. Multiplying the defining relations of $\Gamma(\prod U_{\bJ_i})$ with the invertible section $T_{i,\bJ_i}/T_{i,\bJ'_i}$ on $V$ yields the defining relations of $\Gamma (\prod U_{\bJ'_i})$. This shows that $\Fl(\br,E)$ is well defined as a closed subscheme of $\prod \P^{n^{r_i}-1}_\Funpm$.

\begin{rem}
 Note that in the case of $s=1$ and $r_1=r$, the flag variety $\Fl(\br,E)$ is nothing else than a Grassmannian $\Gr(r,E)$ over $\Funpm$, which Baker and the second author called the \emph{matroid space $\Mat(r,E)$} in \cite{Baker-Lorscheid21} to distinct it from other models of Grassmannians in $\Fun$-geometry (e.g.\ see \cite{LopezPena-Lorscheid12}). We will drop this distinction in our text to appeal better to the reader's intuition from algebraic geometry.
 
 For arbitrary $s\geq1$, the Pl\"ucker embedding factors into closed immersions
 \[\textstyle
  \Fl(\br,E) \ \longrightarrow \ \prod_{i=1}^s \Gr(r_i,E) \ \longrightarrow \ \prod_{i=1}^s \P^{r_i-1}_\Funpm,
 \]
 i.e., the flag variety is a closed subscheme of a product of Grassmannians over $\Funpm$.
\end{rem}

%%%%%%%%%%%%%%%%%%%%%%%%%%%%%%%%%%%%%%%%%%%%%%%%%%%%%%%%%%%%%%%%%%%%%%%%%%%%%%%%%%%%%%%%%%%%%%%%%%%%%%%%%%%%%%%%%%%%%%%%%%%%%%%%%%%%%%%%%%%%%%%%%%%%%%%%%%%%%%%%
\subsection{The universal flag matroid bundle}
\label{subsection: the universal flag matroid bundle}

The Pl\"ucker embedding $\pl:\Fl(\br,E) \to \prod\P^{n^{r_i}-1}_\Funpm$ endows the flag variety over $\Funpm$ with a flag matroid bundle, which is universal for all matroid bundles (see \autoref{thm: moduli property}) and which is defined as follows.

Let $\pl_j=\pi_j\circ\pl\,$ be the composition of the Pl\"ucker embedding $\pl\,$ with the $j$-th coordinate projection $\pi_j:\prod \P^{n^{r_i}-1}_\Funpm\to\P^{n^{r_j}-1}_\Funpm$. Let $\cO(1)$ be the first twisted sheaf on $\P^{n^{r_i}-1}_\Funpm$, which is generated by the sections $T_{\bI}$ with $\bI\in E^{r_i}$. 

\begin{df}
 The \emph{universal flag of Grassmann-Pl\"ucker functions (of rank $\br$ on $E$)} is the sequence of Grassmann-Pl\"ucker functions $\big(\varphi_i^\univ:E^{r_i}\to\Gamma(\Fl(\br,E),\cL_i^\univ)\big)$ with line bundles $\cL^\univ_i=\pl_i^\ast\big(\cO(1)\big)$ and with $\varphi_i^\univ(\bI)=\pl_i^\sharp(T_{\bI})$. The \emph{universal flag matroid bundle $\bM^\univ$ (of rank $\br$ on $E$)} is the equivalence class of the universal flag of Grassmann-Pl\"ucker functions.
\end{df}

Note that the universal flag of Grassmann-Pl\"ucker functions is indeed a flag of Grassmann-Pl\"ucker functions of rank $\br$ on $E$ over $\Fl(\br,E)$ since $\cL_i^\univ$ is generated by $\{\pl_i^\sharp(T_{\bI})\mid\bI\in E^{r_i}\}$, and since the relations \ref{GP2*} and \ref{GP3*} are satisfied by the defining relations \eqref{Fl1}--\eqref{Fl3} of $\Fl(\br,E)$. Consequently, $\bM^\univ$ is a flag matroid bundle on $\Fl(\br,E)$.

%%%%%%%%%%%%%%%%%%%%%%%%%%%%%%%%%%%%%%%%%%%%%%%%%%%%%%%%%%%%%%%%%%%%%%%%%%%%%%%%%%%%%%%%%%%%%%%%%%%%%%%%%%%%%%%%%%%%%%%%%%%%%%%%%%%%%%%%%%%%%%%%%%%%%%%%%%%%%%%%
\subsection{The moduli property}
\label{subsection: the moduli property}

We are prepared to formulate the central result of this section.

\begin{thm}\label{thm: moduli property}
 The flag variety $\Fl(\br,E)$ is the fine moduli space of flag matroid bundles. More explicitly, the maps
 \[
  \begin{array}{cccc}
   \Phi_X: & \Hom(X,\;\Fl(\br,E)) &  \longrightarrow & \cFl(\br,E)(X) \\
           & \alpha:X\to\Fl(\br,E) & \longmapsto   & \alpha^\ast(\bM^\univ),
  \end{array}
 \]
 indexed by ordered blue $\Funpm$-schemes $X$, are functorial bijections.
\end{thm}

\begin{proof}
 Let us fix an $\Funpm$-scheme $X$. We begin with the construction of the inverse bijection $\Psi_X$ to $\Phi_X$ that maps a flag matroid bundle $\bM$ of rank $\br$ on $E$ over $X$ to a morphism $X\to\Fl(\br,E)$.
 
 Let $\big(\varphi_i:E^{r_i}\to\Gamma(X,\cL_i)\big)$ be a flag of Grassmann-Pl\"ucker functions in $X$ that represents the flag matroid bundle $\bM$. Since the global sections in the image of $\varphi_i$ generate the line bundle $\cL_i$, we can apply \cite[Thm.\ 4.20]{Baker-Lorscheid21}, which asserts the existence of a unique morphism $\psi_i:X\to\P^{n^{r_i}-1}$ and a unique isomorphism $\iota_i:\iota^\ast(\cO(1))\to\cL_i$ such that $\iota_i\big(\psi_i^\sharp(T_\bI)\big)=\varphi_i(\bI)$ where we fix an identification of the homogeneous coordinates of $\P^{n^{r_i}-1}$ with $\{T_\bI\mid \bI\in E^{r_i}\}$.
 
 Taking the product over all $i=1,\dotsc,s$ yields a morphism $\psi: X\to \prod\P^{n^{r_i}-1}$. Since the $\varphi_i$ are Grassmann-Pl\"ucker functions, \cite[Thm.\ 5.5]{Baker-Lorscheid21} applies and shows that the image of $\psi$ is contained in $\prod\Gr(r_i,E)$. Since the functions $\varphi_i$ satisfy moreover the Pl\"ucker flag relations, the image is, in fact, contained in $\Fl(\br,E)$, i.e., $\psi$ factors into a uniquely determined morphism $\alpha:X\to\Fl(\br,E)$ followed by the closed immersion $\Fl(\br,E)\to\prod\P^{n^{r_i}-1}$. We define $\Psi_X(\bM)=\alpha$.
 
 Next we show that $\Phi_X\circ\Psi_X=\id$. Since $\cL_i^\univ=\pi_i^\ast(\cO(1))$ for the coordinate projection $\pi_i:\Fl(\br,E)\to\P^{n^{r_i}-1}$, the isomorphism $\iota_i:\iota^\ast(\cO(1))\to\cL_i$ becomes an isomorphism $\widetilde\iota_i:\alpha^\ast(\cL_i^\univ)\to\cL_i$, and we have $\varphi_i(\bI)=\widetilde\iota_i\big(\alpha^\sharp(T_\bI)\big)$. Thus $\alpha^\ast(\bM^\univ)=\bM$, which shows that $\Psi_X$ is right inverse to $\Phi_X$.
 
 Next we show that $\Psi_X\circ\Phi_X=\id$. Let $\alpha:X\to\Fl(\br,E)$ be a morphism and $\bM=\alpha^\ast(\bM^\univ)$ the associated flag matroid bundle over $X$. Then $\bM$ is represented by the flag $(\varphi_1,\dotsc,\varphi_s)$ of Grassmann-Pl\"ucker functions 
 \[
  \varphi_i:\ E^{r_i} \ \stackrel{\varphi_i^\univ}\longrightarrow \ \Gamma\big(\Fl(\br,E),\cL_i^\univ\big) \ \stackrel{\alpha^\sharp}\longrightarrow \ \Gamma(X,\cL_i).
 \]
 The composition of $\alpha$ with the coordinate projection $\pi_i:\Fl(\br,E)\to\P^{n^{r_i}-1}$ yields the morphism $\alpha_i:X\to\P^{n^{r_i}-1}$, which satisfies $\alpha_i^\sharp(T_\bI)=\varphi_i(\bI)$. By the construction of $\Psi_X$, we have thus $\Psi_X(\bM)=\alpha$, which verifies that $\Psi_X$ is a left inverse to $\Phi_X$.
 
 We are left to show the functoriality of $\Phi_X$. Consider a morphism $\beta:Y\to X$ and $\alpha:X\to\Fl(\br,E)$. Then 
 \[
  \Phi_Y\big(\beta^\ast(\alpha)\big) \ = \ \Phi_Y(\alpha\circ\beta) \ = \ (\alpha\circ\beta)^\ast(\bM^\univ) \ = \ \beta^\ast\big(\alpha^\ast(\bM^\univ)\big) \ = \ \beta^\ast(\Phi_X(\alpha)\big),
 \]
 which shows that $\Phi_X$ is functorial in $X$.
\end{proof}

\begin{cor}\label{cor: flag matroids as rational points}
 Let $F$ be an idyll. Then there is a canonical bijection between the set $\Fl(\br,E)(F)$ of $F$-rational points of the flag variety $\Fl(\br,E)$ over $\Funpm$  and the set of flag $F$-matroids of rank $\br$ on $E$.
\end{cor}

\begin{proof}
 This follows at once from \autoref{thm: moduli property} coupled with \autoref{rem: flag matroids as flag matroid bundles}.
\end{proof}

\begin{df}
 Given an idyll $F$ and a flag $F$-matroid $\bM$ of rank $\br$ on $E$, let $\widetilde\bM$ be the corresponding flag matroid bundle on $\Spec F$. We call the inverse image of $\widetilde\bM$ under $\Phi_X$ the \emph{characteristic morphism of $\bM$} and denote it by $\chi_\bM:\Spec F\to\Fl(\br,E)$. 
\end{df}

In other words, the characteristic morphism $\chi_\bM:\Spec F\to\Fl(\br,E)$ is the unique morphism with $\widetilde\bM=\chi_\bM^\ast(\bM^\univ)$.

%%%%%%%%%%%%%%%%%%%%%%%%%%%%%%%%%%%%%%%%%%%%%%%%%%%%%%%%%%%%%%%%%%%%%%%%%%%%%%%%%%%%%%%%%%%%%%%%%%%%%%%%%%%%%%%%%%%%%%%%%%%%%%%%%%%%%%%%%%%%%%%%%%%%%%%%%%%%%%%%
\subsection{Projection onto subflags}
\label{subsection: projection onto subflags}

Let $\bi=(i_1,\dotsc,i_t)$ be a sequence of integers with $1\leq i_1<\dotsb<i_{t}\leq s$. Let $\br'=(r_{i_1},\dotsc,r_{i_{t}})$ and
\[
 \widetilde\pi_\bi: \ \prod_{i=1}^s \Gr(r_i,E) \ \longrightarrow \ \prod_{k=1}^{t} \Gr(r_{i_k},E)
\]
be the morphism that is induced by the coordinate projections.

\begin{prop}\label{prop: coordinate projections}
 The morphism $\widetilde\pi_\bi$ restricts to a morphism $\pi_\bi:\Fl(\br,E)\to\Fl(\br',E)$. Let $F$ be an idyll and $\bM=(M_1,\dotsc,M_s)$ a flag $F$-matroid of rank $\br$ on $E$ with characteristic morphism $\chi_\bM:\Spec F\to\Fl(\br,E)$. Then the characteristic morphism of the flag $F$-matroid $\bM'=(M_{i_1},\dotsc,M_{i_{t}})$ is
 \[
  \chi_{\bM'}=\pi_\bi\circ\chi_\bM: \ \Spec F \ \stackrel{\chi_\bM\;} \longrightarrow \ \Fl(\br,E) \ \stackrel{\pi_{\mathbf{i}}} \longrightarrow \ \Fl(\br',E).
 \]
\end{prop}

\begin{proof}
 This follows at once from the construction of $\Psi_X$ in the proof of \autoref{thm: moduli property} and the definition of $\chi_\bM$ as $\Psi_X\big(\widetilde\bM\big)$ where $\widetilde\bM$ is the matroid bundle over $X=\Spec F$ that corresponds to $\bM$. 
\end{proof}

%%%%%%%%%%%%%%%%%%%%%%%%%%%%%%%%%%%%%%%%%%%%%%%%%%%%%%%%%%%%%%%%%%%%%%%%%%%%%%%%%%%%%%%%%%%%%%%%%%%%%%%%%%%%%%%%%%%%%%%%%%%%%%%%%%%%%%%%%%%%%%%%%%%%%%%%%%%%%%%%
\subsection{Duality}
\label{subsection: duality morphism}

The duality of flag $F$-matroids extends to a functorial dualization of flag matroid bundles, which is reflected by a canonical isomorphism of flag varieties.

Let $\Gr(r,E)=\Fl((r),E)$ be the Grassmannian of rank $r$ on $E=\{1,\dotsc,n\}$. Let $r^\ast=n-r$. By \cite[Thm.\ 5.6]{Baker-Lorscheid21}, there is a canonical isomorphism
\[
 \delta_{r,E}: \ \Gr(r,E) \ \longrightarrow \ \Gr(r^\ast,E)
\]
of $\Funpm$-schemes that is characterized by the pullback formula\footnote{Note that the definition in \cite{Baker-Lorscheid21} is off by a sign, which is corrected in our formula.}
\[
 \delta_{r,E}^\sharp(T_{x_1,\dotsc,x_{r^\ast}}) \ = \ \sign(x_1,\dotsc,x_{r^\ast},x_1',\dotsc,x_r') \cdot T_{x'_1,\dotsc,x'_r}
\]
for all $E=\{x_1,\dotsc,x_{r^\ast},x_1',\dotsc,x_r'\}$, which determines the images of all non-trivial homogeneous coordinates of $\Gr(r^\ast,E)$.

Let $F$ be a tract with involution $\tau$ and $M$ an $F$-matroid with characteristic morphism $\chi_M:\Spec F\to\Gr(r,E)$. Let $\tau^\ast$ be the involution of $\Spec F$ that corresponds to $\tau$. The characteristic morphism of the dual $F$-matroid $M^\ast$ (with respect to $\tau$) is
\[
 \chi_{M^\ast}: \ \Spec F \ \stackrel{\tau^\ast}\longrightarrow \ \Spec F  \ \stackrel{\chi_M}\longrightarrow \ \Gr(r,E)  \ \stackrel{\delta_{r,E}}\longrightarrow \ \Gr(r^\ast,E)
\]
by \cite[Prop.\ 5.8]{Baker-Lorscheid21}. 

We extend this result to flag matroids. Let $\br^\ast=(n-r_s,\dotsc,n-r_1)$. Note that the isomorphisms $\delta_{r_i,E}:\Gr(r_i,E)\to\Gr(r_i^\ast,E)$ determine an isomorphism
\[\textstyle
 \widehat\delta_{\br,E}: \ \prod_{i=1}^s \Gr(r_i,E) \ \longrightarrow \ \prod_{i=1}^s \Gr(r_i^\ast,E).
\]

\begin{thm}\label{thm: dual flag variety}
 The isomorphism $\widehat\delta_{\br,E}$ restricts to an isomorphism
 \[
  \delta_{\br,E}: \ \Fl(\br,E) \ \longrightarrow \ \Fl(\br^\ast,E).
 \]
 Let $F$ be an idyll with involution $\tau$ and $\bM$ a flag $F$-matroid with characteristic morphism $\chi_\bM:\Spec F\to\Fl(\br,E)$. Let $\tau^\ast$ be the involution of $\Spec F$ that corresponds to $\tau$. The characteristic morphism of the dual flag $F$-matroid $\bM^\ast$ (with respect to $\tau$) is
\[
 \chi_{\bM^\ast}: \ \Spec F \ \stackrel{\tau^\ast}\longrightarrow \ \Spec F  \ \stackrel{\chi_\bM\;}\longrightarrow \ \Fl(\br,E)  \ \stackrel{\delta_{\br,E}}\longrightarrow \ \Fl(\br^\ast,E).
\]
\end{thm}

\begin{proof}
 In order to show that $\widehat\delta_{\br,E}$ restricts to an isomorphism $\delta_{\br,E}: \Fl(\br,E) \to \Fl(\br^\ast,E)$, it suffices to verify that the Pl\"ucker flag relations are preserved. Consider the relation
 \[
  0 \ \ \leq \ \  \sum_{k=1}^{r_j+1} \quad \epsilon^k \, \cdot \, T_{i,(y_k,\, x_1,\dotsc,\, x_{r_i-1})} \, \cdot \, T_{j,(y_1,\dotsc,\, \widehat{y_k},\dotsc ,\, y_{r_j+1})}
 \]
 in $\Gamma\big(\Fl(\br,E),\cL_i^\univ\otimes\cL_j^\univ\big)$ for $1\leq i<j\leq s$ and $x_1,\dotsc x_{r_i-1},y_1,\dotsc,y_{r_j+1}\in E$. Since $T_\bI=0$ in $\Gamma\big(\Fl(\br,E),\cL_i^\univ\big)$ if the entries of $\bI$ are not pairwise distinct, we can assume that the $x_1,\dotsc,x_{r_i-1}$ are pairwise distinct and the same for $y_1,\dotsc,y_{r_j+1}$. In other words, 
 \begin{comment}
 \[
  E \ = \ \big\{x_1,\dotsc,x_{r_i-1},x_1',\dotsc,x'_{r^\ast_{i^\ast}+1}\big\} \ = \ \big\{y_1,\dotsc,y_{r_j+1},y_1',\dotsc,y'_{r^\ast_{j^\ast}-1}\big\}
 \]
 for some $x'_k$ and $y'_l$ where $i^\ast=n+1-i$ and $j^\ast=n+1-j$. Applying $\delta_{\br,E}^\sharp$ to both sides of the Pl\"ucker flag relation in question yields
 \end{comment}
  \[
  E \ = \ \big\{x_1,\dotsc,x_{r_i-1},x_1',\dotsc,x'_{r^\ast_{i}+1}\big\} \ = \ \big\{y_1,\dotsc,y_{r_j+1},y_1',\dotsc,y'_{r^\ast_{j}-1}\big\}
 \]
 for some $x'_k$ and $y'_l$. Applying $\delta_{\br,E}^\sharp$ to both sides of the Pl\"ucker flag relation in question yields
 \[
  0 \ \ \leq \ \  \sum_{k=1}^{r_{i}^\ast+1} \quad \epsilon^k \, \cdot \, T_{j,(x'_k,\, y'_1,\dotsc,\, y'_{r^\ast_{j}-1})} \, \cdot \, T_{i,(x'_1,\dotsc,\, \widehat{x'_k},\dotsc ,\, x'_{r^\ast_{i}+1})},
 \]
 up to a common sign change, which depends on an ordering of $E$; see \cite[section 5.5]{Baker-Lorscheid21} for details on the definition of $\delta_{\br,E}$. This is precisely the Pl\"ucker flag relation of $\Fl(\br^\ast,E)$ for the given indices. This shows that the image of $\delta_{\br,E}$ is contained in $\Fl(\br^\ast,E)$. By the symmetry of the argument in $\Fl(\br,E)$ and $\Fl(\br^\ast,E)$, we conclude that $\delta_{\br,E}$ is an isomorphism.

 \begin{comment}
 The characteristic morphism $\chi_\bM:\Spec F\to\Fl(\br,E)$ determines a flag of Grassmann-Pl\"ucker functions $(\varphi_i:E^{r_i}\to F)$ that represents the flag $F$-matroid $\bM$ in terms of $\varphi_i(\bI)=\chi_\bM^\sharp(T_\bI)$ where we identify $\Gamma\big(\Spec F,\chi_\bM^\ast(\cL_i^\univ)\big)=F$.  Let $(\varphi^\ast_{i^\ast}:E^{r^\ast_{i^\ast}}\to F)$ be a flag of Grassmann-Pl\"ucker functions that represents $\bM^\ast$, which is uniquely determined up to a common sign change that depends on the ordering of $E$; see \cite[section 5.5]{Baker-Lorscheid21}. Thus we have for $\bI=(x_1,\dotsc,x_{r_i})$ and $\bI^\ast=(x'_1,\dotsc,x'_{r^\ast_{i^\ast}})$ with $|\bI|\cup|\bI^\ast|=E$ that
  \[
  \sigma_{\bI^\ast}\varphi^*_{i^\ast}(\bI^\ast) \ = \ \tau\big(\varphi_i(\bI)\big) \ = \ \tau\big(\chi_M^\sharp(T_\bI)) \ = \ \tau\big(\chi_{M}^\sharp\big(\delta_{\br,E}^\sharp(\sigma_{\bI^\ast} T_{\bI^\ast})\big)\big)
 \]
 \end{comment}
 The characteristic morphism $\chi_\bM:\Spec F\to\Fl(\br,E)$ determines a flag of Grassmann-Pl\"ucker functions $(\varphi_i:E^{r_i}\to F)$ that represents the flag $F$-matroid $\bM$ in terms of $\varphi_i(\bI)=\chi_\bM^\sharp(T_\bI)$ where we identify $\Gamma\big(\Spec F,\chi_\bM^\ast(\cL_i^\univ)\big)=F$.  Let $(\varphi^\ast_{i}:E^{r^\ast_{i}}\to F)$ be a flag of Grassmann-Pl\"ucker functions that represents $\bM^\ast$, which is uniquely determined up to a common sign change that depends on the ordering of $E$; see \cite[section 5.5]{Baker-Lorscheid21}. Thus we have for $\bI=(x_1,\dotsc,x_{r_i})$ and $\bI^\ast=(x'_1,\dotsc,x'_{r^\ast_{i}})$ with $|\bI|\cup|\bI^\ast|=E$ that
 \[
  \sigma_{\bI^\ast}\varphi^*_{i}(\bI^\ast) \ = \ \tau\big(\varphi_i(\bI)\big) \ = \ \tau\big(\chi_M^\sharp(T_\bI)) \ = \ \tau\big(\chi_{M}^\sharp\big(\delta_{\br,E}^\sharp(\sigma_{\bI^\ast} T_{\bI^\ast})\big)\big)
 \]
 where $\sigma_{\bI^\ast}\in\{1,\epsilon\}$ is a sign that depends on the induced ordering of $|\bI^\ast|\subset E$. After taking classes this yields $\chi_{\bM^\ast}^\ast(\bM^\univ)=\bM^\ast=(\delta_{\br,E}\circ\chi_\bM\circ\tau^\ast)^\ast(\bM^\univ)$. By \autoref{cor: flag matroids as rational points}, the characteristic morphism is uniquely determined by the pullback of the universal flag matroid bundle and thus $\chi_{\bM^\ast}=\delta_{\br,E}\circ\chi_\bM\circ\tau^\ast$, which proves the latter claim of the theorem.
\end{proof}

%%%%%%%%%%%%%%%%%%%%%%%%%%%%%%%%%%%%%%%%%%%%%%%%%%%%%%%%%%%%%%%%%%%%%%%%%%%%%%%%%%%%%%%%%%%%%%%%%%%%%%%%%%%%%%%%%%%%%%%%%%%%%%%%%%%%%%%%%%%%%%%%%%%%%%%%%%%%%%%%%\subsection{}
\subsection{Minors}
\label{subsection: minors of the moduli space}

Let $E'=E-\{e\}$ for some $e\in E$ and $0\leq r\leq n$. Let $V_{r,/e}$ and $V_{r,\backslash e}$ be the closed subschemes of $\Gr(r,E)$ that are defined by the relations
\[
 T_\bI \ = \ 0 \quad \text{if $e\in|\bI|$} \qquad \text{and} \qquad T_\bI \ = \ 0 \quad \text{if $e\notin|\bI|$},
\]
respectively. Let $U_{r,/e}$ and $U_{r,\backslash e}$ be the open subschemes of $\Gr(r,E)$ whose underlying sets are the respective complements of $V_{r,/e}$ and $V_{r,\backslash e}$. As explained in \cite[section 5.6]{Baker-Lorscheid21}, we have natural morphisms
\[
 \Gr(r,E) \ \stackrel{\iota_{r,/e}}\longleftarrow \ U_{r,/e} \amalg V_{r,/e} \ \stackrel{\Psi_{r,/e}}\longrightarrow \ \Gr(r-1,E') \amalg \Gr(r,E')
\]
where $\Psi_{r,/e}$ is given by sending a homogeneous coordinate $T_\bI$ to $T_{\bI e}$ if $e\notin|\bI|$ (which yields $U_{r,/e}\to\Gr(r-1,E')$) and to $T_{\bI}$ if $e\in|\bI|$ (which yields $V_{r,/e}\to\Gr(r,E')$). Similarly, we have natural morphisms
\[
 \Gr(r,E) \ \stackrel{\iota_{r,\backslash e}}\longleftarrow \ U_{r,\backslash e} \amalg V_{r,\backslash e} \ \stackrel{\Psi_{r,\backslash e}}\longrightarrow \ \Gr(r,E') \amalg \Gr(r-1,E').
\]
For $k=0,\dotsc,s$, let $W_{\br,k,/e}$ be the intersection, or fiber product over $\prod\Gr(r_i,E)$, of $\Fl(\br,E)$ with
\[
 U_{r_1,/e} \ \times \ \dotsb \ \times \ U_{r_k,/e} \ \times \ V_{r_{k+1},/e}  \ \times \ \dotsb \ \times \ V_{r_s,/e}
\]
and let $W_{\br,k,\backslash e}$ be the intersection of $\Fl(\br,E)$ with
\[
 V_{r_1,\backslash e} \ \times \ \dotsb \ \times \ V_{r_k,\backslash e} \ \times \ U_{r_{k+1},\backslash e}  \ \times \ \dotsb \ \times \ U_{r_s,\backslash e},
\]
which are locally closed subschemes of $\Fl(\br,E)$. Let
\[\textstyle
 \iota_{\br,/e}: \coprod_{k=0}^s W_{\br,k,/e} \ \longrightarrow \ \Fl(\br,E) \quad \text{and} \quad \iota_{\br,\backslash e}: \coprod_{k=0}^s W_{\br,k,\backslash e} \ \longrightarrow \ \Fl(\br,E) 
\]
be the respective disjoint unions of the embeddings as subschemes. For $0\leq k\leq s$, let
\[
 \br_k \ = \ (r_1-1,\dotsc,r_k-1,\; r_{k+1},\dotsc,r_s).
\]
The morphisms $\Psi_{r_i,/e}$ and $\Psi_{r_i,\backslash e}$ induce the respective morphisms
\[
 \Psi_{\br,k,/e}: \ W_{\br,k,/e} \ \longrightarrow \ \Fl(\br_k,E) \qquad \text{and} \qquad \Psi_{\br,k,\backslash e}: \ W_{\br,k,\backslash e} \ \longrightarrow \ \Fl(\br_k,E).
\]

\begin{thm}\label{thm: minors of flag varieties}
 Let $F$ be an idyll. Then the induced maps
 \[\textstyle
   \coprod_{k=0}^s W_{\br,k,/e}(F) \ \longrightarrow \ \Fl(\br,E)(F) \qquad \text{and} \qquad \coprod_{k=0}^s W_{\br,k,\backslash e}(F) \ \longrightarrow \ \Fl(\br,E)(F) 
 \]
 of $F$-rational point sets are bijections. Let $\bM$ be a flag $F$-matroid of rank $\br$ on $E$ with characteristic morphism $\chi_\bM\in\Fl(\br,E)(F)$. Let $k_{/e}$ and $k_{\backslash e}$ be the unique indices for which $\chi_\bM\in W_{\br,k_{/e},/e}(F)$ and $\chi_\bM\in W_{\br,k_{\backslash e},\backslash e}(F)$. Then the characteristic functions of $\bM/e$ and $\bM\backslash e$ are
 \[
  \chi_{\bM/e}: \ \Spec F \ \stackrel{\chi_{\bM}}\longrightarrow \ W_{\br,k_{/e},/e} \ \stackrel{\Psi_{\br,k_{/e},/e}}\longrightarrow \ \Fl(\br_{k_{/e}},E')
 \]
 and
 \[
  \chi_{\bM\backslash e}: \ \Spec F \ \stackrel{\chi_{\bM}}\longrightarrow \ W_{\br,k_{\backslash e},\backslash e} \ \stackrel{\Psi_{\br,k_{\backslash e},\backslash e}}\longrightarrow \ \Fl(\br_{k_{\backslash e}},E'),
 \]
 respectively.
\end{thm}

\begin{proof}
 In this proof, we will only derive the claims about contracting $e$ and omit a proof of the claims about excluding $e$, which is analogous. 
 
 Let $\bM=(M_1,\dotsc,M_s)$ a flag $F$-matroid with characteristic morphism $\chi_\bM:\Spec F\to\Fl(\br,E)$. Then $\chi_\bM$ factors through $W_{\br,k,/e}$ if and only if $e$ is not a loop for $M_1,\dotsc,M_k$, but it is a loop for $M_{k+1},\dotsc,M_s$. Since these conditions are mutually exclusive for distinct $k$ and since the factorization of $\chi_\bM$ through $W_{\br,k,/e}$ is unique by the nature of locally closed subschemes, we conclude that the map $\coprod W_{\br,k,/e}(F)\to\Fl(\br,E)(F)$ is injective.
 
 It is surjective for the following reason. If $e$ is a loop for $M_i$ and $i<j$, then $e$ is also a loop for $M_j$; see \cite[Lemma 1]{Recski05}. Thus there is a $k$ such that $e$ is not a loop for $M_1,\dotsc,M_k$, but it is a loop for $M_{k+1},\dotsc,M_s$, which means that $\chi_\bM$ factors through $W_{\br,k,/e}$ and shows that $\coprod W_{\br,k,/e}(F)\to\Fl(\br,E)(F)$ is surjective.

 The latter claim of the theorem can be deduced as follows. The characteristic morphism $\chi_\bM$ of $\bM$ is determined by its compositions with the coordinate projections $\pi_i:\Fl(\br,E)\to\Gr(r_i,E)$ for $i=1,\dotsc,s$. By \autoref{prop: coordinate projections}, $\pi_i\circ\chi_\bM:\Spec F\to\Gr(r_i,E)$ is the characteristic morphism of the $F$-matroid $M_i$. Note that $\pi_i(W_{\br,k,/e})$ equals $U_{r_i,/e}$ for $i\leq k$ and $V_{r_i,/e}$ for $i>k$. A comparison with \cite[Thm.\ 5.9]{Baker-Lorscheid21} yields that $\pi_i\circ\Psi_{\br,k_{/e},/e}\circ\chi_\bM$ is the characteristic morphism of the $F$-matroid $M_i/e$. By the definition of $\bM/e$ as $(M_1/e,\dotsc,M_s/e)$, we see that $\Psi_{\br,k_{/e},/e}\circ\chi_\bM$ is the characteristic morphism of $\bM/e$, as claimed.
\end{proof}

\begin{rem}
 The compatibility of minors of $F$-flag matroids with duality extends to geometry in the sense that the duality $\delta_{\br,E}:\Fl(\br,E)\to\Fl(\br^\ast,E)$ restricts to an isomorphism $\delta\vert_{W_{\br,k,/e}}:W_{\br,k,/e}\to  W_{\br^\ast,k,\backslash e}$ and that the diagram
 \[
  \beginpgfgraphicnamed{tikz/fig1}
   \begin{tikzcd}[column sep=2.5cm]
    \Fl(\br,E) \ar[d,"\delta_{\br,E}"',"\sim"] & W_{\br,k,/e} \ar[l,"\iota_{\br,/e}"'] \ar[r,"\Psi_{\br,k,/e}"] \ar[d,"\delta\vert_{W_{\br,k,/e}}","\sim"'] & \Fl(\br_k,E')  \ar[d,"\delta_{\br_k,E'}","\sim"'] \\
    \Fl(\br^\ast,E) & W_{\br^\ast,k,\backslash e} \ar[l,"\iota_{\br,\backslash e}^\ast"'] \ar[r,"\Psi_{\br,k,\backslash e}^\ast"] & \Fl(\br_k^\ast,E')
   \end{tikzcd}
  \endpgfgraphicnamed
 \]
 is commutative for all $k=0,\dotsc,s$. This can be deduced from the corresponding fact for (usual) matroid bundles (see \cite[Thm.\ 5.9]{Baker-Lorscheid21}), but for the sake of a compact presentation we omit the details.
\end{rem}

%%%%%%%%%%%%%%%%%%%%%%%%%%%%%%%%%%%%%%%%%%%%%%%%%%%%%%%%%%%%%%%%%%%%%%%%%%%%%%%%%%%%%%%%%%%%%%%%%%%%%%%%%%%%%%%%%%%%%%%%%%%%%%%%%%%%%%%%%%%%%%%%%%%%%%%%%%%%%%%%%\subsection{}
\subsection{Flags of minors}
\label{subsection: flags of minors morphism}

Let $n'=n+r_s-r_1$ and $E'=\{1,\dotsc,n'\}$. For $i=1,\dotsc,s$, let $J_i=\{n+1,\dotsc,n+r_i-r_1\}$ and $I_i=\{n+r_i-r_1+1,\dotsc,n'\}$. Recall from \autoref{thm: flags of minors} that an $F$-matroid $M'$ of rank $r_s$ on $E'$ gives rise to the flag $F$-matroid $(M'\minor{I_1}{J_1},\dotsc,M'\minor{I_s}{J_s})$ of rank $\br$ on $E$. 

This association extends to a rational map $\mu_{\br,E}:\Gr(r_s,E')\DashedArrow[->]\Fl(\br,E)$, which we describe in terms of the images of the multi-homogeneous coordinates of $\Fl(\br,E)$. Namely, for $e_1,\dotsc,e_{r_i}\in E$, the coordinate $T_{e_1,\dotsc,e_{r_i}}$ of the $i$-th factor in $\prod_{i=1}^s\P_\Funpm^{n^{r_i}-1}$ is mapped to $T_{e_1,\dotsc,e_{r_i},n+1,\dotsc,n+r_s-r_i}$.

The domain of the rational map $\mu_{\br,E}:\Gr(r_s,E')\DashedArrow[->]\Fl(\br,E)$ is as follows. For $I=\{e_1,\dotsc,e_{r_s}\}\subset E'$, we define $V_I$ as the intersection of $\Gr(r',E')$ with the canonical open $U_{x_1,\dotsc,x_{r_s}}$ of $\P^{(n')^{r_s}-1}$. Note that $V_I$ does not depend on the order of $x_1,\dotsc,x_{r_s}$ and that $V_I=\emptyset$ if $\# I<r_s$ due to the defining equations of $\Gr(r_s,E')$. Let $\binom{E'}{r_s}$ be the collection of $r_s$-subsets of $E'$,
\[\textstyle
 \Omega \ = \ \big\{ I \in \binom{E'}{r_s} \, \big| \, n+1,\dotsc,n'\in I \big\} \quad \text{and} \quad \Omega^c \ = \ \big\{ I \in \binom{E'}{r_s} \, \big| \, n+1,\dotsc,n'\notin I \big\}
\]
The domain of $\Gr(r_s,E')\DashedArrow[->]\Fl(\br,E)$ is
\[
 W_{\br,E} \ = \ \bigcup_{I\in\Omega,\, J\in\Omega^c} V_I\cap V_J,
\]
which is the locus of all matroids $M'$ for which $\{n+1,\dotsc,n'\}$ is independent and co-independent. In other words, we can consider $\mu_{\br,E}$ as morphism $\mu_{\br,E}:W_{\br,E}\to\Fl(\br,E)$. We omit further details for the sake of a compact presentation and proceed with the central statement about this morphism (without proof).

\begin{thm}\label{thm: flags of minors morphism}
 Let $F$ be a tract and $M'$ be an $F$-matroid of rank $r_s$ on $E'$ with characteristic morphism $\chi_{M'}:\Spec F\to\Gr(r_s,E')$ whose image we assume to be in $W_{\br,E}$. Let $\bM=(M_1,\dotsc,M_s)$ be the flag $F$-matroid of rank $\br$ on $E$ with $M_i=M'\minor{I_i}{J_i}$ and $\chi_\bM:\Spec F\to \Fl(\br, E)$ its characteristic morphism. Then $\chi_\bM=\mu_{\br,E}\circ\chi_{M'}$.
\end{thm}

%%%%%%%%%%%%%%%%%%%%%%%%%%%%%%%%%%%%%%%%%%%%%%%%%%%%%%%%%%%%%%%%%%%%%%%%%%%%%%%%%%%%%%%%%%%%%%%%%%%%%%%%%%%%%%%%%%%%%%%%%%%%%%%%%%%%%%%%%%%%%%%%%%%%%%%%%%%%%%%%%\subsection{}
\subsection{Rational point sets}
\label{subsection: rational point sets}

A \emph{topological idyll} is an idyll $F$ together with a topology such that the multiplication $F\times F\to F$ and the inversion $F^\times\to F^\times$ are continuous maps and such that $\{0\}$ is closed. Examples of topological idylls are topological fields (considered as idylls with the same topology), the tropical hyperfield with the real topology, the Krasner hyperfield $\K$ with the topology generated by the open subset $\{1\}$ and the sign hyperfield $\S$ with the topology generated by the open subsets $\{1\}$ and $\{-1\}$.

The sets of $F$-rational points $X(F)$ for an $\Funpm$-scheme $X$ and a topological idyll $F$ come with the \emph{fine topology on $X(F)$}, which is characterized by the following properties:
\begin{enumerate}
 \item for affine $X=\Spec B$, the fine topology $X(F)=\Hom(B,F)$ carries the compact-open topology with respect to the discrete topology for $B$;
 \item the canonical bijection $F\to \A_F^1(F)$ is a homeomorphism;
 \item the canonical bijection $(X\times Y)(F)\to X(F)\times Y(F)$ is a homeomorphism;
 \item for every morphism $Y \to X$, the canonical map $Y(F)\to X(F)$ is continuous;
 \item for every open / closed immersion $Y\to X$, the canonical inclusion $Y(F)\to X(F)$ is an open / closed topological embedding;
 \item for every covering of $X$ by ordered blue open subschemes $U_i$, a subset $W$ of $X(F)$ is open if and only if $W\cap U_i(F)$ is open in $U_i(F)$ for every $i$;
 \item for every continuous morphism $F\to F'$, the induced map $X(F)\to X(F')$ is continuous.
\end{enumerate}
For more details, see \cite[section 6]{Lorscheid15} and \cite[section 5.10]{Baker-Lorscheid21}.

In the following, we briefly discuss the topological spaces that arise from the topological idylls that we mention above.

\subsubsection{Topological fields}
Let $K$ be a topological field. Then $\Fl(\br,E)(K)$ is canonically homeomorphic to the set of $K$-rational points of the usual flag variety of type $\br$ flags of subspaces of $K^E$, equipped with the strong topology stemming from $K$.

As a particular example, we consider $\R$ in its incarnation as an idyll. Then $\Fl(\br,E)(\R)$ is canonically homeomorphic to the real flag manifold of type $\br$ flags of linear subspaces of $\R^E$.

\subsubsection{The tropical hyperfield}\label{subsubsection: tropical points of the flag variety}
The space $\Fl(\br,E)(\T)$ of tropical points is canonically homeomorphic to the flag Dressian $FlDr(r_1,\dotsc,r_s)$ of Brandt, Eur and Zhang's paper \cite{Brandt-Eur-Zhang21}. They are naturally tropical subvarieties of a product $\prod\P^{n^{r_i}-1}(\T)$ of projective tropical spaces (in their incarnations as tropical varieties). We regain several insights of \cite{Brandt-Eur-Zhang21} from our point of view.

By \autoref{cor: flag matroids as rational points}, a point of $\Fl(\br,E)(\T)$ corresponds to a flag $\T$-matroid. Since the covectors of a $\T$-matroid, or valuated matroid, form a tropical linear space, a flag $\T$-matroid corresponds to a flag of tropical linear spaces in $\T^E$ by \autoref{thm-cryptomorphism}. In conclusion, $\Fl(\br,E)(\T)$ corresponds bijectively to rank $\br$-flags of tropical linear spaces in $\T^E$.

The fine topology of $\Fl(\br,E)(\T)$ corresponds to the variation of flags of tropical linear space in terms of open neighbourhoods in $\T^E$ of a given flag or, equivalently, to a variation of the coefficients of the defining equations of the tropical linear spaces within an open interval.

We alert the reader that tropical Grassmannians are defined as the tropicalizations of classical Grassmannians in tropical geometry, which are in general proper subvarieties of Dressians. This terminological artifact leads to a slight inconsistency in the notation $\Gr(r,E)$ for the (underlying scheme of the) Dressian. Still, the generalization of tropical Grassmannians to flag varieties can also be recovered from our viewpoint: let $K$ be a field with a non-archimedean absolute value $v:K\to\R_{\geq0}$, which can be seen as an idyll morphism $v:K\to\T$, see \cite[Thm.\ 2.2]{Lorscheid19}. The tropical flag variety of type $\br$ on $E$ stemming from $K$ is the image of the map $v_\ast:\Fl(\br,E)(K)\to\Fl(\br,E)(\T)$.

\subsubsection{The Krasner hyperfield}

The natural topology for $\K$ is the poset topology with $0<1$, which turns $\K$ into the terminal object in the category of topological tracts. In consequence, the finite set $\Fl(\br,E)(\K)$ inherits the topology of a poset, which equals the subspace topology of $\Fl(\br,E)(\K)$ considered as a subset of the underlying topological space of $\Fl(\br,E)$ together with the Zariski topology. In particular, the closed points of $\Fl(\br,E)(\K)$ correspond to all flag matroids $(M_1,\dotsc,M_s)$ of rank $\br$ on $E$ for which the matroids $M_1,\dotsc,M_s$ have each exactly one basis.

The poset $\Fl(\br,E)(\K)$ has a unique maximal element $\hat 1$, or generic point, which is the rank $\br$-flag of uniform matroids on $E$. Thus $\{\hat 1\}$ is an open subset of $\Fl(\br,E)(\K)$ and, in consequence, $\Fl(\br,E)(\K)$ can be contracted to the point $\hat 1$, which generalizes the analogous result of Anderson and Davis for the ``hyperfield Grassmannian'' $\Gr(e,E)(\K)$; see \cite[section 6]{Anderson-Davis19}.

\subsubsection{The sign hyperfield}

Similar to the case of the Krasner hyperfield, $\Fl(\br,E)(\S)$ is a finite poset with the order topology. In this case, the topology is highly non-trivial and links to the (disproven) MacPhersonian conjecture in the case of the Grassmannian $\Gr(r,E)(\S)$, see \cite{Mnev-Ziegler93,Liu17,Anderson-Davis19}. An interesting question is whether the results from \cite{Anderson-Davis19} generalize to flag varieties. In particular, we pose the following question.
\begin{problem*}
 Does the continuous map $\sign_\ast:\Fl(\br,E)(\R)\to\Fl(\br,E)(\S)$ induce a surjection in mod $2$ cohomology?
\end{problem*}

\subsubsection{The triangular hyperfield}

Yet another interesting example of a topological idyll is Viro's \emph{triangle hyperfield} $\V$ (see \cite{Viro11,Anderson-Davis19}) whose underlying monoid is $\R_{\geq0}$, whose ambient semiring is $\N[\R_{>0}]$ and whose partial order is generated by the relations $0\leq a+b+c$ whenever $|a-b|\leq c\leq a+b$, endowed with the real topology. Alternatively $\V$ can be described as the hyperfield quotient of $\C$ by the unit circle $\S^1=\{z\in \C\mid |z|=1\}$. The continuous morphisms $\R\to\C\to\V$ induce continuous maps
\[
 \Fl(\br,E)(\R) \ \longrightarrow \ \Fl(\br,E)(\C) \ \longrightarrow \ \Fl(\br,E)(\V).
\]

\subsubsection{The regular partial field}

The natural choice of topology for the regular partial field is the discrete topology, which retains the position of the regular partial field as an initial object in the category of topological tracts. This topology turns $\Fl(\br,E)(\Funpm)$ into a discrete point set. The unique continuous map $\Funpm\to F$ into any other topological tract $F$ induces a continuous map $\Fl(\br,E)(\Funpm)\to\Fl(\br,E)(F)$.

%%%%%%%%%%%%%%%%%%%%%%%%%%%%%%%%%%%%%%%%%%%%%%%%%%%%%%%%%%%%%%%%%%%%%%%%%%%%%%%%%%%%%%%%%%%%%%%%%%%%%%%%%%%%%%%%%%%%%%%%%%%%%%%%%%%%%%%%%%%%%%%%%%%%%%%%%%%%%%%%%%%%%%%%%%%
%%%%%%%%%%%%%%%%%%%%%%%%%%%%%%%%%%%%%%%%%%%%%%%%%%%%%%%%%%%%%%%%%%%%%%%%%%%%%%%%%%%%%%%%%%%%%%%%%%%%%%%%%%%%%%%%%%%%%%%%%%%%%%%%%%%%%%%%%%%%%%%%%%%%%%%%%%%%%%%%%%%%%%%%%%%

\bibliographystyle{alpha}
\bibliography{flag}

\end{document}